\documentclass[11pt]{amsart}
\usepackage[latin9]{inputenc}
\usepackage{enumitem}
\usepackage{amsthm}
\usepackage{amstext}
\usepackage{amssymb}
\usepackage{esint}

\makeatletter
\numberwithin{equation}{section}
\numberwithin{figure}{section}
\usepackage{enumitem}		
\theoremstyle{plain}
\newtheorem{thm}{\protect\theoremname}[section]
  \theoremstyle{plain}
  \newtheorem{assumption}[thm]{\protect\assumptionname}
  \theoremstyle{plain}
  \newtheorem{lem}[thm]{\protect\lemmaname}
  \theoremstyle{definition}
  \newtheorem{defn}[thm]{\protect\definitionname}
  \theoremstyle{remark}
  \newtheorem{rem}[thm]{\protect\remarkname}
  \theoremstyle{remark}
  \newtheorem*{rem*}{\protect\remarkname}
 \newlist{casenv}{enumerate}{4}
 \setlist[casenv]{leftmargin=*,align=left,widest={iiii}}
 \setlist[casenv,1]{label={{\itshape\ \casename} \arabic*.},ref=\arabic*}
 \setlist[casenv,2]{label={{\itshape\ \casename} \roman*.},ref=\roman*}
 \setlist[casenv,3]{label={{\itshape\ \casename\ \alph*.}},ref=\alph*}
 \setlist[casenv,4]{label={{\itshape\ \casename} \arabic*.},ref=\arabic*}
  \theoremstyle{plain}
  \newtheorem{cor}[thm]{\protect\corollaryname}

 \usepackage{amsthm}\usepackage{amsfonts}\usepackage{latexsym}

\let\emptyset\varnothing

  \providecommand{\assumptionname}{Assumption}
  \providecommand{\corollaryname}{Corollary}
  \providecommand{\definitionname}{Definition}
  \providecommand{\lemmaname}{Lemma}
  \providecommand{\remarkname}{Remark}
 \providecommand{\casename}{Case}
\providecommand{\theoremname}{Theorem}

\makeatother

  \providecommand{\assumptionname}{Assumption}
  \providecommand{\corollaryname}{Corollary}
  \providecommand{\definitionname}{Definition}
  \providecommand{\lemmaname}{Lemma}
  \providecommand{\remarkname}{Remark}
 \providecommand{\casename}{Case}
\providecommand{\theoremname}{Theorem}

\begin{document}

\title{Resolution of the wavefront set using general continuous wavelet
transforms}

\author{Jonathan Fell, Hartmut Führ, Felix Voigtlaender}

\address{J. Fell\\
 RWTH Aachen University \\
 Lehrstuhl C für Mathematik (Analysis) \\
 Pontdriesch 10 \\
 D-52062 Aachen \\
 Germany}

\address{H.~F{ü}hr, F. Voigtlaender\\
Lehrstuhl A f{ü}r Mathematik \\
RWTH Aachen University \\
 D-52056 Aachen\\
 Germany}
\begin{abstract}
We consider the problem of characterizing the wavefront set of a tempered
distribution $u\in\mathcal{S}'(\mathbb{R}^{d})$ in terms of its continuous
wavelet transform, where the latter is defined with respect to a suitably
chosen dilation group $H\subset{\rm GL}(\mathbb{R}^{d})$. In this
paper we develop a comprehensive and unified approach that allows
to establish characterizations of the wavefront set in terms of rapid
coefficient decay, for a large variety of dilation groups.

For this purpose, we introduce two technical conditions on the dual
action of the group $H$, called microlocal admissibilty and (weak)
cone approximation property. Essentially, microlocal admissibilty
sets up a systematical relationship between the scales in a wavelet
dilated by $h\in H$ on one side, and the matrix norm of $h$ on the
other side. The (weak) cone approximation property describes the ability
of the wavelet system to adapt its frequency-side localization to
arbitrary frequency cones. Together, microlocal admissibility and
the weak cone approximation property allow the characterization of
points in the wavefront set using multiple wavelets. Replacing the
weak cone approximation by its stronger counterpart gives access to
single wavelet characterizations.

We illustrate the scope of our results by discussing -- in any dimension
$d\ge2$ -- the similitude, diagonal and shearlet dilation groups,
for which we verify the pertinent conditions. As a result, similitude
and diagonal groups can be employed for multiple wavelet characterizations,
whereas for the shearlet groups a single wavelet suffices. In particular,
the shearlet characterization (previously only established for $d=2$)
holds in arbitrary dimensions. 
\end{abstract}
\maketitle
\global\long\def\with{\,\middle|\,}

\noindent \textbf{\small Keywords:}{\small {} wavefront set; square-integrable
group representation; continuous wavelet transform; anisotropic wavelet
systems; shearlets}{\small \par}

\noindent \textbf{\small AMS Subject Classification:}{\small {} 42C15;
42C40; 46F12}{\small \par}

\section{Introduction}

\subsection{Regular directed points and the wavefront set}

The wavefront set was introduced by Hörmander in \cite{Hoe_71} as
a means of analyzing mapping properties of Fourier integral operators.
This set can be understood as a particular model for singularities
in an otherwise regular object (e.g., edges in images), see the discussion
in \cite{CaDo,KuLa}. The ability to resolve the wavefront set (i.e.,
to characterize this set via coefficient decay) has become somewhat
of a benchmark problem for generalized wavelet systems and related
constructions in dimensions two and higher.

Before we give precise definitions, let us introduce some notation.
Given $R>0$ and $x\in\mathbb{R}^{d}$, we let $B_{R}(x)$ and $\overline{B_{R}}\left(x\right)$
denote the open/closed ball with radius $R$ and center $x$, respectively.
We let $S^{d-1}\subset\mathbb{R}^{d}$ denote the unit sphere. By
a neighborhood of $\xi\in S^{d-1}$, we will always mean a \emph{relatively
open} set $W\subset S^{d-1}$ with $\xi\in W$. Given $R>0$ and an
open set $W\subset S^{d-1}$, we let 
\[
C(W):=\left\{ r\xi'\with\xi'\in W,\, r>0\right\} =\left\{ \xi\in\mathbb{R}^{d}\setminus\left\{ 0\right\} \with\frac{\xi}{\left|\xi\right|}\in W\right\} \quad\text{ and }\quad C(W,R):=C(W)\setminus\overline{B_{R}}(0).
\]
Both sets are clearly open subsets of $\mathbb{R}^{d}\setminus\left\{ 0\right\} $
and thus of $\mathbb{R}^{d}$.

Given a tempered distribution $u$, we call $(x,\xi)\in\mathbb{R}^{d}\times S^{d-1}$
a \textbf{regular directed point of $u$} if there exists $\varphi\in C_{c}^{\infty}(\mathbb{R}^{d})$,
identically one in a neighborhood of $x$, as well as a $\xi$-neighborhood
$W\subset S^{d-1}$ such that for all $N\in\mathbb{N}$ there exists
a constant $C_{N}>0$ with 
\begin{equation}
\forall\xi'\in C\left(W\right)~:~\left|\widehat{\varphi u}(\xi')\right|\le C_{N}(1+|\xi'|)^{-N}.\label{eqn:decay_cond_reg_dir}
\end{equation}
A simple observation, which will nonetheless be important for the
following, is that this decay condition effectively only concerns
the behaviour at large frequencies: Since $\varphi u$ is a compactly
supported distribution, its Fourier transform is a continuous (even
smooth) function \cite[Theorem 7.23]{RudinFA}, and thus we may replace
$C\left(W\right)$ in equation \eqref{eqn:decay_cond_reg_dir} by
$C(W,R)$ for any $R>0$.

Informally speaking, regular directed points describe oriented local
regularity behaviour of a tempered distribution: If $(x,\xi)$ is
a regular directed point of $u$, then $u$ can be considered smooth
at $x$, when viewed in direction $\xi$. We define the \textbf{wavefront
set} of $u$ as the set of points $(x,\xi)$ which are not regular
directed points of $u$. The results in our paper will all be stated
as criteria for regular directed points.

\subsection{Continuous wavelet transforms in higher dimensions}

From the outset, the continuous wavelet transform in dimension one
has been understood as the ideal tool to analyze local regularity
of functions, see e.g. \cite{Ho_book} for an extensive discussion.
In higher dimensions, there are increasingly many possible generalizations
of the continuous wavelet transform available, and it is currently
not well-understood (except for isolated examples such as the shearlet
group \cite{KuLa,Grohs_2011} and the similitude group \cite{Pilipovic_et_al_2006})
how these different transforms fare at resolving the wavefront set.
It is the chief purpose of this paper to develop criteria that allow
to tackle this question in a unified and comprehensive approach.

Before we give a more detailed description of the aims of this paper,
let us introduce the necessary notions pertaining to continuous wavelet
transforms in some detail. We fix a closed matrix group $H<{\rm GL}(d,\mathbb{R})$,
the so-called \textbf{dilation group}, and let $G=\mathbb{R}^{d}\rtimes H$.
This is the group of affine mappings generated by $H$ and all translations.
Elements of $G$ are denoted by pairs $(x,h)\in\mathbb{R}^{d}\times H$,
and the product of two group elements is given by $(x,h)(y,g)=(x+hy,hg)$.
The left Haar measure of $G$ is given by ${\rm d}(x,h)=|\det(h)|^{-1}{\rm d}x\,{\rm d}h$.

$G$ acts unitarily on ${\rm L}^{2}(\mathbb{R}^{d})$ by the \textbf{quasi-regular
representation} defined by 
\begin{equation}
[\pi(x,h)f](y)=|\det(h)|^{-1/2}\cdot f\left(h^{-1}(y-x)\right)~.\label{eqn:def_quasireg}
\end{equation}
In particular, $\pi$ induces an action of $G$ on $\mathcal{S}(\mathbb{R}^{d})$,
the space of Schwartz functions.

We write $\left\langle \cdot\mid\cdot\right\rangle :\mathcal{S}'(\mathbb{R}^{d})\times\mathcal{S}\left(\mathbb{R}^{d}\right)\to\mathbb{C}$
for the natural extension of the ${\rm L}^{2}$-scalar product, which
means $\langle u\mid\psi\rangle:=u\left(\overline{\psi}\right)=\left\langle u,\overline{\psi}\right\rangle $.
For future reference, let us observe that a straightforward calculation
yields 
\begin{equation}
\left[\mathcal{F}\left(\pi\left(x,h\right)f\right)\right]\left(\xi\right)=\left|\det\left(h\right)\right|^{1/2}\cdot e^{-2\pi i\left\langle x,\xi\right\rangle }\cdot\widehat{f}\left(h^{T}\xi\right)\label{eq:QuasiRegularOnFourierSide}
\end{equation}
for $f\in{\rm L}^{1}\left(\mathbb{R}^{d}\right)+{\rm L}^{2}\left(\mathbb{R}^{d}\right)$.
Here, as in the remainder of the paper, we use the convention 
\[
\mathcal{F}f\left(\xi\right)=\widehat{f}\left(\xi\right)=\int_{\mathbb{R}^{d}}f\left(x\right)\cdot e^{-2\pi i\left\langle x,\xi\right\rangle }\,{\rm d}x
\]
for the Fourier transform of $f\in L^{1}\left(\mathbb{R}^{d}\right)$.

Now, given a tempered distribution $u\in\mathcal{S}'\left(\mathbb{R}^{d}\right)$
and some $\psi\in\mathcal{S}(\mathbb{R}^{d})$, we define the \textbf{wavelet
transform} of $u$ with respect to $\psi$ by 
\[
W_{\psi}u:G\to\mathbb{C}~,~(x,h)\mapsto\langle u\mid\pi(x,h)\psi\rangle.
\]

\subsection{Characterizing regular directed points by wavelet transform decay}

The basic principle of local regularity analysis via wavelets in dimension
one is that \emph{smoothness of a function $f$ near $x$ is equivalent
to decay of wavelet coefficients $W_{\psi}f(y,s)$ for $y$ near $x$
and small scales $s$.} It is the aim of this paper to extend this
principle to directional smoothness, by providing criteria for regular
directed points. The result should be an oriented version of the above-mentioned
principle: \emph{Smoothness of a function $f$ near $x$ in direction
$\xi$ is equivalent to decay of wavelet coefficients $W_{\psi}f(y,h)$
for $y$ near $x$, for small-scale wavelets $\pi(y,h)\psi$ oscillating
in directions near $\xi$.}

More specifically, given a continuous wavelet transform $W_{\psi}$
based on a suitable choice of dilation group $H$ and wavelet $\psi$,
we wish to establish results of the following form: 
\begin{eqnarray}
\quad(x,\xi)\mbox{ is a regular directed point of \ensuremath{u}} & \Leftrightarrow & \exists\mbox{ neighborhood }U\mbox{ of }x\,\forall y\in U\label{eqn:single_wvlt_char}\\
 &  & \forall h\in K\,\forall N\in\mathbb{N}~:|W_{\psi}u(y,h)|\le C_{N}\|h\|^{N}~,\nonumber 
\end{eqnarray}
where $K\subset H$ is a suitable subset of dilations which explicitly
depends on $\psi$ and a certain (sufficiently small) cone containing
directions near $\xi$. Intuitively, $K$ contains those $h\in H$
such that $\pi(0,h)\psi$ is a small-scale wavelet oscillating in
direction $\xi$.

A less ambitious approach allows \emph{multiple wavelets} for the
characterization of regular directed points, i.e., we ask whether,
for a suitable family $(\psi_{\lambda})_{\lambda\in\Lambda}$ of wavelets,
the following characterization is available: 
\begin{eqnarray}
\quad(x,\xi)\mbox{ is a regular directed point of \ensuremath{u}} & \Leftrightarrow & \exists\lambda\in\Lambda\,\exists\mbox{ neighborhood }U\mbox{ of }x\,\forall y\in U\label{eqn:multiple_wvlt_char}\\
 &  & \forall h\in K\,\forall N\in\mathbb{N}~:|W_{\psi_{\lambda}}u(y,h)|\le C_{N}\|h\|^{N}~,\nonumber 
\end{eqnarray}
again with $K\subset H$ depending explicitly on $\psi_{\lambda}$
and a suitable choice of frequency cone.

The literature contains examples related to both types of characterizations:
\cite[Theorem 5.1]{KuLa} is a single wavelet characterization of
the wavefront set using the shearlet transform, for directions $\xi=(\xi_{1},\xi_{2})^{T}\in\mathbb{R}^{2}$
belonging to the horizontal cone characterized by $\frac{|\xi_{1}|}{|\xi_{2}|}<1$.
The remaining directions are then characterized by a second shearlet
transform with coordinate axes interchanged. By contrast, \cite[Theorem 7]{Pilipovic_et_al_2006}
can be understood as a somewhat weaker form of the multiple wavelet
characterization (\ref{eqn:multiple_wvlt_char}), by describing regular
directed points in terms of the decay of $W_{\psi_{\lambda}}(\mathcal{\phi}u)$,
with a given family of wavelets $\psi_{\lambda}$ and additionally,
a freely choosable cutoff function $\phi$. Here, the underlying dilation
group is the similitude group (in certain dimensions), consisting
of rotations combined with scalar dilations.

\subsection{Proof strategy: Understanding the role of the dual action}

In this paper, we will provide a general framework that allows to
understand and extend the two mentioned examples, and to formulate
sufficient conditions for the characterizations (\ref{eqn:single_wvlt_char})
and (\ref{eqn:multiple_wvlt_char}). For this purpose, we will make
systematic use of the \textbf{dual action} and its properties. Mathematically
speaking, the dual action is just the (right) action $\mathbb{R}^{d}\times H\ni(\xi,h)\mapsto h^{T}\xi\in\mathbb{R}^{d}$.
This action plays a decisive role for the study of many questions
in connection with general continuous wavelet transforms, e.g. for
representation-theoretic questions such as irreducibility and direct
integral decompositions \cite{BeTa,Fu_LN}, in connection with general
wavelet inversion formulae \cite{FuMa,LaWeWeWi,Fu10}, or for wavelet
coorbit theory \cite{Fu_coorbit,Fu_atom,FuVo}. The results in this
paper naturally fit into this wider context.

When considering wavefront set characterizations with general dilation
groups, several obstacles arise: Most dilation groups (except for
the similitude group) do not have a built-in orientation and scale
parameter, thus it may be difficult or even impossible to have a meaningful
notion of small scale wavelets oscillating in direction $\xi$, and
thus to properly define the set $K\subset H$ in (\ref{eqn:single_wvlt_char})
and (\ref{eqn:multiple_wvlt_char}). The answer to this problem is
provided by the dual action: We restrict our attention to bandlimited
wavelets, say ${\rm supp}(\widehat{\psi})\subset V$, with a suitable
compact set $V$. Then equation (\ref{eq:QuasiRegularOnFourierSide})
reveals that the Fourier support of $\pi(y,h)\psi$ is contained in
$h^{-T}V$. This simple observation allows to assign direction and
scale to wavelets indexed by group elements $h\in H$. We start with
the direction part: Given a frequency cone $C(W,R)\subset\mathbb{R}^{d}$,
we introduce certain cone-affiliated subsets $K_{i},K_{o}\subset H$,
with $h\in K_{i}$ whenever $h^{-T}V\subset C(W,R)$ and $h\in K_{o}$
whenever $h^{-T}V\cap C(W,R)\not=\emptyset$. Thus $h\in K_{i}$ or
$K_{o}$ does allow to predict oscillatory behaviour of $\pi(y,h)\psi$
in direction $W$.

Furthermore, since our targeted wavelet characterizations measure
rapid decay in terms of the matrix norm $\|h\|$, we want to be able
to interpret this norm as a scale parameter for $\pi(y,h)\psi$. The
condition of \emph{microlocal admissibility} of the dual action is
tailored to establish a systematic relationship between the matrix
norm $\|h\|$ and the frequencies in the support of $(\pi(y,h)\psi)^{\wedge}$,
and thus ultimately permits a meaningful interpretation of the matrix
norm $\|h\|$ as the scale of $\pi(y,h)\psi$. These notions combined
will then allow to understand $\pi(y,h)\psi$, for $h\in K_{i}$ or
$K_{o}$, as a wavelet near $y$ of scale proportionate to $\|h\|\preceq R^{-\alpha}$
(for some positive $\alpha$) and oscillating in the directions contained
in $W$.

Besides these interpretation issues for elements of $H$, there is
a second challenge related to the study of wavelet criteria for wavefront
sets, which particularly concerns the question whether a single wavelet
suffices to characterize the wavefront set. Recall that the definition
of regular directed points involves two types of localization: Localization
in location, as expressed by the possibility to choose arbitrary cutoff
functions $\varphi$, as well as localization with respect to directions,
as expressed by the choice of the frequency cone $C(W,R)$. It is
fairly easy to see that wavelet transforms can adapt to the first
kind, in particular when using a notion of scale that is related to
the matrix norm. The second type of localization however is more subtle:
In order to adapt to arbitrarily small cone apertures (corresponding
to the set $W$), the wavelet system must be able to make increasingly
fine distinctions between orientations, at least as the scales go
to zero (i.e., as $R\to\infty$). It was observed in \cite{CaDo}
that the classical tensor wavelet system associated to a multiresolution
analysis does not possess this feature: The angular resolution does
not change over scales, and hence the wavelet system is not able to
resolve the wavefront set.

Consequently, \cite{CaDo} introduced curvelets as an alternative
to wavelets, with improved angular resolution: The curvelet system
is indexed by a family of circularly equidistant angle and logarithmically
equidistant scale parameters, with the number of angles doubling at
every other scale%
\footnote{As pointed out in \cite{CaDo}, the decomposition of frequencies that
underlies curvelets, was introduced earlier in \cite{Stein_93}, where
it is called \emph{second dyadic decomposition}.%
}, and this feature does allow to characterize regular directed points
via the decay of curvelet coefficients. It should be noted that curvelets
do not fit into the scheme presented here, since they are not based
on the action of a dilation group; they are however somewhat closely
related to shearlets \cite{KuLa}, which do arise in the manner sketched
above from the action of the so-called shearlet-dilation group (as
noted later in \cite{DaKuStTe}). The central insight of \cite{KuLa}
was that shearlets show a similar frequency-side behaviour as curvelets,
and as a result, they also characterize the wavefront set, at least
for directions in the horizontal cone.

Other groups, such as the similitude group, generate wavelet systems
that have the same angular resolution across all scales, and consequently,
their ability to resolve the wavefront set is limited. However, by
switching the wavelets, if necessary, it is possible to attain arbitrary
angular resolution. Thus the similitude group lends itself to a multiple
wavelet characterization of regular directed points.

In the setting of general continuous wavelet transforms over arbitrary
dilation groups, we therefore need a mathematical description of phenomena
like increasing frequency resolution for large scales. In view of
the fact that the curvelet system was described primarily by the frequency
localization of the different curvelets, it is not surprising that
once again the dual action proves useful; in particular, the sets
$K_{i}$ and $K_{o}$ defined above naturally enter in this context.
It turns out that there are precise and workable conditions available,
formulated in terms of inclusion properties of the cone-affiliated
sets $K_{o}$ and $K_{i}$ from above (related to varying parameters
$W,R$ and $V$), which allow such an assessment. Here, the pertinent
notions are the \emph{cone approximation property} introduced in Section
\ref{sect:character}, and a somewhat weaker version. Together with
microlocal admissibility, the cone approximation property allows to
establish single wavelet characterizations, whereas the weak cone
approximation property provides multiple wavelet characterizations.

\subsection{Overview of the paper}

The paper is structured as follows: Section \ref{sect:dual_action}
contains a discussion of the various conditions we impose on the dual
action that allow to meaningfully assign scale and direction to a
dilated wavelet. We introduce the cone-affiliated subsets $K_{i}$
and $K_{o}$, and study their basic properties, as well as the notion
of microlocal admissibility of the action. Section \ref{sect:central}
contains the central technical result of this paper: Theorem \ref{thm:almost_char}
relates the property that $(x,\xi)$ is a regular directed point to
the decay of wavelet coefficients near $x$ and dilations in certain
cone-affiliated sets $K_{i},K_{o}$. This result is not quite a characterization,
as it concludes the decay for the set $K_{i}$, and requires it for
the larger set $K_{o}$. In order to close this gap, Section \ref{sect:character}
introduces the cone approximation properties, which then enable us
to formulate and prove (single- or multiple) wavelet characterizations
in Theorem \ref{thm:char_wfset}. For the single orbit case (equivalently:
whenever $\pi$ is irreducible), the characterization results are
particularly satisfactory, see Corollary \ref{cor:char_wfset_orbit_weak_cone_approx}.
An alternative -- more geometric -- description of the (weak) cone-approximation
property is then shortly discussed in Section \ref{sect:geometric}.

Finally, in Section \ref{sect:examples}, we demonstrate that microlocal
admissibility and (weak) cone approximation property are actually
verifiable for many concrete and interesting cases. Specifically,
in each dimension $d\ge2$ we consider the diagonal, similitude and
shearlet groups, and show that there are multiple wavelet characterizations
available for the first two groups, and that single wavelet characterizations
hold for the shearlet case. This considerably extends the previously
known results: The similitude group case was considered in \cite{Pilipovic_et_al_2006},
and our results require extra conditions on the wavelets on the one
hand, but do not require a local cutoff function. For the shearlet
groups, the result was only established for $d=2$ in \cite{KuLa,Grohs_2011}.
The diagonal case seems to be completely new.

\section{Conditions on the dual action}

\label{sect:dual_action}Throughout this paper we will write $V\Subset\mathcal{O}$
to indicate that $V,\mathcal{O}\subset\mathbb{R}^{d}$ are open sets
and that the closure $\overline{V}\subset\mathcal{O}$ is a compact
subset of $\mathcal{O}$. We will always assume that the dilation
group $H$ fulfils the following assumptions, which are mostly connected
to (partial) wavelet inversion.
\begin{assumption}
\label{assume:proper_dual} There exists an open, $H^{T}$-invariant
subset $\mathcal{O}\subset\mathbb{R}^{d}$ with the following properties: 
\begin{enumerate}[label=(\alph*)]
\item \label{enu:DualActionIsProper}The dual action of $H$ on $\mathcal{O}$
is proper, i.e., for all compact sets $K\subset\mathcal{O}$, the
set 
\[
H_{K}:=\left\{ (h,\xi)\in H\times\mathcal{O}\with(h^{T}\xi,\xi)\in K\times K\right\} 
\]
is compact. 
\item \label{enu:DualActionContainsRays}For each $\xi\in\mathcal{O}$,
we have $\mathbb{R}^{+}\xi\subset\mathcal{O}$, where $\mathbb{R}^{+}:=\left(0,\infty\right)$. 
\item \label{enu:ExistenceOfAdmissibleFunction}There exists a Schwartz
function $\psi$ such that $\widehat{\psi}$ is compactly supported
inside $\mathcal{O}$, and in addition, $\psi$ fulfils the \textbf{admissibility
condition} 
\begin{equation}
\forall\xi\in\mathcal{O}~:~\int_{H}|\widehat{\psi}(h^{T}\xi)|^{2}\,{\rm d}h=1.\label{eq:AdmissibilityCondition}
\end{equation}

\item \label{enu:PolynomialGrowthOfMeasureOfIntersection}Given $\emptyset\neq V\Subset\mathcal{O}$
and $\xi\in\mathcal{O}$, we define 
\[
H_{\xi,V}=\left\{ h\in H\with h^{T}\xi\in V\right\} =\left(h\mapsto h^{T}\xi\right)^{-1}\left(V\right),
\]
which is a relatively compact open set because of $H_{\xi,V}\subset\pi_{1}\left(H_{\left\{ \xi\right\} \cup\overline{V}}\right)$,
where $\pi_{1}$ is the projection on the first coordinate.

We assume that for each $\emptyset\neq V\Subset\mathcal{O}$, there
are constants $C,\alpha\geq0$ such that the estimate 
\[
\forall\xi\in\mathcal{O}~:~\mu_{H}(H_{\xi,V})\le C\cdot\left(1+\left|\xi\right|\right)^{\alpha}
\]
is fulfilled, where $\mu_{H}$ denotes the left Haar measure on $H$.

\end{enumerate}
\end{assumption}
These assumptions may seem somewhat arbitrary and complicated, but
they are fulfilled in many relevant cases. In particular, if $\pi$
is an \textbf{(irreducible) square-integrable representation}, then
the results in \cite{Fu10,Fu96} show that the dual action has a single
open orbit $\mathcal{O}=\left\{ h^{T}\xi_{0}\with h\in H\right\} \subset\mathbb{R}^{d}$
of full measure (for some $\xi_{0}\in\mathcal{O}$), such that in
addition the stabilizer group $H_{\xi_{0}}=\left\{ h\in H\with h^{T}\xi_{0}=\xi_{0}\right\} $
is compact. In this case, any nonzero $\psi$ with $\widehat{\psi}\in C_{c}^{\infty}(\mathcal{O})$
will be admissible, after suitable normalization because the integral
in equation \eqref{eq:AdmissibilityCondition} is invariant under
the change $\xi\mapsto g^{T}\xi$ for $g\in H$ by left-invariance
of the Haar measure. But $H^{T}\xi=\mathcal{O}$ for any $\xi\in\mathcal{O}$.
Hence, the integral in equation \eqref{eq:AdmissibilityCondition}
is constant on $\mathcal{O}$.

Properness of the action follows from compactness of the stabilizer,
because a Baire-category argument shows that this implies that the
projection $p_{\xi_{0}}:H\rightarrow\mathcal{O},h\mapsto h^{T}\xi_{0}$
is proper.

Part \ref{enu:DualActionContainsRays} of the assumption follows from
the fact that $H^{T}\left(r\xi\right)=r\cdot H^{T}\xi=r\mathcal{O}$
is an open orbit. But the dual action only has one open orbit, which
implies $r\xi\in r\mathcal{O}=\mathcal{O}$. Finally, part \ref{enu:PolynomialGrowthOfMeasureOfIntersection}
of the assumption is taken care of in the following lemma: 
\begin{lem}
Assume that $\mathcal{O}=H^{T}\xi_{0}$ is an open $H$-orbit, with
associated compact stabilizers. Then 
\[
\mu_{H}\left(H_{\xi,V}\right)=\mu_{H}\left(H_{\xi_{0},V}\right)<\infty
\]
holds for all $\xi\in\mathcal{O}$ and $\emptyset\neq V\Subset\mathcal{O}$.\end{lem}
\begin{proof}
Observe that if $\xi=g^{T}\xi_{0}$ with $g\in H$, then 
\[
H_{\xi,V}=\left\{ h\in H\with h^{T}g^{T}\xi_{0}\in V\right\} =\left\{ h\in H\with gh\in H_{\xi_{0},V}\right\} =g^{-1}H_{\xi_{0},V},
\]
showing that $H_{\xi,V}$ is a left translate of $H_{\xi_{0},V}$.
Since $\mu_{H}$ is left-invariant and $H_{\xi_{0},V}$ is precompact,
the claim follows.
\end{proof}
While the irreducible case provides the most satisfying results in
this paper, we have chosen to discuss the problem in a somewhat more
general setting, for the benefit of further investigations.

We next formally define the sets $K_{i}$ and $K_{o}$ which will
allow to associate group elements to directions. 
\begin{defn}
Let $\emptyset\neq W\subset S^{d-1}$ be open with $W\subset\mathcal{O}$
(which implies $C\left(W\right)\subset\mathcal{O}$). Furthermore,
let $\emptyset\neq V\Subset\mathcal{O}$ and $R>0$. We define 
\[
K_{i}(W,V,R):=\left\{ h\in H\with h^{-T}V\subset C(W,R)\right\} 
\]
as well as 
\[
K_{o}(W,V,R):=\left\{ h\in H\with h^{-T}V\cap C(W,R)\not=\emptyset\right\} .
\]
If the parameters are provided by the context, we will simply write
$K_{i}$ and $K_{o}$. Here, the subscripts $i/o$ stand for ``inner/outer''.
\end{defn}
These two types of sets are the central tool of our analysis. The
intuition behind their definition is that $K_{i}$ contains all dilations
$h$ with the property that the wavelets $\pi(y,h)\psi$ only ``see''
directions in the cone $C(W,R)$. Hence, local regularity in these
directions should entail a decay estimate for the wavelet coefficients
$\left(W_{\psi}u\right)\left(y,h\right)$ with $h\in K_{i}$. Here,
we used the property ${\rm supp}\left(\mathcal{F}\left[\pi\left(x,h\right)\psi\right]\right)\subset h^{-T}{\rm supp}\left(\smash{\widehat{\psi}}\right)\subset h^{-T}V$
which is immediate from equation \eqref{eq:QuasiRegularOnFourierSide}
as long as ${\rm supp}\left(\smash{\widehat{\psi}}\right)\subset V$
holds.

Conversely, $K_{o}$ contains all those dilations that contribute
to the (formal) \emph{wavelet reconstruction} 
\[
\widehat{\varphi u}\left(\xi\right)=\int_{\mathbb{R}^{d}}\int_{H}\left(W_{\psi}\varphi u\right)\left(y,h\right)\cdot\left(\mathcal{F}\left[\pi\left(y,h\right)\psi\right]\right)\left(\xi\right)\,\frac{{\rm d}h}{\left|\det\left(h\right)\right|}\,{\rm d}y
\]
of the frequency content of a (localized) tempered distribution $\varphi u$,
for $\xi\in C(W,R)$, again under the assumption ${\rm supp}\left(\smash{\widehat{\psi}}\right)\subset V$.
Thus, decay estimates for wavelet coefficients with dilations in $h\in K_{o}$
should allow to predict local regularity of $u$ in these directions.

The wavelet criteria that we will establish (cf. Theorem \ref{thm:almost_char})
and their proofs can be seen as a mathematically rigourous implementation
of these ideas. 
\begin{rem}
\label{rem:incl_prop}We have $K_{i}\subset K_{o}$. Furthermore,
$K_{o}\subset H$ is open and $K_{i}\subset H$ is a $G_{\delta}$
set. In particular, $K_{o}$ and $K_{i}$ are Borel-measurable. Also,
\[
K_{i}^{-T}V\subset C(W,R).
\]
In fact, $K_{i}$ is the largest set fulfilling this inclusion.

Another easy but useful observation is that $W\subset W'$, $V\subset V'$
and $R_{1}\ge R_{2}$ together imply 
\[
K_{o}(W,V,R_{1})\subset K_{o}(W',V',R_{2}),
\]
whereas $W\subset W'$, $V\supset V'$ and $R_{1}\ge R_{2}$ together
entail 
\[
K_{i}(W,V,R_{1})\subset K_{i}(W',V',R_{2}).
\]
\end{rem}
\begin{proof}
We only prove that $K_{o}$ is open and that $K_{i}$ is a $G_{\delta}$-set.
The other properties are easy to verify. First observe that 
\[
K_{o}=\bigcup_{\xi\in V}\left(h\mapsto h^{-T}\xi\right)^{-1}\left(C\left(W,R\right)\right)
\]
is open, because $C\left(W,R\right)$ is open.

Next, note that $V\subset\mathcal{O}\subset\mathbb{R}^{d}$ is an
open subset of $\mathbb{R}^{d}$, so that $V=\bigcup_{\ell}K_{\ell}$
is $\sigma$-compact. The definition of $K_{i}$ easily yields 
\[
K_{i}\left(W,V,R\right)=\bigcap_{\ell\in\mathbb{N}}K_{i}\left(W,K_{\ell},R\right),
\]
so that it suffices to show that each set $K_{i}\left(W,K_{\ell},R\right)$
is open.

To this end, let $h\in K_{i}\left(W,K_{\ell},R\right)$ be arbitrary.
This implies that $h^{-T}K_{\ell}\subset C\left(W,R\right)$ is a
compact set. As $C\left(W,R\right)$ is open, there is some $\varepsilon>0$
satisfying $B_{\varepsilon}\left(h^{-T}K_{\ell}\right)\subset C\left(W,R\right)$.
Let $L\subset H$ be an arbitrary compact unit-neighborhood. This
implies that 
\[
L\times h^{-T}K_{\ell}\rightarrow\mathbb{R}^{d},\left(g,\xi\right)\mapsto g^{-T}\xi
\]
is uniformly continuous. In particular, 
\[
\left|g^{-T}h^{-T}\xi-k^{-T}h^{-T}\xi\right|<\varepsilon
\]
holds for all $\xi\in K_{\ell}$ and all $g,k\in L$ with $\left\Vert g-k\right\Vert <\delta$
for a suitable $\delta>0$.

Setting $k={\rm id}$, we derive 
\[
\left(gh\right)^{-T}\xi\in B_{\varepsilon}\left(h^{-T}\xi\right)\subset B_{\varepsilon}\left(h^{-T}K_{\ell}\right)\subset C\left(W,R\right)
\]
for all $\xi\in K_{\ell}$ and $g\in L$ with $\left\Vert g-{\rm id}\right\Vert <\delta$.
Thus, $\left(B_{\delta}\left({\rm id}\right)\cap L\right)\cdot h\subset K_{i}\left(W,K_{\ell},R\right)$,
which implies that $K_{i}\left(W,K_{\ell},R\right)$ is an open subset
of $H$.
\end{proof}
We now come to the central technical assumption concerning the dual
action. Given a matrix $h$, we let $\|h\|$ denote the operator norm
of the induced linear map with respect to the euclidean norm. 
\begin{defn}
\label{defn:micro_regular}Let $\xi\in\mathcal{O}\cap S^{d-1}$ and
$\emptyset\neq V\Subset\mathcal{O}$. The dual action is called \textbf{$V$-microlocally
admissible in direction $\xi$} if there exists a $\xi$-neighborhood
$W_{0}\subset S^{d-1}\cap\mathcal{O}$ and some $R_{0}>0$ such that
the following hold: 
\begin{enumerate}[label=(\alph*)]
\item \label{enu:NormOfInverseEstimateOnKo}There exist $\alpha_{1}>0$
and $C>0$ such that 
\[
\|h^{-1}\|\le C\cdot\|h\|^{-\alpha_{1}}
\]
holds for all $h\in K_{o}(W_{0},V,R_{0})$. 
\item \label{enu:NormIntegrability}There exists $\alpha_{2}>0$ such that
\[
\int_{K_{o}(W_{0},V,R_{0})}\|h\|^{\alpha_{2}}\,{\rm d}h<\infty.
\]

\end{enumerate}
The dual action is called \textbf{microlocally admissible in direction
$\xi$} if it is $V$-microlocally admissible in direction $\xi$
for some $\emptyset\neq V\Subset\mathcal{O}$. It is called \textbf{globally
$V$-microlocally admissible} if there is $\emptyset\neq V\Subset\mathcal{O}$
such that the dual action is $V$-microlocally admissible in direction
$\xi$ for all $\xi\in\mathcal{O}\cap S^{d-1}$.
\end{defn}
We will see in the discussion below that these conditions are indeed
fulfilled in a variety of cases. 
\begin{rem}
\label{rem:MicrolocalAdmissiblityVCanBeShrunk}A simple but important
consequence of the above-observed inclusion properties for the $K_{o}$
(in particular the $V$-dependence, cf. Remark \ref{rem:incl_prop})
is that if the dual action is $V$-microlocally admissible in direction
$\xi$, it is $V'$-microlocally admissible in this direction for
all open $\emptyset\neq V'\subset V$. One can even choose the same
$W_{0},R_{0},\alpha_{1},\alpha_{2}$ and $C$ for all $V'\subset V$.

A similar reasoning allows to see that one may check the existence
of $W_{0},V,R_{0}$ fulfilling conditions \ref{enu:NormOfInverseEstimateOnKo}
and \ref{enu:NormIntegrability} separately, since decreasing $W_{0}$
and $V$ as well as increasing $R_{0}$ decreases $K_{o}(W_{0},V,R_{0})$,
hence it preserves the validity of properties \ref{enu:NormOfInverseEstimateOnKo}
and \ref{enu:NormIntegrability} in Definition \ref{defn:micro_regular}.
\end{rem}
In the case of a single orbit $\mathcal{O}=H^{T}\xi_{0}$, it suffices
to check $V$-microlocal admissibility in only one direction, as the
following lemma shows. 
\begin{lem}
\label{lem:SingleOrbitMicrolocalAdmissibility}Assume that $\mathcal{O}=H^{T}\xi_{0}$
is a single open orbit and let $\emptyset\neq V\Subset\mathcal{O}$.

If the dual action is $V$-microlocally admissible in direction $\xi_{1}$
for some $\xi_{1}\in\mathcal{O}\cap S^{d-1}$, then the dual action
is globally $V$-microlocally admissible.

One can even use the same exponents $\alpha_{1},\alpha_{2}$ (cf.
Definition \ref{defn:micro_regular}) for all $\xi\in\mathcal{O}\cap S^{d-1}$.\end{lem}
\begin{proof}
By assumption, there are $R_{0}>0$ and some $\xi_{1}$-neighborhood
$W_{0}\subset S^{d-1}\cap\mathcal{O}$ as well as $C,\alpha_{1},\alpha_{2}>0$
such that the conditions in Definition \ref{defn:micro_regular} are
fulfilled.

Observe that $B_{1}\left(\xi_{1}\right)\subset\mathbb{R}^{d}\setminus\left\{ 0\right\} $,
so that the map 
\[
\Phi:B_{1}\left(\xi_{1}\right)\to S^{d-1},w\mapsto\frac{w}{\left|w\right|}
\]
is well-defined and continous with $\Phi\left(\xi_{1}\right)=\xi_{1}$.
Hence, $\xi_{1}\in\Phi^{-1}\left(W_{0}\right)$, where the latter
set is open in $B_{1}\left(\xi_{1}\right)$. This shows that there
is some $\gamma\in\left(0,1\right)$ with 
\begin{equation}
\frac{w}{\left|w\right|}\in W_{0}\text{ for all }w\in B_{\gamma}\left(\xi_{1}\right).\label{eq:SingleOrbitMicrolocalAdmissibilityNormalizationInclusion}
\end{equation}

Now, let $\xi\in\mathcal{O}\cap S^{d-1}$ be arbitrary. As $\mathcal{O}$
is a single orbit with $\xi_{1}\in\mathcal{O}$, we get $\xi=h_{\xi}^{T}\xi_{1}$
for some $h_{\xi}\in H$. Define $R_{0}':=\left\Vert h_{\xi}\right\Vert R_{0}$
and $W_{0}':=\left[h_{\xi}^{T}\cdot B_{\gamma}\left(\xi_{1}\right)\right]\cap S^{d-1}$.
Observe that $W_{0}'$ is indeed a neighborhood of $\xi=h_{\xi}^{T}\xi_{1}$.
We will now prove 
\begin{equation}
K_{o}\left(W_{0}',V,R_{0}'\right)\subset h_{\xi}^{-1}\cdot K_{o}\left(W_{0},V,R_{0}\right).\label{eq:SingleOrbitMicrolocalAdmissibilityFundamentalInclusion}
\end{equation}
This will entail 
\[
\left\Vert \vphantom{h_{\xi}}\smash{h^{-1}}\right\Vert =\left\Vert \vphantom{h_{\xi}}\smash{\left(h_{\xi}h\right)^{-1}h_{\xi}}\right\Vert \leq\left\Vert \vphantom{h_{\xi}}\smash{\left(h_{\xi}h\right)^{-1}}\right\Vert \cdot\left\Vert h_{\xi}\right\Vert \overset{\left(\dagger\right)}{\leq}C\cdot\left\Vert h_{\xi}\right\Vert \cdot\left\Vert h_{\xi}h\right\Vert ^{-\alpha_{1}}\overset{\left(\ddagger\right)}{\leq}C\cdot\left\Vert h_{\xi}\right\Vert \left\Vert \vphantom{h_{\xi}}\smash{h_{\xi}^{-1}}\right\Vert ^{\alpha_{1}}\cdot\left\Vert h\right\Vert ^{-\alpha_{1}}
\]
for all $h\in K_{o}\left(W_{0}',V,R_{0}'\right)$, which is nothing
but part \ref{enu:NormOfInverseEstimateOnKo} of Definition \ref{defn:micro_regular}
at $\xi$ (with the same exponent $\alpha_{1}$). Here, we used $h_{\xi}h\in K_{o}\left(W_{0},V,R_{0}\right)$
at $\left(\dagger\right)$ and $\left\Vert h\right\Vert =\left\Vert \vphantom{h_{\xi}}\smash{h_{\xi}^{-1}h_{\xi}h}\right\Vert \leq\left\Vert \vphantom{h_{\xi}}\smash{h_{\xi}^{-1}}\right\Vert \cdot\left\Vert h_{\xi}h\right\Vert $
at $\left(\ddagger\right)$.

Furthermore, 
\begin{align*}
\int_{K_{o}\left(W_{0}',V,R_{0}'\right)}\left\Vert h\right\Vert ^{\alpha_{2}}\,{\rm d}h & \leq\int_{h_{\xi}^{-1}\cdot K_{o}\left(W_{0},V,R_{0}\right)}\left\Vert h_{\xi}^{-1}\cdot h_{\xi}h\right\Vert ^{\alpha_{2}}\,{\rm d}h\\
 & =\int_{K_{o}\left(W_{0},V,R_{0}\right)}\left\Vert \vphantom{h_{\xi}}\smash{h_{\xi}^{-1}}\cdot g\right\Vert ^{\alpha_{2}}\,{\rm d}g\\
 & \leq\left\Vert \vphantom{h_{\xi}}\smash{h_{\xi}^{-1}}\right\Vert ^{\alpha_{2}}\cdot\int_{K_{o}\left(W_{0},V,R_{0}\right)}\left\Vert g\right\Vert ^{\alpha_{2}}\,{\rm d}g<\infty,
\end{align*}
so that part \ref{enu:NormIntegrability} of Definition \ref{defn:micro_regular}
is also satisfied at $\xi$ (with the same exponent $\alpha_{2}$).

It remains to prove equation \eqref{eq:SingleOrbitMicrolocalAdmissibilityFundamentalInclusion}.
To this end, let $h\in K_{o}\left(W_{0}',V,R_{0}'\right)$ be arbitrary.
This yields some $v\in V$ with $h^{-T}v\in C\left(W_{0}',R_{0}'\right)$.
Hence, there are $w'\in W_{0}'\subset h_{\xi}^{T}\cdot B_{\gamma}\left(\xi_{1}\right)$
and $r>0$ with $h^{-T}v=r\cdot w'=r\cdot h_{\xi}^{T}w$ for some
$w\in B_{\gamma}\left(\xi_{1}\right)$. Together with equation \eqref{eq:SingleOrbitMicrolocalAdmissibilityNormalizationInclusion},
we see 
\[
\left(h_{\xi}h\right)^{-T}v=h_{\xi}^{-T}h^{-T}v=r\left|w\right|\cdot\frac{w}{\left|w\right|}\in r\left|w\right|\cdot W_{0}\subset C\left(W_{0}\right).
\]
Finally, $\left|h^{-T}v\right|>R_{0}'=\left\Vert h_{\xi}\right\Vert R_{0}$
because of $h^{-T}v\in C\left(W_{0}',R_{0}'\right)$. This implies
\[
\left\Vert h_{\xi}\right\Vert R_{0}<\left|h^{-T}v\right|=\left|h_{\xi}h_{\xi}^{-T}h^{-T}v\right|\leq\left\Vert h_{\xi}\right\Vert \cdot\left|h_{\xi}^{-T}h^{-T}v\right|=\left\Vert h_{\xi}\right\Vert \cdot\left|\left(h_{\xi}h\right)^{-T}v\right|
\]
and hence $\left|\vphantom{h_{\xi}}\smash{\left(h_{\xi}h\right)^{-T}}v\right|>R_{0}$.
In summary, we conclude 
\[
\left(h_{\xi}h\right)^{-T}v\in\left(h_{\xi}h\right)^{-T}V\cap C\left(W_{0},R_{0}\right)\neq\emptyset,
\]
which means $h_{\xi}h\in K_{o}\left(W_{0},V,R_{0}\right)$. Thus,
equation \eqref{eq:SingleOrbitMicrolocalAdmissibilityFundamentalInclusion}
is established.
\end{proof}
The following lemma provides important intuition for condition \ref{enu:NormOfInverseEstimateOnKo}
of microlocal admissibility of the dual action, by establishing a
systematic relationship between the norm of $h$ and the norms of
the frequencies contained in the support of $(\pi(y,h)\psi)^{\wedge}$. 
\begin{lem}
\label{lem:related_to_assumptions}Assume that the closed group $H\leq{\rm GL}\left(\mathbb{R}^{d}\right)$
satisfies the assumptions \ref{assume:proper_dual}. Then $0\notin\mathcal{O}$.
Furthermore, the following hold: 
\begin{enumerate}[label=(\alph*)]
\item \label{enu:XiEstimatedByH}Assume that $\emptyset\neq V\Subset\mathcal{O}$.
Then there exists a constant $C_{1}=C_{1}\left(V\right)>0$ such that,
for all $h\in H$ and all $\xi'\in V$: 
\[
\left|h^{-T}\xi'\right|^{-1}\le C_{1}\cdot\|h\|.
\]

\item \label{enu:HEstimatedByXi}Assume that the dual action fulfils condition
\ref{defn:micro_regular}\ref{enu:NormOfInverseEstimateOnKo}, for
some $\emptyset\not=V\Subset\mathcal{O}$, a suitable $\xi$-neighborhood
$W_{0}\subset S^{d-1}$ and some $R_{0}>0$.

Then there exist $\alpha>0$ and $C_{2}>0$ such that 
\[
\|h\|\le C_{2}\cdot\left|h^{-T}\xi'\right|^{-\alpha}.
\]
holds for all $h\in K_{o}(W_{0},V,R_{0})$ and $\xi'\in V$. 

\item \label{enu:NormBoundedOnKoAndDeterminantEstimate}Assume that the
dual action fulfils condition \ref{defn:micro_regular}\ref{enu:NormOfInverseEstimateOnKo},
for some $\emptyset\not=V\Subset\mathcal{O}$, a suitable $\xi$-neighborhood
$W_{0}\subset S^{d-1}$ and some $R_{0}>0$.

Then for all $\xi$-neighborhoods $W\subset W_{0}\subset S^{d-1}$
and all  $R\geq R_{0}$, the following is true: 
\[
{\rm sup}_{h\in K_{o}(W,V,R)}\left\Vert h\right\Vert \leq{\rm sup}_{h\in K_{o}(W_{0},V,R_{0})}\left\Vert h\right\Vert <\infty.
\]
Furthermore the inequalities 
\begin{equation}
\left|\det\left(h\right)\right|^{-\beta}\leq C_{3}^{\beta}\cdot\left\Vert h\right\Vert ^{-d\beta\alpha_{1}}\label{eq:InverseDeterminantEstimate}
\end{equation}
and 
\begin{equation}
\left(1+\left\Vert h^{-1}\right\Vert \right)^{M}\leq C_{4}\cdot\left\Vert h\right\Vert ^{-\alpha_{1}M},\label{eq:InverseNormEstimate}
\end{equation}
with $\alpha_{1}$ as in Definition \ref{defn:micro_regular}, hold
for all $h\in K_{o}\left(W,V,R\right)$ and all $\beta>0$ and $M\in\mathbb{N}_{0}$
for $C_{4}=C_{4}\left(M,W_{0},R_{0},V\right)>0$ and an absolute constant
$C_{3}=C_{3}\left(V,W_{0},R_{0}\right)>0$.

\end{enumerate}
\end{lem}
\begin{proof}
We first observe $0\notin\mathcal{O}$, because otherwise $H\times\left\{ 0\right\} =H_{\left\{ 0\right\} }$
is compact (cf. the properness assumption for the dual action in part
\ref{enu:DualActionIsProper} of Assumption \ref{assume:proper_dual}),
so that $H$ is compact. But by part \ref{enu:ExistenceOfAdmissibleFunction}
of Assumption \ref{assume:proper_dual} there is an admissible function
$\psi\in\mathcal{S}\left(\mathbb{R}^{d}\right)$ with $\widehat{\psi}\in C_{c}^{\infty}\left(\mathcal{O}\right)$.
This implies 
\[
0\neq1=\int_{H}\left|\widehat{\psi}\left(h^{T}\xi\right)^{2}\right|\,{\rm d}h
\]
for all $\xi\in\mathcal{O}$, so that there is some $h\in H$ with
$h^{T}\xi\in{\rm supp}\left(\smash{\widehat{\psi}}\right)$. Hence,
\[
\xi\in h^{-T}{\rm supp}\left(\smash{\widehat{\psi}}\right)\subset H^{T}{\rm supp}\left(\smash{\widehat{\psi}}\right),
\]
where the latter set is compact.

This shows that $\mathcal{O}\subset H^{T}{\rm supp}\left(\smash{\widehat{\psi}}\right)$
has to be bounded. But $\mathcal{O}$ is open with $0\in\mathcal{O}$,
so that there is some $\xi\in\mathcal{O}\setminus\left\{ 0\right\} $.
By part \ref{enu:DualActionContainsRays} of Assumption \ref{assume:proper_dual},
this yields $\mathbb{R}^{+}\xi\subset\mathcal{O}$, which contradicts
boundedness. This contradiction shows $0\notin\mathcal{O}$.

As $\overline{V}\subset\mathcal{O}\subset\mathbb{R}^{d}\setminus\left\{ 0\right\} $
is compact, we conclude that $C_{0}:=\min_{\xi'\in\overline{V}}\left|\xi'\right|$
is positive.

For the proof of \ref{enu:XiEstimatedByH}, we note that each $\xi'\in V$
satisfies the estimate 
\[
\left|\xi'\right|=\left|h^{T}h^{-T}\xi'\right|\le\|h^{T}\|\cdot\left|h^{-T}\xi'\right|,
\]
which entails 
\[
\left|h^{-T}\xi'\right|^{-1}\le\left\Vert h\right\Vert \big/\left|\xi'\right|\leq\frac{1}{C_{0}}\cdot\left\Vert h\right\Vert .
\]

For the proof of part \ref{enu:HEstimatedByXi}, we observe that our
assumptions yield constants $C,\alpha_{1}>0$ such that 
\[
\left|h^{-T}\xi'\right|\le\left\Vert h^{-1}\right\Vert \cdot\left|\xi'\right|\le C\cdot\max_{\eta\in\overline{V}}\left|\eta\right|\cdot\left\Vert h\right\Vert ^{-\alpha_{1}}
\]
holds for all $h\in K_{o}\left(W_{0},V,R_{0}\right)$. Taking both
sides to the power $-1/\alpha_{1}$ yields the claim with $\alpha=1/\alpha_{1}>0$.

For the proof of part \ref{enu:NormBoundedOnKoAndDeterminantEstimate},
let $h\in K_{o}(W,V,R)\subset K_{o}\left(W_{0},V,R_{0}\right)$ for
arbitrary $W\subset W_{0}$ and $R\geq R_{0}$ with $W_{0},R_{0}$
as above. By definition, this implies that there is some $\xi'\in h^{-T}V\cap C(W,R)$.
In particular, $|\xi'|>R\geq R_{0}$, and part \ref{enu:HEstimatedByXi},
applied to $h^{T}\xi'\in V$, yield 
\[
\|h\|\le C_{2}\cdot\left|h^{-T}h^{T}\xi'\right|^{-\alpha}=C_{2}\cdot\left|\xi'\right|^{-\alpha}\le C_{2}\cdot R_{0}^{-\alpha}.
\]
For the proof of equation \eqref{eq:InverseDeterminantEstimate},
recall that Hadamard's inequality implies $\left|\det\left(g\right)\right|\leq\left\Vert g\right\Vert ^{d}$
for all $g\in\mathbb{R}^{d\times d}$. We conclude 
\[
\left|\det\left(h\right)\right|^{-\beta}=\left|\det\left(h^{-1}\right)\right|^{\beta}\leq\left\Vert h^{-1}\right\Vert ^{\beta d}\leq C^{\beta d}\cdot\left\Vert h\right\Vert ^{-\alpha_{1}\beta d}
\]
for all $h\in K_{o}\left(W_{0},V,R_{0}\right)\supset K_{o}\left(W,V,R\right)$,
where the estimate in the last step is due to part \ref{enu:NormOfInverseEstimateOnKo}
of Definition \ref{defn:micro_regular}.

Finally, 
\[
1=\left\Vert h^{-1}h\right\Vert \leq\left\Vert h^{-1}\right\Vert \cdot\left\Vert h\right\Vert \leq C_{4}\cdot\left\Vert h^{-1}\right\Vert 
\]
holds for all $h\in K_{o}\left(W_{0},V,R_{0}\right)\supset K_{o}\left(W,V,R\right)$,
because $\left\Vert h\right\Vert $ is bounded on this set. Using
the constant $C$ provided by Definition \ref{defn:micro_regular}\ref{enu:NormOfInverseEstimateOnKo},
we see 
\[
\left(1+\left\Vert h^{-1}\right\Vert \right)^{M}\leq\left(\left(C_{4}+1\right)\left\Vert h^{-1}\right\Vert \right)^{M}\leq\left[C\left(1+C_{4}\right)\cdot\left\Vert h\right\Vert ^{-\alpha_{1}}\right]^{M},
\]
which establishes estimate \eqref{eq:InverseNormEstimate}.
\end{proof}

\section{Wavelet criteria for regular directed points}

\label{sect:central}We first establish some basic growth or decay
estimates concerning wavelet transforms of tempered distributions
and Schwartz functions. In the following, we use the Schwartz norms
\[
\left|\psi\right|_{N}:=\max_{\alpha\in\mathbb{N}_{0}^{d},\,\left|\alpha\right|\leq N}\sup_{z\in\mathbb{R}^{d}}(1+|z|)^{N}\left|\partial^{\alpha}\psi(z)\right|.
\]
Note that these norms are invariant under complex conjugation. 
\begin{lem}
\label{lem:wc_decay_general}Let $\psi,\varphi\in\mathcal{S}(\mathbb{R}^{d})$
and $u\in\mathcal{S}'(\mathbb{R}^{d})$. 
\begin{enumerate}[label=(\alph*)]
\item \label{enu:QuasiRegularSchwartzNormForFunctions}For all $N\in\mathbb{N}$
and $(x,h)\in G$, the following inequality holds: 
\[
\left|\pi(x,h)\psi\right|_{N}\le C_{N}|\psi|_{N}\cdot\left|\det(h)\right|^{-1/2}(1+\|h^{-1}\|)^{N}\cdot\max\left\{ 1,\|h\|^{N}\right\} \cdot(1+|x|)^{N},
\]
with a constant $C_{N}$ independent of $\psi,x,h$. 
\item \label{enu:WaveletTrafoGeneralEstimateForDistributions}There exists
$N=N\left(u\right)\in\mathbb{N}$ such that for all $\left(x,h\right)\in G$,
the following inequality holds: 
\[
|W_{\psi}u(x,h)|\le C\cdot\left|\det(h)\right|^{-1/2}(1+\|h^{-1}\|)^{N}\cdot\max\left\{ 1,\|h\|^{N}\right\} \cdot(1+|x|)^{N}
\]
with $C>0$ depending on $\psi$ and $u$ but not on $x,h$. 
\item \label{enu:WaveletTrafoGeneralEstimateForSchwartzFunctions}For all
$N\in\mathbb{N}$ and $(x,h)\in G$, we have 
\[
|W_{\psi}\varphi(x,h)|\le C_{N}|\varphi|_{d+N+1}|\psi|_{N}\cdot\left|\det(h)\right|^{-1/2}(1+\|h^{-1}\|)^{N}\cdot\max\left\{ 1,\|h\|^{N}\right\} \cdot(1+|x|)^{-N}
\]
with $C_{N}$ independent of $\varphi,\psi,h,x$. 
\item \label{enu:WaveletTrafoEstimateForSchwartzFunctionsOnTheCone}Assume
that the dual action is $V$-microlocally admissible at $\xi$ for
some $\emptyset\neq V\Subset\mathcal{O}$ and that $\psi\in\mathcal{S}\left(\mathbb{R}^{d}\right)$
with ${\rm supp}(\widehat{\psi})\subset V$. Choose $R_{0}>0$ and
a $\xi$-neighborhood $W_{0}\subset S^{d-1}$ as in Definition \ref{defn:micro_regular}.
Then 
\[
|W_{\psi}\varphi(x,h)|\le C_{M,N,\psi,W_{0},V,R_{0}}\cdot|\varphi|_{M+N}\cdot\left|\det(h)\right|^{-1/2}\cdot(1+|x|)^{-N}\|h\|^{M}
\]
holds for all $x\in\mathbb{R}^{d}$, $h\in K_{o}\left(W_{0},V,R_{0}\right)$
and $M,N\in\mathbb{N}$, where the constant $C_{M,N,\psi,W_{0},V,R_{0}}$
is independent of $x,h$ and $\varphi$. 
\end{enumerate}
\end{lem}
\begin{proof}
For the proof of part \ref{enu:QuasiRegularSchwartzNormForFunctions},
we first compute the effect of dilation on the Schwartz norm. Here
we have 
\begin{align*}
\left|\pi(0,h)\psi\right|_{N} & =\left|\det(h)\right|^{-1/2}\cdot\sup\left\{ (1+|z|)^{N}|\partial^{\alpha}(y\mapsto\psi(h^{-1}y))(z)|\with|\alpha|\le N,z\in\mathbb{R}^{d}\right\} \\
 & \le C_{N}\left|\det(h)\right|^{-1/2}\cdot(1+\|h^{-1}\|)^{N}\cdot\sup\left\{ (1+|z|)^{N}|(\partial^{\alpha}\psi)(h^{-1}z)|\with|\alpha|\le N,z\in\mathbb{R}^{d}\right\} ,
\end{align*}
using the chain rule. Now we can continue estimating 
\begin{align*}
\ldots & \le C_{N}|\psi|_{N}\left|\det(h)\right|^{-1/2}\cdot(1+\|h^{-1}\|)^{N}\cdot\sup\left\{ (1+|z|)^{N}(1+|h^{-1}z|)^{-N}\with z\in\mathbb{R}^{d}\right\} \\
 & \le C_{N}'\cdot|\psi|_{N}\cdot\left|\det(h)\right|^{-1/2}(1+\|h^{-1}\|)^{N}\cdot\max\left\{ 1,\|h\|^{N}\right\} .
\end{align*}
In the last step, we made use of the elementary estimate 
\[
1+\left|z\right|=1+\left|hh^{-1}z\right|\leq1+\left\Vert h\right\Vert \cdot\left|h^{-1}z\right|\leq\left(1+\left\Vert h\right\Vert \right)\cdot\left(1+\left|h^{-1}z\right|\right)
\]
which leads to $\left(1+\left|z\right|\right)^{N}\left(1+\left|h^{-1}z\right|\right)^{-N}\leq\left(1+\left\Vert h\right\Vert \right)^{N}$.

By a similar (but easier) argument (using the inequality \eqref{eq:TriangleInequalityProductForm}
below), we get 
\[
\left|\pi(x,{\rm id})\psi\right|_{N}\le(1+|x|)^{N}|\psi|_{N}.
\]
Now \ref{enu:QuasiRegularSchwartzNormForFunctions} follows from these
calculations together with $\pi(x,h)=\pi\left(x,{\rm id}\right)\circ\pi\left(0,h\right)$.

For $u\in\mathcal{S}'(\mathbb{R}^{d})$, there exists $N=N\left(u\right)\in\mathbb{N}$
and a constant $C=C\left(u\right)>0$ such that $\left|u\left(\psi\right)\right|\leq C\cdot\left|\psi\right|_{N}$
holds for all Schwartz functions $\psi\in\mathcal{S}\left(\mathbb{R}^{d}\right)$.
Together with the definition $\left(W_{\psi}u\right)\left(x,h\right)=\left\langle u\mid\pi\left(x,h\right)\psi\right\rangle $,
we see that part \ref{enu:WaveletTrafoGeneralEstimateForDistributions}
follows from part \ref{enu:QuasiRegularSchwartzNormForFunctions}.

Part \ref{enu:WaveletTrafoGeneralEstimateForSchwartzFunctions} follows
from similar considerations: We first consider the decay behaviour
of convolution products of Schwartz functions: For any $N\in\mathbb{N}$,
we use the inequality 
\begin{equation}
1+\left|x\right|\leq1+\left|x-y\right|+\left|y\right|\leq\left(1+\left|x-y\right|\right)\left(1+\left|y\right|\right)\label{eq:TriangleInequalityProductForm}
\end{equation}
to derive 
\begin{align*}
\left(1+\left|y\right|\right)^{-d-1-N}\left(1+\left|x-y\right|\right)^{-N} & =\left(1+\left|y\right|\right)^{-d-1}\left[\left(1+\left|x-y\right|\right)\left(1+\left|y\right|\right)\right]^{-N}\\
 & \leq\left(1+\left|y\right|\right)^{-d-1}\cdot\left(1+\left|x\right|\right)^{-N}
\end{align*}
and thus 
\begin{align*}
\left|\left(\varphi\ast\psi\right)(x)\right| & \le|\varphi|_{d+N+1}|\psi|_{N}\int_{\mathbb{R}^{d}}(1+|y|)^{-d-1-N}(1+|x-y|)^{-N}\,{\rm d}y\\
 & \le C\cdot|\varphi|_{d+N+1}|\psi|_{N}\cdot(1+|x|)^{-N}
\end{align*}
with $C=C_{d}=\int_{\mathbb{R}^{d}}\left(1+\left|y\right|\right)^{-d-1}\,{\rm d}y$.

We now combine this observation with part \ref{enu:QuasiRegularSchwartzNormForFunctions}
and with the (easily verifiable) identity 
\[
\left(W_{\psi}\varphi\right)(x,h)=(\varphi\ast\pi(0,h)\psi^{*})(x),
\]
where $\psi^{\ast}\left(y\right)=\overline{\psi\left(-y\right)}$,
to obtain the desired estimate.

Finally, for part \ref{enu:WaveletTrafoEstimateForSchwartzFunctionsOnTheCone},
we first apply the standard Fourier-analytic arguments relating smoothness
on the Fourier side and decay on the space side together with 
\begin{align*}
\left(W_{\psi}\varphi\right)\left(x,h\right) & =\left[\mathcal{F}^{-1}\left(\mathcal{F}\left(\varphi\ast\left[\pi\left(0,h\right)\psi^{\ast}\right]\right)\right)\right]\left(x\right)\\
 & =\left[\mathcal{F}^{-1}\left(\widehat{\varphi}\cdot\left(\pi(0,h)\psi^{*}\right)^{\wedge}\right)\right]\left(x\right)
\end{align*}
to derive 
\[
|W_{\psi}\varphi(x,h)|\le C_{N}\cdot(1+|x|)^{-N}\cdot\max_{|\alpha|\le N}\left\Vert \partial^{\alpha}\left(\widehat{\varphi}\cdot\left(\pi(0,h)\psi^{*}\right)^{\wedge}\right)\right\Vert _{1}.
\]
Using Leibniz' formula, together with the fact that $(\pi(0,h)\psi^{*})^{\wedge}=\left|\det\left(h\right)\right|^{1/2}\cdot\overline{\widehat{\psi}\left(h^{T}\cdot\right)}$
is supported inside $h^{-T}V$, each integrand can be estimated by
\begin{align*}
 & \hphantom{\leq\,}\left|\left[\partial^{\alpha}\left(\widehat{\varphi}\cdot\left(\pi\left(0,h\right)\psi^{\ast}\right)^{\wedge}\right)\right]\left(\xi'\right)\right|\\
 & \leq C_{N}\cdot\left|\det(h)\right|^{1/2}(1+\|h\|)^{N}\cdot\sum_{\beta+\gamma=\alpha}\binom{\alpha}{\gamma}\,\left|\left(\smash{\partial^{\beta}\widehat{\varphi}}\right)(\xi')\right|\,\left|\left(\smash{\partial^{\gamma}\widehat{\psi}}\right)(h^{T}\xi')\right|\\
 & \leq C_{N}'\cdot\left[|\det(h)|^{1/2}(1+\|h\|)^{N}\sum_{|\gamma|\le N}\left|\left(\smash{\partial^{\gamma}\widehat{\psi}}\right)(h^{T}\xi')\right|\right]\cdot|\varphi|_{M+N}\cdot\sup_{\eta\in h^{-T}V}(1+|\eta|)^{-M}.
\end{align*}
But Lemma \ref{lem:related_to_assumptions}\ref{enu:XiEstimatedByH}
yields 
\[
\sup_{\eta\in h^{-T}V}(1+|\eta|)^{-M}=\sup_{\xi'\in V}(1+|h^{-T}\xi'|)^{-M}\leq\sup_{\xi'\in V}\left|h^{-T}\xi'\right|^{-M}\le C_{M,V}''\cdot\|h\|^{M}.
\]
Furthermore, Lemma \ref{lem:related_to_assumptions}\ref{enu:NormBoundedOnKoAndDeterminantEstimate}
allows to bound the factor $(1+\|h\|)^{N}$ uniformly on $K_{o}(W_{0},V,R_{0})$.

In total, we arrive at 
\begin{align*}
|W_{\psi}\varphi(x,h)| & \le C_{N,M,V,W_{0},R_{0}}\cdot|\varphi|_{M+N}\cdot(1+|x|)^{-N}\|h\|^{M}\sum_{|\gamma|\le N}|\det(h)|^{1/2}\int_{\mathbb{R}^{d}}\left|\left(\smash{\partial^{\gamma}\widehat{\psi}}\right)(h^{T}\xi')\right|\,{\rm d}\xi'\\
 & =C_{N,M,V,W_{0},R_{0}}\cdot|\varphi|_{M+N}\cdot(1+|x|)^{-N}\|h\|^{M}\cdot|\det(h)|^{-1/2}\sum_{|\gamma|\le N}\int_{\mathbb{R}^{d}}\left|\left(\smash{\partial^{\gamma}\widehat{\psi}}\right)(\eta)\right|\,{\rm d}\eta.\qedhere
\end{align*}

\end{proof}
We next address how wavelet coefficient decay and/or regular directed
points are affected by certain multiplication operators, either in
space or in frequency domain. The first observation pertains to localization
in the space domain and is well-known. For the sake of completeness,
we nevertheless provide a proof. 
\begin{lem}
\label{lem:reg_points_loc}Let $u\in\mathcal{S}'(\mathbb{R}^{d})$
and $(x,\xi)\in\mathbb{R}^{d}\times S^{d-1}$. Let $\varphi\in C_{c}^{\infty}(\mathbb{R}^{d})$
be identically one in some neighborhood of $x$. Then $(x,\xi)$ is
a regular directed point of $u$ iff it is a regular directed point
of $\varphi u$. \end{lem}
\begin{proof}
``$\Leftarrow$'': If $\varphi\equiv1$ on $B_{\varepsilon}\left(x_{0}\right)$
and $\left(x_{0},\xi_{0}\right)$ is a regular directed point of $\varphi u$,
then there is some function $\psi\in C_{c}^{\infty}\left(\mathbb{R}^{d}\right)$
with $\psi\equiv1$ on $B_{\delta}\left(x_{0}\right)$ for some $\delta>0$
such that 
\[
\left|\mathcal{F}\left(\psi\cdot\varphi u\right)\left(\xi\right)\right|\leq C_{N}\cdot\left(1+\left|\xi\right|\right)^{-N}
\]
holds for all $N\in\mathbb{N}$ and all $\xi\in C\left(W\right)$
for some (fixed) $\xi_{0}$-neighborhood $W\subset S^{d-1}$. Because
of $\psi\varphi\in C_{c}^{\infty}\left(\mathbb{R}^{d}\right)$ with
$\psi\varphi\equiv1$ on $B_{\min\left\{ \varepsilon,\delta\right\} }\left(x_{0}\right)$,
this means that $\left(x_{0},\xi_{0}\right)$ is a regular directed
point of $u$.

``$\Rightarrow$'': The following is loosely based on the proof
of \cite[Part B, Lemma 1.1.1]{PseudoDifferentialAndNashMoser}. Let
$\left(x_{0},\xi_{0}\right)$ be a regular directed point of $u$.
We will show the more general claim that $\left(x_{0},\xi_{0}\right)$
is a regular directed point of $\varphi u$ for any $\varphi\in\mathcal{S}\left(\mathbb{R}^{d}\right)$.
By definition of a regular directed point, there is $\psi\in C_{c}^{\infty}\left(\mathbb{R}^{d}\right)$
with $\psi\equiv1$ on a neighborhood of $x_{0}$ and a $\xi_{0}$-neighborhood
$W\subset S^{d-1}$ such that 
\begin{equation}
\left|\widehat{\psi u}\left(\xi\right)\right|\leq C_{N}\cdot\left(1+\left|\xi\right|\right)^{-N}\label{eq:SpaceLocalizationAssumption}
\end{equation}
holds for all $\xi\in C\left(W\right)$ for all $N\in\mathbb{N}$.

By definition of the relative topology, we have $B_{\delta}\left(\xi_{0}\right)\cap S^{d-1}\subset W$
for some $\delta\in\left(0,1\right)$. Let $c:=\delta/8<\frac{1}{2}$.
This implies $1-c>1/2$ and hence 
\[
\frac{2c}{1-c}\leq4c=\frac{\delta}{2}.
\]
We now show 
\begin{equation}
\forall\xi\in C\left(B_{\delta/2}\left(\xi_{0}\right)\right)\,\forall\eta\in\mathbb{R}^{d}\text{ with }\left|\eta\right|<c\left|\xi\right|:\qquad\xi-\eta\in C\left(W\right).\label{eq:SpaceLocalizationFundamentalInequality}
\end{equation}
To this end, first observe 
\begin{equation}
\left|\xi-\eta\right|\geq\left|\xi\right|-\left|\eta\right|\geq\left(1-c\right)\left|\xi\right|\geq\left|\xi\right|/2>0\label{eq:SpaceLocalizationXiMinusEtaIsLarge}
\end{equation}
and $\left|\left|\xi\right|-\left|\xi-\eta\right|\right|\leq\left|\eta\right|<c\left|\xi\right|$,
which implies 
\begin{align*}
\left|\frac{\xi-\eta}{\left|\xi-\eta\right|}-\xi_{0}\right| & \leq\left|\frac{\xi-\eta}{\left|\xi-\eta\right|}-\frac{\xi}{\left|\xi\right|}\right|+\left|\frac{\xi}{\left|\xi\right|}-\xi_{0}\right|\\
 & <\frac{\left|\left(\left|\xi\right|-\left|\xi-\eta\right|\right)\xi-\left|\xi\right|\eta\right|}{\left|\xi-\eta\right|\cdot\left|\xi\right|}+\frac{\delta}{2}\\
 & \leq\frac{c\left|\xi\right|^{2}+c\left|\xi\right|^{2}}{\left(1-c\right)\cdot\left|\xi\right|^{2}}+\frac{\delta}{2}\\
 & \leq\frac{\delta}{2}+\frac{\delta}{2}=\delta
\end{align*}
and hence $\xi-\eta\in\left|\xi-\eta\right|\cdot\left(B_{\delta}\left(\xi_{0}\right)\cap S^{d-1}\right)\subset C\left(W\right)$.

Now, set $\varrho:=\mathcal{F}^{-1}\varphi$. Recall from \cite[Theorem 7.23]{RudinFA}
that the Fourier transform of a compactly supported (tempered) distribution
$f$ is given by (integration against) the smooth, polynomially bounded
function $\widehat{f}\left(\xi\right)=f\left(e^{-2\pi i\left\langle \cdot,\xi\right\rangle }\right)$.
This implies 
\begin{align*}
\left(\mathcal{F}\left(\psi\cdot\varphi u\right)\right)\left(\xi\right) & =\left(\psi\cdot\varphi u\right)\left(e^{-2\pi i\left\langle \cdot,\xi\right\rangle }\right)\\
 & =\left(\psi u\right)\left(\varphi\cdot e^{-2\pi i\left\langle \cdot,\xi\right\rangle }\right)\\
 & =\left(\psi u\right)\left(\widehat{L_{\xi}\varrho}\right)\\
 & =\widehat{\psi u}\left(L_{\xi}\varrho\right)\\
 & =\int_{\mathbb{R}^{d}}\widehat{\psi u}\left(\xi-\eta\right)\cdot\varrho\left(-\eta\right)\,{\rm d}\eta.
\end{align*}
We now split the domain of the last integral into the parts $\left|\eta\right|<c\left|\xi\right|$
and $\left|\eta\right|\geq c\left|\xi\right|$. For the first part,
we use equations \eqref{eq:SpaceLocalizationAssumption}, \eqref{eq:SpaceLocalizationFundamentalInequality}
and \eqref{eq:SpaceLocalizationXiMinusEtaIsLarge} to estimate, for
each $\xi\in C\left(B_{\delta/2}\left(\xi_{0}\right)\right)$, 
\begin{align*}
\left|\int_{\left|\eta\right|<c\left|\xi\right|}\widehat{\psi u}\left(\xi-\eta\right)\cdot\varrho\left(-\eta\right)\,{\rm d}\eta\right| & \leq C_{N}\cdot\int_{\left|\eta\right|<c\left|\xi\right|}\left(1+\left|\xi-\eta\right|\right)^{-N}\cdot\left|\varrho\left(-\eta\right)\right|\,{\rm d}\eta\\
 & \leq C_{N}\cdot\left(1+\frac{\left|\xi\right|}{2}\right)^{-N}\cdot\left\Vert \varrho\right\Vert _{1}\\
 & \leq2^{N}C_{N}\cdot\left\Vert \varrho\right\Vert _{1}\cdot\left(1+\left|\xi\right|\right)^{-N}.
\end{align*}
For the second part, observe that \cite[Theorem 7.23]{RudinFA} shows
that $\widehat{\psi u}$ is a polynomially bounded function, i.e.
$\vphantom{\left|\widehat{\psi u}\right|}\left|\smash{\widehat{\psi u}}\left(\xi-\eta\right)\right|\leq C\cdot\left(1+\left|\xi-\eta\right|\right)^{M}$
for suitable $M\in\mathbb{N}_{0}$ and $C>0$ for all $\xi,\eta\in\mathbb{R}^{d}$.
Together with 
\[
1+\left|\xi-\eta\right|\leq1+\left|\xi\right|+\left|\eta\right|\leq1+\left(1+c^{-1}\right)\left|\eta\right|\leq\left[1+\left(1+c^{-1}\right)\right]\left(1+\left|\eta\right|\right)=:C_{c}\cdot\left(1+\left|\eta\right|\right),
\]
this leads to 
\begin{align*}
\left|\int_{\left|\eta\right|\geq c\left|\xi\right|}\widehat{\psi u}\left(\xi-\eta\right)\cdot\varrho\left(-\eta\right)\,{\rm d}\eta\right| & \leq C\left|\varrho\right|_{M+K+d+1}\int_{\left|\eta\right|\geq c\left|\xi\right|}\left(1+\left|\xi-\eta\right|\right)^{M}\cdot\left(1+\left|\eta\right|\right)^{-M-K-d-1}\,{\rm d}\eta\\
 & \leq C\cdot C_{c}^{M}\cdot\left|\varrho\right|_{M+K+d+1}\cdot\left(1+c\left|\xi\right|\right)^{-K}\cdot\int_{\left|\eta\right|\geq c\left|\xi\right|}\left(1+\left|\eta\right|\right)^{-d-1}\,{\rm d}\eta\\
 & \leq C'\cdot\left(1+\left|\xi\right|\right)^{-K}
\end{align*}
for each $K\in\mathbb{N}_{0}$, where $C'$ is of the form $C'=C'\left(\varrho,C,K,M,d,c\right)=C'\left(\delta,d,\varphi,\psi,u,K\right)$.
This completes the proof. 
\end{proof}
The following Fourier localization statement for regular directed
points is probably folklore, but since it is central to our argument,
we include a proof.
\begin{lem}
\label{lem:Four_loc}Let $\zeta:\mathbb{R}^{d}\to\mathbb{C}$ denote
a $C^{\infty}$-function with polynomially bounded partial derivatives.
Then, for $u\in\mathcal{S}'(\mathbb{R}^{d})$, the Fourier localization
$P_{\zeta}u=\mathcal{F}^{-1}\left(\zeta\cdot\widehat{u}\right)$,
i.e. 
\[
P_{\zeta}u:\,\mathcal{S}(\mathbb{R}^{d})\to\mathbb{C},\,\varphi\mapsto\widehat{u}\left(\zeta\cdot\mathcal{F}^{-1}\varphi\right)
\]
is a well-defined tempered distribution. If 
\[
\zeta|_{C(W,R)}\equiv1
\]
holds for some $\xi_{0}\in S^{d-1}$, some $R>0$ and some $\xi_{0}$-neighborhood
$W\subset S^{d-1}$ and if $\left(x_{0},\xi_{0}\right)$ is a regular
directed point of $P_{\zeta}u$, then $\left(x_{0},\xi_{0}\right)$
is also a regular directed point of $u$.\end{lem}
\begin{rem*}
One can show (using a similar proof) that the reverse implication
is also valid, but we will not need this in the following.\end{rem*}
\begin{proof}
Well-definedness of $P_{\zeta}u$ follows from the fact that $\varphi\mapsto\varphi\cdot\zeta$
is a continuous linear operator on $\mathcal{S}(\mathbb{R}^{d})$
(because $\zeta$ has polynomially bounded derivatives) and because
the Fourier transform is a homeomorphism $\mathcal{F}:\mathcal{S}\left(\mathbb{R}^{d}\right)\to\mathcal{S}\left(\mathbb{R}^{d}\right)$.

Let $R'>\max\left\{ 1,R\right\} $ be arbitrary and define $W_{\delta}:=B_{\delta}\left(\xi_{0}\right)\cap S^{d-1}$
for $\delta>0$. The main geometric fact on which the proof is based
is the following estimate, valid for all $\delta\in\left(0,1\right)$,
\begin{equation}
\forall\xi\in C\left(W_{\delta/2},\frac{4}{3}R'\right)\,\forall y\in\mathbb{R}^{d}\setminus C\left(W_{\delta},R'\right):\qquad\left|y-\xi\right|\geq\frac{\delta}{4}\cdot\left|\xi\right|.\label{eq:FourierSideLocalizationGeometricEstimate}
\end{equation}
For the proof of this inequality, we distringuish two cases:
\begin{casenv}
\item We have $\left|\left|\xi\right|-\left|y\right|\right|>\frac{\delta}{4}\left|\xi\right|$.
In this case, simply note 
\[
\left|y-\xi\right|\geq\left|\left|y\right|-\left|\xi\right|\right|>\frac{\delta}{4}\cdot\left|\xi\right|.
\]

\item We have $\left|\left|\xi\right|-\left|y\right|\right|\leq\frac{\delta}{4}\left|\xi\right|$.
Using $\delta\leq1$ and $\left|\xi\right|>\frac{4}{3}R'$, this yields
\[
\left|y\right|\geq\left|\xi\right|-\frac{\delta}{4}\left|\xi\right|\geq\frac{3}{4}\left|\xi\right|>R'.
\]
In particular, $\left|y\right|>0$. In case of $\frac{y}{\left|y\right|}\in B_{\delta}\left(\xi_{0}\right)$,
this would imply $y=\left|y\right|\frac{y}{\left|y\right|}\in C\left(W_{\delta}\right)$
and hence $y\in C\left(W_{\delta},R'\right)$ in contradiction to
the assumptions of equation \eqref{eq:FourierSideLocalizationGeometricEstimate}.
Hence, $\frac{y}{\left|y\right|}\notin B_{\delta}\left(\xi_{0}\right)$.

But $\xi\in C\left(W_{\delta/2},\frac{4}{3}R'\right)\subset C\left(W_{\delta/2}\right)$,
which yields some $\xi'\in W_{\delta/2}=B_{\delta/2}\left(\xi_{0}\right)\cap S^{d-1}$
and $r>0$ with $\xi=r\cdot\xi'$. This implies $r=\left|\xi\right|$
and hence $\frac{\xi}{\left|\xi\right|}=\xi'\in B_{\delta/2}\left(\xi_{0}\right)$.
Together, we arrive at 
\[
\left|\frac{y}{\left|y\right|}-\frac{\xi}{\left|\xi\right|}\right|>\frac{\delta}{2}
\]
and thus 
\begin{align*}
\left|y-\xi\right| & =\left|\xi\right|\cdot\left|\frac{\xi}{\left|\xi\right|}-\frac{y}{\left|\xi\right|}\right|\\
 & \geq\left|\xi\right|\cdot\left|\frac{\xi}{\left|\xi\right|}-\frac{y}{\left|y\right|}\right|-\left|\xi\right|\cdot\left|\frac{y}{\left|y\right|}-\frac{y}{\left|\xi\right|}\right|\\
 & >\frac{\delta}{2}\left|\xi\right|-\left|\frac{y}{\left|y\right|}\cdot\left[\left|\xi\right|-\left|y\right|\right]\right|\\
 & =\frac{\delta}{2}\left|\xi\right|-\left|\left|\xi\right|-\left|y\right|\right|\geq\frac{\delta}{4}\left|\xi\right|,
\end{align*}
where we used the assumption $\left|\left|\xi\right|-\left|y\right|\right|\leq\frac{\delta}{4}\left|\xi\right|$
of this case in the last step.

\end{casenv}
Let $\varphi\in C_{c}^{\infty}\left(\mathbb{R}^{d}\right)$ be arbitrary
and set $\psi_{\varphi}:=\mathcal{F}^{-1}\varphi$. Using the formula
$\widehat{f}\left(\xi\right)=f\left(e^{-2\pi i\left\langle \xi,\cdot\right\rangle }\right)$
for the Fourier transform of a compactly supported (tempered) distribution
$f$ -- as given in \cite[Theorem 7.23]{RudinFA} -- we calculate
\begin{align*}
g_{\varphi}\left(\xi\right):=\left(\mathcal{F}\left[\varphi\cdot P_{1-\zeta}u\right]\right)\left(\xi\right) & =\left(P_{1-\zeta}u\right)\left(\varphi\cdot e^{-2\pi i\left\langle \xi,\cdot\right\rangle }\right)\\
 & =\widehat{u}\left(\left(1-\zeta\right)\cdot\mathcal{F}^{-1}\left(\varphi\cdot e^{-2\pi i\left\langle \xi,\cdot\right\rangle }\right)\right)\\
 & =\widehat{u}\left(\left(1-\zeta\right)\cdot L_{\xi}\left(\mathcal{F}^{-1}\varphi\right)\right)\\
 & =\widehat{u}\left(\left(1-\zeta\right)\cdot L_{\xi}\psi_{\varphi}\right),
\end{align*}
where $L_{\xi}\psi_{\varphi}$ is the left-translate of $\psi_{\varphi}$,
defined by $\left(L_{\xi}\psi_{\varphi}\right)\left(\eta\right)=\psi_{\varphi}\left(\eta-\xi\right)$.

By definition of the relative topology, $W_{\delta}\subset W$ holds
for some $\delta>0$. We want to show that $g_{\varphi}$ has rapid
decay on $C\left(W_{\delta/2},\frac{4}{3}R'\right)$. To this end,
first note that 
\[
\left|\widehat{u}\left(f\right)\right|\leq C_{u}\cdot\left|f\right|_{K}
\]
holds for suitable $C_{u}>0$ and $K=K\left(u\right)\in\mathbb{N}$
for all $f\in\mathcal{S}\left(\mathbb{R}^{d}\right)$, because $\widehat{u}$
is a tempered distribution. Furthermore, $\zeta$ has polynomially
bounded derivatives of all orders, so that the same is true of $1-\zeta$.
Also, $1-\zeta$ vanishes on $C\left(W,R\right)\supset C\left(W_{\delta},R'\right)$.

Let $\xi\in C\left(W_{\delta/2},\frac{4}{3}R'\right)$. Recall the
estimate $\left|y-\xi\right|\geq\frac{\delta}{4}\left|\xi\right|$
from equation \eqref{eq:FourierSideLocalizationGeometricEstimate}
which is valid for all $y\in\mathbb{R}^{d}\setminus C\left(W_{\delta},R'\right)$.
This also implies 
\[
1+\left|y\right|\leq1+\left|y-\xi\right|+\left|\xi\right|\leq1+\left(1+\frac{4}{\delta}\right)\left|y-\xi\right|\leq C_{\delta}\cdot\left(1+\left|y-\xi\right|\right).
\]
Together, we derive 
\begin{align*}
\left(1+\left|y-\xi\right|\right)^{-\left(N+K+L\right)} & =\left(1+\left|y-\xi\right|\right)^{-\left(K+L\right)}\cdot\left(1+\left|y-\xi\right|\right)^{-N}\\
 & \leq C_{\delta}^{K+L}\cdot\left(1+\left|y\right|\right)^{-\left(K+L\right)}\cdot\left(1+\frac{\delta}{4}\left|\xi\right|\right)^{-N}\\
 & \leq C_{\delta}^{K+L}\cdot\left(\min\left\{ 1,\frac{\delta}{4}\right\} \right)^{-N}\cdot\left(1+\left|y\right|\right)^{-\left(K+L\right)}\cdot\left(1+\left|\xi\right|\right)^{-N}\\
 & =C_{\delta,N,K,L}\cdot\left(1+\left|y\right|\right)^{-\left(K+L\right)}\cdot\left(1+\left|\xi\right|\right)^{-N}.
\end{align*}
Via Leibniz' formula, we arrive at 
\begin{align*}
 & \hphantom{=\,}\left(1+\left|y\right|\right)^{K}\cdot\left|\left[\partial^{\alpha}\left(\left(1-\zeta\right)\cdot L_{\xi}\psi_{\varphi}\right)\right]\left(y\right)\right|\\
 & \leq\left(1+\left|y\right|\right)^{K}\cdot\sum_{\beta\leq\alpha}\binom{\alpha}{\beta}\cdot\left|\left(\partial^{\beta}\left(1-\zeta\right)\right)\left(y\right)\right|\cdot\left|\left(\partial^{\alpha-\beta}\psi_{\varphi}\right)\left(y-\xi\right)\right|\\
 & \leq C\cdot\left(1+\left|y\right|\right)^{K}\cdot\sum_{\beta\leq\alpha}\binom{\alpha}{\beta}\cdot\chi_{\mathbb{R}^{d}\setminus C\left(W_{\delta},R'\right)}\left(y\right)\cdot\left(1+\left|y\right|\right)^{L}\cdot\left|\psi_{\varphi}\right|_{N+K+L}\cdot\left(1+\left|y-\xi\right|\right)^{-\left(N+K+L\right)}\\
 & \leq C\cdot\left|\psi_{\varphi}\right|_{N+K+L}C_{\delta,N,K,L}\cdot\left[\sum_{\beta\leq\alpha}\binom{\alpha}{\beta}\right]\cdot\left(1+\left|\xi\right|\right)^{-N}
\end{align*}
for all $\alpha\in\mathbb{N}_{0}^{d}$ with $\left|\alpha\right|\leq K$
and suitable constants $C=C\left(\zeta,K\right)>0$ and $L=L\left(\zeta,K\right)\in\mathbb{N}$.

All in all, this establishes 
\[
\left|g_{\varphi}\left(\xi\right)\right|=\left|\widehat{u}\left(\left(1-\zeta\right)\cdot L_{\xi}\psi_{\varphi}\right)\right|\leq C_{u}\cdot\left|\left(1-\zeta\right)\cdot L_{\xi}\psi_{\varphi}\right|_{K}\leq C_{u,\delta,\zeta,\varphi,N}\cdot\left(1+\left|\xi\right|\right)^{-N}
\]
for all $N\in\mathbb{N}$ and all $\xi\in C\left(W_{\delta/2},\frac{4}{3}R'\right)$.

Now assume that $(x_{0},\xi_{0})$ is a regular directed point of
$P_{\zeta}u$. Pick $\varphi\in C_{c}^{\infty}(\mathbb{R}^{d})$ identically
one in a neighborhood of $x_{0}$, as well as a $\xi_{0}$-neighborhood
$W'\subset S^{d-1}$ and some $R''>0$ such that 
\[
\left|\left(\varphi\cdot P_{\zeta}u\right)^{\wedge}(\xi)\right|\le C_{N}\cdot(1+|\xi|)^{-N}
\]
holds for all $\xi\in C(W',R'')$.

As an easy consequence of the definitions, $u=P_{\zeta}u+P_{1-\zeta}u$
and hence 
\[
\mathcal{F}\left(\varphi\cdot u\right)=\mathcal{F}\left(\varphi\cdot P_{\zeta}u\right)+\mathcal{F}\left(\varphi\cdot P_{1-\zeta}u\right).
\]
But $\mathcal{F}\left(\varphi\cdot P_{\zeta}u\right)$ is of rapid
decay on $C\left(W',R''\right)$, whereas rapid decay of $g_{\varphi}=\mathcal{F}\left(\varphi\cdot P_{1-\zeta}u\right)$
on $C\left(W_{\delta/2},\frac{4}{3}R'\right)$ was established above.
Hence, $\mathcal{F}\left(\varphi\cdot u\right)$ decays rapidly on
\[
C\left(W'\cap W_{\delta/2},\max\left\{ R'',\frac{4}{3}R'\right\} \right),
\]
so that $\left(x_{0},\xi_{0}\right)$ is a regular directed point
of $u$.
\end{proof}
Since we aim at characterizing regular directed points using wavelet
transform decay, the following result is a natural counterpart to
Lemma \ref{lem:reg_points_loc}. 
\begin{lem}
\label{lem:wavelet_local}Let $u\in\mathcal{S}'(\mathbb{R}^{d})$,
$\psi\in\mathcal{S}(\mathbb{R}^{d})$ and $\varphi\in C_{c}^{\infty}(\mathbb{R}^{d})$,
with $\varphi|_{B_{\varepsilon_{1}}(x)}\equiv1$ for some $x\in\mathbb{R}^{d}$
and $\varepsilon_{1}>0$. Assume that the dual action is $V$-microlocally
admissible in direction $\xi$, and that ${\rm supp}(\widehat{\psi})\subset V$
for some $\emptyset\neq V\Subset\mathcal{O}$.

Then there exist constants $C_{N}>0$ ($N\in\mathbb{N}$), such that
the estimate 
\[
\left|W_{\psi}u(y,h)-\left(W_{\psi}(\varphi u)\right)(y,h)\right|\le C_{N}\|h\|^{N}
\]
holds for all $y\in B_{\varepsilon_{1}/2}(x)$ and all $h\in K_{o}(W_{0},V,R_{0})$,
as soon as $R_{0}>0$ and the $\xi$-neighborhood $W_{0}\subset S^{d-1}$
satisfy part \ref{enu:NormOfInverseEstimateOnKo} of the definition
of $V$-microlocal admissibility (Definition \ref{defn:micro_regular}). \end{lem}
\begin{proof}
We first employ the standard continuity criterion for tempered distributions
to estimate 
\begin{align*}
\left|W_{\psi}u(y,h)-\left(W_{\psi}(\varphi u)\right)(y,h)\right| & =|\langle u\mid(1-\overline{\varphi})\cdot(\pi(y,h)\psi)\rangle|\\
 & \le C\cdot|(1-\overline{\varphi})\cdot(\pi(y,h)\psi)|_{M},
\end{align*}
for a suitable $C>0$ and $M\in\mathbb{N}$ (depending only on $u\in\mathcal{S}'\left(\mathbb{R}^{d}\right)$).
We use the estimate $\left|z-y\right|=\left|hh^{-1}\left(z-y\right)\right|\leq\left\Vert h\right\Vert \cdot\left|h^{-1}\left(z-y\right)\right|$
to derive 
\[
\left(1+\left|h^{-1}\left(z-y\right)\right|\right)^{-K}\leq\left|h^{-1}\left(z-y\right)\right|^{-K}\leq\left\Vert h\right\Vert ^{K}\cdot\left|z-y\right|^{-K}.
\]
An application of Lemma \ref{lem:related_to_assumptions}\ref{enu:NormBoundedOnKoAndDeterminantEstimate}
shows 
\[
\left|\det\left(h\right)\right|^{-1/2}\leq C_{1}\cdot\left\Vert h\right\Vert ^{-\alpha_{1}d/2}
\]
and 
\[
\left(1+\left\Vert h^{-1}\right\Vert \right)^{M}\leq C_{2}\cdot\left\Vert h\right\Vert ^{-M\alpha_{1}}
\]
for all $h\in K_{o}\left(W_{0},V,R_{0}\right)$, with $C_{2}=C_{2}\left(M,W_{0},R_{0},V\right)>0$
and $\alpha_{1}>0$ as in Definition \ref{defn:micro_regular}.

Using Leibniz' formula, the chain rule and the fact that $1-\overline{\varphi}$
vanishes on $B_{\varepsilon_{1}}(x)$, we derive 
\begin{align*}
 & \hphantom{\leq\,}\left|\left(1-\overline{\varphi}\right)\cdot\left(\pi\left(y,h\right)\psi\right)\right|_{M}\\
 & \le C_{\varphi,M}\cdot\max_{\left|\alpha\right|\leq M}\sup_{z\in\mathbb{R}^{d}\setminus B_{\varepsilon_{1}}\left(x\right)}(1+|z|)^{M}\left|\partial^{\alpha}|_{z'=z}\left(z'\mapsto\left|\det\left(h\right)\right|^{-1/2}\cdot\psi(h^{-1}(z'-y))\right)\right|\\
 & \le C_{\varphi,M}'\cdot\left\Vert h\right\Vert ^{-\alpha_{1}d/2}\cdot\max_{\left|\alpha\right|\leq M}\sup_{z\in\mathbb{R}^{d}\setminus B_{\varepsilon_{1}}\left(x\right)}\left(1+\left|z\right|\right)^{M}\left(1+\left\Vert h^{-1}\right\Vert \right)^{M}\left|\left(\partial^{\alpha}\psi\right)\left(h^{-1}\left(z-y\right)\right)\right|\\
 & \le\left|\psi\right|_{K}C_{\varphi,M,W_{0},R_{0},V}\cdot\left\Vert h\right\Vert ^{-\alpha_{1}\left(M+\frac{d}{2}\right)}\cdot\sup_{z\in\mathbb{R}^{d}\setminus B_{\varepsilon_{1}}\left(x\right)}\left(1+\left|z\right|\right)^{M}\left(1+\left|h^{-1}\left(z-y\right)\right|\right)^{-K}\\
 & \le\left|\psi\right|_{K}C_{\varphi,M,W_{0},R_{0},V}\cdot\|h\|^{K-\alpha_{1}\left(M+\frac{d}{2}\right)}\cdot\sup_{z\in\mathbb{R}^{d}\setminus B_{\varepsilon_{1}}\left(x\right)}(1+|z|)^{M}|z-y|^{-K}
\end{align*}
for all $K\geq M$.

But as soon as $K\ge M$, $z\in\mathbb{R}^{d}\setminus B_{\varepsilon_{1}}\left(x\right)$
and $y\in\overline{B_{\varepsilon_{1}/2}}\left(x\right)$, we have
\[
\left|z-y\right|=\left|\left(z-x\right)+\left(x-y\right)\right|\geq\left|z-x\right|-\left|x-y\right|\geq\left|z-x\right|-\frac{\varepsilon_{1}}{2}\geq\frac{\varepsilon_{1}}{2}.
\]
There are now two cases for $z\in\mathbb{R}^{d}\setminus B_{\varepsilon_{1}}\left(x\right)$.
If $\left|z\right|\geq2\cdot\left(\left|x\right|+\frac{\varepsilon_{1}}{2}\right)\geq\varepsilon_{1}$,
then $\left|\frac{y}{\left|z\right|}\right|\leq\frac{\left|x\right|+\frac{\varepsilon_{1}}{2}}{\left|z\right|}\leq\frac{1}{2}$
and hence
\[
\frac{\left(1+\left|z\right|\right)^{M}}{\left|z-y\right|^{K}}=\frac{\left(\frac{1}{\left|z\right|}+1\right)^{M}}{\left|1-\frac{y}{\left|z\right|}\right|^{M}}\cdot\left|z-y\right|^{M-K}\leq\left(\frac{2}{\varepsilon_{1}}\right)^{K-M}\cdot\left[\frac{1+\frac{1}{\left|z\right|}}{1-\frac{1}{2}}\right]^{M}\leq\left(\frac{2}{\varepsilon_{1}}\right)\cdot2^{M}\cdot\left(1+\frac{1}{\varepsilon_{1}}\right)^{M}.
\]
If otherwise $\left|z\right|\leq2\cdot\left(\left|x\right|+\frac{\varepsilon_{1}}{2}\right)$,
we observe that
\[
\left(\left[\mathbb{R}^{d}\setminus B_{\varepsilon_{1}}\left(x\right)\right]\cap\overline{B_{2\cdot\left(\left|x\right|+\frac{\varepsilon_{1}}{2}\right)}}\left(0\right)\right)\times\overline{B_{\varepsilon_{1}/2}}\left(x\right)\to\mathbb{R},\left(z,y\right)\mapsto\left(1+\left|z\right|\right)^{M}\cdot\left|z-y\right|^{-K}
\]
is a continuous function on a compact set and hence bounded. All in
all, this shows that the constant
\[
C_{K,M,\varepsilon_{1},x}:=\sup_{y\in\overline{B_{\varepsilon_{1}/2}}\left(x\right)}\,\sup_{z\in\mathbb{R}^{d}\setminus B_{\varepsilon_{1}}\left(x\right)}\left(1+\left|z\right|\right)^{M}\cdot\left|z-y\right|^{-K}
\]
is finite. Here, we note that $x\in\mathbb{R}^{d}$ and $\varepsilon_{1}>0$
are fixed.

Since $K$ can be chosen arbitrarily large and $\left\Vert h\right\Vert $
is bounded on $K_{o}\left(W_{0},V,R_{0}\right)$ (cf. Lemma \ref{lem:related_to_assumptions}\ref{enu:NormBoundedOnKoAndDeterminantEstimate}),
so that large powers of $\left\Vert h\right\Vert $ can be estimated
by (constant multiplies of) smaller powers of $\left\Vert h\right\Vert $,
this establishes the desired decay estimate. 
\end{proof}
We can now formulate wavelet criteria for regular directed points.
Note that the following theorem can \emph{not} be understood as a
characterization in the strict sense, since the necessary condition
in part \ref{enu:RegularDirectedPointNecessaryCondition} concludes
a certain decay behaviour on $K_{i}(W,V,R)$, whereas the sufficient
condition in part \ref{enu:RegularDirectedPointSufficientCondition}
requires this behaviour on the \emph{larger} set $K_{o}(W,V,R)$.

Another important detail of the theorem is that the neighborhoods
$U,W$ and also $R>0$ in part \ref{enu:RegularDirectedPointNecessaryCondition}
can be chosen \emph{independently} of the wavelet $\psi$. This will
become important for the proof of wavelet characterizations using
multiple wavelets.
\begin{thm}
\label{thm:almost_char}Assume that the dual action is $V$-microlocally
admissible in direction $\xi$. Let $u\in\mathcal{S}'(\mathbb{R}^{d})$
and $(x,\xi)\in\mathbb{R}^{d}\times(\mathcal{O}\cap S^{d-1})$. 
\begin{enumerate}[label=(\alph*)]
\item \label{enu:RegularDirectedPointNecessaryCondition}If $(x,\xi)$
is a regular directed point of $u$, then there exists a neighborhood
$U$ of $x$, some $R>0$ and a $\xi$-neighborhood $W\subset S^{d-1}$
such that for \emph{all} admissible $\psi\in\mathcal{S}(\mathbb{R}^{d})$
with ${\rm supp}(\widehat{\psi})\subset V$, the following estimate
holds: 
\[
\forall N\in\mathbb{N}\,\exists C_{N}>0\,\forall y\in U\,\forall h\in K_{i}(W,V,R)~:~|W_{\psi}u(y,h)|\le C_{N}\|h\|^{N}.
\]
For each such $\psi$, we even have 
\[
\forall N\in\mathbb{N}\,\exists C_{N}>0\,\forall y\in U\,\forall h\in K_{i}(W,\,\widehat{\psi}^{-1}\left(\mathbb{C}\setminus\left\{ 0\right\} \right),\, R)~:~|W_{\psi}u(y,h)|\le C_{N}\cdot\left\Vert h\right\Vert ^{N}.
\]

\item \label{enu:RegularDirectedPointSufficientCondition}Let $\psi\in\mathcal{S}(\mathbb{R}^{d})$
be admissible with ${\rm supp}(\widehat{\psi})\subset V$. Assume
that $U$ is a neighborhood of $x$ and that there are $R>0$ and
a $\xi$-neighborhood $W\subset S^{d-1}$ such that 
\[
\forall N\in\mathbb{N}\,\exists C_{N}>0\,\forall y\in U\,\forall h\in K_{o}(W,V,R)~:~|W_{\psi}u(y,h)|\le C_{N}\cdot\left\Vert h\right\Vert ^{N}.
\]
Then $(x,\xi)$ is a regular directed point of $u$. 
\end{enumerate}
\end{thm}
\begin{proof}
In the remainder of the proof, we will need the formula 
\begin{equation}
\left(W_{\psi}u(\cdot,h)\right)^{\wedge}(\xi')=\widehat{u}(\xi')\cdot\left|\det\left(h\right)\right|^{1/2}\cdot\overline{\widehat{\psi}(h^{T}\xi')}\qquad\forall\xi'\in\mathbb{R}^{d}\label{eq:WaveletTrafoFourierTrafoForCompactlySupportedDistributions}
\end{equation}
for the Fourier transform of the Wavelet transform, which follows
by the convolution theorem. For later reference, we provide a direct
calculation valid for compactly supported $u$: Recall from \cite[Theorem 7.23]{RudinFA}
that the tempered distribution $\widehat{u}$ is given by integration
against a smooth, polynomially bounded function (again denoted by
$\widehat{u}$). Using the definition of the Fourier transform for
tempered distributions, we can hence write 
\begin{align}
\left(W_{\psi}u\right)\left(y,h\right) & =\left\langle u\mid\pi\left(y,h\right)\psi\right\rangle =\left\langle u,\overline{\pi\left(y,h\right)\psi}\right\rangle \nonumber \\
 & =\left\langle \widehat{u},\mathcal{F}^{-1}\left(\overline{\pi\left(y,h\right)\psi}\right)\right\rangle \nonumber \\
 & \overset{\left(\dagger\right)}{=}\left|\det\left(h\right)\right|^{1/2}\cdot\int_{\mathbb{R}^{d}}\widehat{u}\left(\xi\right)\cdot\overline{\widehat{\psi}\left(h^{T}\xi\right)}\cdot e^{2\pi i\left\langle y,\xi\right\rangle }\,{\rm d}\xi\label{eq:WaveletTransformUsingFourierTransforms}\\
 & =\left|\det\left(h\right)\right|^{1/2}\cdot\left[\mathcal{F}^{-1}\left(\widehat{u}\cdot\overline{\widehat{\psi}\left(h^{T}\cdot\right)}\right)\right]\left(y\right),\nonumber 
\end{align}
where we used equation \eqref{eq:QuasiRegularOnFourierSide} together
with $\mathcal{F}^{-1}\overline{f}=\overline{\widehat{f}}$ at $\left(\dagger\right)$.
But $\widehat{u}$ has polynomially bounded derivatives and $\widehat{\psi}\left(h^{T}\cdot\right)$
is a Schwartz function. Together, this entails $\widehat{u}\cdot\overline{\widehat{\psi}\left(h^{T}\cdot\right)}\in\mathcal{S}\left(\mathbb{R}^{d}\right)$,
so that Fourier inversion finally yields equation \eqref{eq:WaveletTrafoFourierTrafoForCompactlySupportedDistributions}.
The analogous formula (with a similar, but easier proof) also holds
for $W_{\psi}\varphi$, for any Schwartz function $\varphi$.

Let us now prove part \ref{enu:RegularDirectedPointNecessaryCondition}.
To this end, assume that $(x,\xi)$ is a regular directed point of
$u$. Fix $\varphi\in C_{c}^{\infty}(\mathbb{R}^{d})$ satisfying
$\varphi\equiv1$ on $B_{\varepsilon_{1}}(x)$ for some $\varepsilon_{1}>0$
as well as 
\begin{equation}
\left|\widehat{\varphi u}\left(\xi'\right)\right|\le C_{N}\cdot\left(1+\left|\xi'\right|\right)^{-N}\label{eq:RegularDirectedPointNecessaryConditionPrerequisite}
\end{equation}
for all $\xi'\in C(W_{1})$ and $N\in\mathbb{N}$, where $W_{1}\subset S^{d-1}$
is a $\xi$-neighborhood.

Let $R_{0}>0$ and the $\xi$-neighborhood $W_{0}\subset S^{d-1}$
be provided by the assumption of $V$-microlocal admissibility in
direction $\xi$ (cf. Definition \ref{defn:micro_regular}). Furthermore,
set $U:=B_{\varepsilon_{1}/2}(x)$. Observe that all choices up to
this point only depend on $u,x,\xi$ and $V$, but not on the wavelet
$\psi$.

Set $V':=\widehat{\psi}^{-1}\left(\mathbb{C}\setminus\left\{ 0\right\} \right)\subset V$
and observe (cf. Remark \ref{rem:incl_prop}) the chain of inclusions
\begin{align}
K_{i}\left(W_{0}\cap W_{1},V,R_{0}\right) & \subset K_{i}\left(W_{0}\cap W_{1},V',R_{0}\right)\subset K_{i}\left(W_{0},V',R_{0}\right)\nonumber \\
 & \subset K_{o}\left(W_{0},V',R_{0}\right)\subset K_{o}\left(W_{0},V,R_{0}\right).\label{eq:NecessaryConditionInclusionChain}
\end{align}
By Lemma \ref{lem:wavelet_local}, it is sufficient to show 
\[
|\left(W_{\psi}(\varphi u)\right)(y,h)|\le C_{N}\|h\|^{N}
\]
for all $y\in U$ and $h\in K_{i}(W_{0}\cap W_{1},V',R_{0})$. This
will entail rapid decay of $W_{\psi}u$ on the set $U\times K_{i}\left(W_{0}\cap W_{1},V',R_{0}\right)\supset U\times K_{i}\left(W_{0}\cap W_{1},V,R_{0}\right)$.

For such $h$, we use equation \eqref{eq:WaveletTransformUsingFourierTransforms}
with $\varphi u$ instead of $u$ -- together with the fact that $\varphi u$
is compactly supported -- to estimate 
\[
\left|\left(W_{\psi}(\varphi u)\right)(y,h)\right|\le\left|\det\left(h\right)\right|^{1/2}\cdot\int_{\mathbb{R}^{d}}\left|(\varphi u)^{\wedge}(\xi')\right|\cdot\left|\widehat{\psi}(h^{T}\xi')\right|\,{\rm d}\xi'.
\]
We observe that the definition of $K_{i}\left(W_{0}\cap W_{1},V',R_{0}\right)$
implies that the set $h^{-T}V'$ on which $\widehat{\psi}(h^{T}\cdot)$
does not vanish is contained in $C(W_{0}\cap W_{1})\subset C\left(W_{1}\right)$
for each $h\in K_{i}\left(W_{0}\cap W_{1},V',R_{0}\right)$, so that
equation \eqref{eq:RegularDirectedPointNecessaryConditionPrerequisite}
yields $|\widehat{\varphi u}(\xi')|\le C_{N}\cdot\left(1+\left|\xi'\right|\right)^{-N}$
on this set.

But $\widehat{\psi}\left(h^{T}\xi'\right)\neq0$ also implies $h^{T}\xi'\in V'\subset V$,
so that Lemma \ref{lem:related_to_assumptions}\ref{enu:XiEstimatedByH}
yields $C=C\left(V\right)>0$ with 
\[
\left|\xi'\right|^{-1}=\left|h^{-T}h^{T}\xi'\right|^{-1}\leq C\cdot\left\Vert h\right\Vert ,
\]
which leads to 
\[
\left|\widehat{\varphi u}\left(\xi'\right)\right|\le C_{N}\cdot\left(1+\left|\xi'\right|\right)^{-N}\leq C_{N}\cdot\left|\xi'\right|^{-N}\leq C_{N}'\cdot\|h\|^{N}.
\]
In summary, we obtain 
\begin{align*}
\left|\left(W_{\psi}(\varphi u)\right)(y,h)\right| & \le C_{N}'\cdot\|h\|^{N}\cdot\left|\det\left(h\right)\right|^{1/2}\int_{\mathbb{R}^{d}}|\widehat{\psi}(h^{T}\xi')|\,{\rm d}\xi'\\
 & =C_{N}'\cdot\|h\|^{N}\cdot\left|\det\left(h\right)\right|^{-1/2}\|\widehat{\psi}\|_{1}.\\
 & \leq\|\widehat{\psi}\|_{1}\cdot C_{N}''\cdot\left\Vert h\right\Vert ^{N-\frac{\alpha_{1}d}{2}},
\end{align*}
wher the last estimate is due to Lemma \ref{lem:related_to_assumptions}\ref{enu:NormBoundedOnKoAndDeterminantEstimate}
and to the chain of inclusions in equation \eqref{eq:NecessaryConditionInclusionChain}.

The same lemma also yields that $\left\Vert h\right\Vert $ is bounded
on $K_{o}\left(W_{0},V,R_{0}\right)\supset K_{i}\left(W_{0}\cap W_{1},V',R_{0}\right)$,
so that large powers of $\left\Vert h\right\Vert $ can be estimated
by smaller powers. Hence, the above inequality implies the desired
decay estimate, because $N\in\mathbb{N}$ is arbitrary.

For the proof of part \ref{enu:RegularDirectedPointSufficientCondition},
observe that we may assume $u$ to be compactly supported by Lemma
\ref{lem:reg_points_loc} and Lemma \ref{lem:wavelet_local}. By assumption,
\begin{equation}
|W_{\psi}u(y,h)|\le C_{N}\|h\|^{N}\label{eq:SufficientConditionRepeated}
\end{equation}
holds for all $N\in\mathbb{N}$, $y\in B_{\varepsilon_{1}}(x)$ and
$h\in K_{o}(W,V,R)$. With $R_{0}>0$ and $W_{0}\subset S^{d-1}\cap\mathcal{O}$
as in Definition \ref{defn:micro_regular}, we may assume $W\subset W_{0}\subset\mathcal{O}$
and $R>R_{0}$. In particular, this implies $C\left(W,R\right)\subset\mathcal{O}$,
thanks to Assumption \ref{assume:proper_dual}\ref{enu:DualActionContainsRays}.

We let 
\begin{equation}
\eta:=\int_{\mathbb{R}^{d}}\int_{K_{o}(W,V,R)}W_{\psi}u(y,h)\cdot\pi(y,h)\psi\,\frac{{\rm d}h}{|{\rm det}(h)|}\,{\rm d}y,\label{eqn:def_eta}
\end{equation}
as well as 
\begin{equation}
\zeta:\mathbb{R}^{d}\to\left[0,\infty\right),\xi'\mapsto\int_{K_{o}(W,V,R)}\left|\widehat{\psi}(h^{T}\xi')\right|^{2}{\rm d}h.\label{eqn:def_zeta}
\end{equation}
Here the integral defining $\eta$ is to be understood in the weak
sense, i.e., given $\varphi\in\mathcal{S}(\mathbb{R}^{d})$, we let
\begin{align}
\langle\eta\mid\varphi\rangle & =\int_{\mathbb{R}^{d}}\int_{K_{o}(W,V,R)}W_{\psi}u(y,h)\cdot\langle\pi(y,h)\psi\mid\varphi\rangle_{L^{2}}\frac{{\rm d}h}{|{\rm det}(h)|}\,{\rm d}y\nonumber \\
 & =\int_{\mathbb{R}^{d}}\int_{K_{o}\left(W,V,R\right)}W_{\psi}u\left(y,h\right)\cdot\overline{\left(W_{\psi}\varphi\right)\left(y,h\right)}\,\frac{{\rm d}h}{|{\rm det}(h)|}\,{\rm d}y.\label{eqn:weak_def_eta}
\end{align}
We now intend to show the following statements: 
\begin{enumerate}
\item \label{enu:AlmostCharacterizationEtaWellDefined}$\eta$ is a well-defined
element of $\mathcal{S}'(\mathbb{R}^{d})$, and the integral in equation
\eqref{eqn:weak_def_eta} converges absolutely for all $\varphi\in\mathcal{S}\left(\mathbb{R}^{d}\right)$, 
\item \label{enu:AlmostCharacterizationZetaSmooth}$\zeta$ is a $C^{\infty}$-function
with polynomially bounded derivatives, 
\item \label{enu:AlmostCharacterizationEtaIsFourierLocalizationWithZeta}$\eta=P_{\zeta}u$,
and $\zeta|_{C(W,R)}\equiv1$, 
\item \label{enu:AlmostCharacterizationEtaIsSmooth}in a suitable neighborhood
of $x$, the distribution $\eta$ is given by a $C^{\infty}$-function. 
\end{enumerate}
It is worth noting that only the last part will use the assumption
regarding the decay of $W_{\psi}u$ (cf. equation \eqref{eq:SufficientConditionRepeated}).

These observations combined will yield that $(x,\xi)$ is a regular
directed point of $u$: By observation \eqref{enu:AlmostCharacterizationEtaIsSmooth},
$\left(x,\xi\right)$ is a regular directed point of $\eta$, and
then the observations \eqref{enu:AlmostCharacterizationZetaSmooth}
and \eqref{enu:AlmostCharacterizationEtaIsFourierLocalizationWithZeta},
together with Lemma \ref{lem:Four_loc}, show that $\left(x,\xi\right)$
is also a regular directed point of $u$.

Let us now provide the details. For the well-definedness of $\eta$,
we use parts \ref{enu:WaveletTrafoGeneralEstimateForDistributions}
and \ref{enu:WaveletTrafoEstimateForSchwartzFunctionsOnTheCone} of
Lemma \ref{lem:wc_decay_general} to estimate, with a suitable $N\in\mathbb{N}$
depending on $u$, and arbitrary $M,K\in\mathbb{N}$, $x\in\mathbb{R}^{d}$
and $h\in K_{o}\left(W,V,R\right)\subset K_{o}\left(W_{0},V,R_{0}\right)$:
\begin{align*}
|W_{\psi}u(y,h)| & \le C_{u,\psi}\cdot\left|\det\left(h\right)\right|^{-1/2}\cdot(1+\|h^{-1}\|)^{N}\cdot\max\left\{ 1,\smash{\left\Vert h\right\Vert ^{N}}\right\} \cdot(1+|y|)^{N},\\
\left|W_{\psi}\varphi\left(y,h\right)\right| & \leq C_{M,K,\psi,W_{0},V,R_{0}}\left|\varphi\right|_{M+K}\cdot\left|\det\left(h\right)\right|^{-1/2}\cdot\left(1+\left|y\right|\right)^{-K}\left\Vert h\right\Vert ^{M}.
\end{align*}
Recall (from Lemma \ref{lem:related_to_assumptions}\ref{enu:NormBoundedOnKoAndDeterminantEstimate})
that the norm $\left\Vert h\right\Vert $ is bounded on $K_{o}(W_{0},V,R_{0})\supset K_{o}\left(W,V,R\right)$,
so that the factor $\max\left\{ 1,\smash{\left\Vert h\right\Vert ^{N}}\right\} $
can be bounded by a constant while estimating the following integral
(which is nothing but the absolute version of the integral defining
$\left\langle \eta\mid\varphi\right\rangle $): 
\begin{align*}
 & \hphantom{=\:}\int_{\mathbb{R}^{d}}\int_{K_{o}(W,V,R)}\left|W_{\psi}u(y,h)\cdot\overline{W_{\psi}\varphi(y,h)}\right|\frac{{\rm d}h}{|\det(h)|}\,{\rm d}y\\
 & \leq C'\cdot\left|\varphi\right|_{M+K}\cdot\int_{\mathbb{R}^{d}}\left(1+\left|y\right|\right)^{N-K}\,{\rm d}y\cdot\int_{K_{o}(W,V,R)}\frac{\left(1+\left\Vert h^{-1}\right\Vert \right)^{N}}{\left|\det\left(h\right)\right|^{2}}\cdot\left\Vert h\right\Vert ^{M}\,{\rm d}h
\end{align*}
for a suitable constant $C'=C'\left(M,K,\psi,W_{0},V,R_{0},u\right)$
and $N=N\left(u\right)$.

The first integral is finite as soon as $K\ge N+d+1$. For the second
integral, we can use Lemma \ref{lem:related_to_assumptions}\ref{enu:NormBoundedOnKoAndDeterminantEstimate}
to estimate the integrand by 
\[
\frac{(1+\|h^{-1}\|)^{N}}{|{\rm det}(h)|^{2}}\|h\|^{M}\le C''\cdot\left\Vert h\right\Vert ^{-\alpha_{1}N}\cdot\left\Vert h\right\Vert ^{-2d\alpha_{1}}\cdot\|h\|^{M}=C''\|h\|^{M-\alpha_{1}(N+2d)}.
\]
But $\left\Vert h\right\Vert $ is bounded on $K_{o}$ by Lemma \ref{lem:related_to_assumptions}\ref{enu:NormBoundedOnKoAndDeterminantEstimate},
so that large powers of $\left\Vert h\right\Vert $ can be estimated
by small powers. Hence, part \ref{enu:NormIntegrability} of Definition
\ref{defn:micro_regular} implies that the integral is finite, as
soon as 
\[
M-\alpha_{1}(N+2d)\ge\alpha_{2}.
\]
Thus well-definedness of $\eta\in\mathcal{S}'(\mathbb{R}^{d})$ (and
absolute convergence of the integral in equation \eqref{eqn:weak_def_eta})
is established.

For the proof of property \eqref{enu:AlmostCharacterizationZetaSmooth},
we note that $H\leq{\rm GL}\left(\mathbb{R}^{d}\right)$ is a closed
subgroup and hence $\sigma$-compact, i.e. $H=\bigcup_{\ell\in\mathbb{N}}K_{\ell}$
with $K_{\ell}\subset K_{\ell+1}$ and compact $K_{\ell}$ for all
$\ell$. Define 
\[
\zeta_{\ell}:\mathbb{R}^{d}\to\left[0,\infty\right),\xi'\mapsto\int_{K_{\ell}\cap K_{o}(W,V,R)}\left|\widehat{\psi}(h^{T}\xi')\right|^{2}{\rm d}h
\]
for each $\ell\in\mathbb{N}$. Smoothness of $\left|\smash{\widehat{\psi}}\right|^{2}$
together with compactness of $K_{\ell}$ easily imply that differentiation
and integration can be interchanged in the definition of $\zeta_{\ell}$,
so that each $\zeta_{\ell}$ is a smooth function. Moreover, monotone
convergence implies $\zeta_{\ell}\left(\xi'\right)\to\zeta\left(\xi'\right)$
for all $\xi'\in\mathbb{R}^{d}$.

We first observe $\zeta_{\ell}\equiv0$ on $\mathcal{O}^{c}$, because
$\zeta_{\ell}\left(\xi'\right)\neq0$ would imply $\widehat{\psi}\left(h^{T}\xi'\right)\neq0$
for some $h\in H$ and hence $h^{T}\xi'\in{\rm supp}\left(\smash{\widehat{\psi}}\right)\subset\mathcal{O}$,
which entails $\xi'\in h^{-T}\mathcal{O}\subset\mathcal{O}$, because
$\mathcal{O}$ is $H^{T}$-invariant. We will now show 
\begin{equation}
\left|\left(\smash{\partial^{\beta}\zeta_{\ell}}\right)\left(\xi'\right)\right|\leq C_{\beta}\cdot\left(1+\left|\xi'\right|\right)^{\alpha}\label{eq:AlmostCharacterizationZetaApproximationsHaveBoundedDerivatives}
\end{equation}
for all $\beta\in\mathbb{N}_{0}^{d}$ and $\xi'\in\mathcal{O}$ (and
hence for $\xi'\in\overline{\mathcal{O}}$ by continuity of $\partial^{\beta}\zeta_{\ell}$),
where the constants $C_{\beta}>0$ are independent of $\ell\in\mathbb{N}$
and of $\xi'\in\mathcal{O}$ and where $\alpha\geq0$ is taken from
Assumption \ref{assume:proper_dual}\ref{enu:PolynomialGrowthOfMeasureOfIntersection}.
This implies that estimate \eqref{eq:AlmostCharacterizationZetaApproximationsHaveBoundedDerivatives}
is even valid on all of $\mathbb{R}^{d}$ because of $\zeta_{\ell}\equiv0$
on the open set $\overline{\mathcal{O}}^{c}\subset\mathcal{O}^{c}$
(see above).

Local boundedness of the higher derivatives then implies (local) equicontinuity
of $\left(\partial^{\beta}\zeta_{\ell}\right)_{\ell}$ for all $\beta$.
Thus, an Arzela-Ascoli argument implies locally uniform convergence
$\partial^{\beta}\zeta_{\ell}\to\zeta_{\beta}$ (along some subsequence%
\footnote{Using $\zeta_{\beta}=\partial^{\beta}\zeta$, one can show that the
convergence $\partial^{\beta}\zeta_{\ell}\to\partial^{\beta}\zeta$
even holds without restricting to a subsequence.%
}) with continuous functions $\zeta_{\beta}:\mathbb{R}^{d}\to\mathbb{R}$.
The pointwise convergence $\zeta_{\ell}\to\zeta$ implies $\zeta_{0}=\zeta$,
which then entails that $\zeta$ is smooth with $\partial^{\beta}\zeta=\zeta_{\beta}$
for all $\beta\in\mathbb{N}_{0}^{d}$. Finally, we get 
\[
\left|\left(\smash{\partial^{\beta}\zeta}\right)\left(\xi'\right)\right|=\lim_{\ell\to\infty}\left|\left(\smash{\partial^{\beta}\zeta_{\ell}}\right)\left(\xi'\right)\right|\leq C_{\beta}\cdot\left(1+\left|\xi'\right|\right)^{\alpha}
\]
for all $\xi'\in\overline{\mathcal{O}}$. Together with $\zeta\equiv0$
on $\overline{\mathcal{O}}^{c}\subset\mathcal{O}^{c}$, we see that
all derivatives of $\zeta$ are polynomially bounded.

In order to prove the estimate \eqref{eq:AlmostCharacterizationZetaApproximationsHaveBoundedDerivatives},
we first note that in the evaluation of $\zeta_{\ell}(\xi')$, the
domain of integration can be reduced to $K_{o}(W,V,R)\cap K_{\ell}\cap H_{\xi',V}$,
using ${\rm supp}(\widehat{\psi})\subset V$, and the definition of
$H_{\xi',V}$ (cf. Assumption \ref{assume:proper_dual}). Thus, 
\[
\left|(\partial^{\beta}\zeta_{\ell})(\xi')\right|\le\int_{K_{o}(W,V,R)\cap H_{\xi',V}}\left|\left[\partial^{\beta}\left(\left|\smash{\widehat{\psi}}\left(h^{T}\cdot\right)\right|^{2}\right)\right]\left(\xi'\right)\right|\,{\rm d}h.
\]
Using the chain rule, we see that the integrand can be estimated by
\[
\left|\left[\partial^{\beta}\left(\left|\smash{\widehat{\psi}}\left(h^{T}\cdot\right)\right|^{2}\right)\right]\left(\xi'\right)\right|\le C_{\left|\beta\right|}\cdot(1+\|h\|)^{|\beta|}\cdot\max_{\left|\sigma\right|\leq\left|\beta\right|}\left\Vert \partial^{\sigma}|\widehat{\psi}|^{2}\right\Vert _{\sup}.
\]
In particular, since the norm $\left\Vert h\right\Vert $ is bounded
on $K_{o}\left(W_{0},V,R\right)\supset K_{o}(W,V,R)$ (cf. Lemma \ref{lem:related_to_assumptions}\ref{enu:NormBoundedOnKoAndDeterminantEstimate}),
we obtain the bound 
\[
\left|(\partial^{\beta}\zeta_{\ell})(\xi')\right|\le C_{\left|\beta\right|}'\cdot\mu_{H}(H_{\xi',V}),
\]
where $C_{\left|\beta\right|}'>0$ is independent of $\ell\in\mathbb{N}$
and $\xi'\in\mathcal{O}$. But $\mu_{H}(H_{\xi',V})\leq C\cdot\left(1+\left|\xi'\right|\right)^{\alpha}$
by assumption \ref{assume:proper_dual}\ref{enu:PolynomialGrowthOfMeasureOfIntersection}.
This proves equation \eqref{eq:AlmostCharacterizationZetaApproximationsHaveBoundedDerivatives}
and thus also observation \eqref{enu:AlmostCharacterizationZetaSmooth}.

For observation \eqref{enu:AlmostCharacterizationEtaIsFourierLocalizationWithZeta},
first recall that we are assuming $u$ to be of compact support. Hence,
equation \eqref{eq:WaveletTrafoFourierTrafoForCompactlySupportedDistributions}
yields 
\[
\left(W_{\psi}u(\cdot,h)\right)^{\wedge}(\xi')=\widehat{u}(\xi')\cdot\left|\det\left(h\right)\right|^{1/2}\cdot\overline{\widehat{\psi}(h^{T}\xi')}\qquad\forall\xi'\in\mathbb{R}^{d}.
\]
Using Fubini's theorem in equation \eqref{eqn:weak_def_eta} and then
Plancherel's theorem on the inner integral, we conclude 
\begin{align*}
\left\langle \eta,\widehat{\varphi}\right\rangle =\left\langle \eta\mid\overline{\widehat{\varphi}}\right\rangle =\left\langle \eta\mid\mathcal{F}^{-1}\overline{\varphi}\right\rangle  & =\int_{K_{o}\left(W,V,R\right)}\int_{\mathbb{R}^{d}}W_{\psi}u\left(y,h\right)\cdot\overline{\left[W_{\psi}\left(\mathcal{F}^{-1}\overline{\varphi}\right)\right]\left(y,h\right)}\,{\rm d}y\,\frac{{\rm d}h}{|{\rm det}(h)|}\\
 & =\int_{K_{o}\left(W,V,R\right)}\int_{\mathbb{R}^{d}}\widehat{u}\left(\xi'\right)\cdot\overline{\widehat{\psi}\left(h^{T}\xi'\right)}\cdot\overline{\overline{\varphi}\left(\xi'\right)\cdot\overline{\widehat{\psi}\left(h^{T}\xi'\right)}}\,{\rm d}\xi'\,{\rm d}h\\
 & =\int_{\mathbb{R}^{d}}\widehat{u}\left(\xi'\right)\varphi\left(\xi'\right)\cdot\int_{K_{o}\left(W,V,R\right)}\left|\smash{\widehat{\psi}}\left(h^{T}\xi'\right)\right|^{2}\,{\rm d}h\,{\rm d}\xi'\\
 & =\left\langle \widehat{u},\zeta\cdot\varphi\right\rangle =\left\langle \widehat{u},\zeta\cdot\mathcal{F}^{-1}\widehat{\varphi}\right\rangle \\
 & =\left(P_{\zeta}u\right)\left(\widehat{\varphi}\right)=\left\langle P_{\zeta}u,\widehat{\varphi}\right\rangle ,
\end{align*}
where we used the definition of $P_{\zeta}u$ given in Lemma \ref{lem:Four_loc}
in the last line. Note that the application of Fubini's theorem is
justified by the absolute convergence of the integrals proved in observation
\eqref{enu:AlmostCharacterizationEtaWellDefined}. As this holds for
all $\varphi\in\mathcal{S}\left(\mathbb{R}^{d}\right)$, we conclude
$\eta=P_{\zeta}u$.

To conclude the proof of observation \eqref{enu:AlmostCharacterizationEtaIsFourierLocalizationWithZeta},
let $\xi'\in C\left(W,R\right)\subset\mathcal{O}$ be arbitrary. For
any $h\in H$ with $\widehat{\psi}\left(h^{T}\xi'\right)\neq0$, we
have $h^{T}\xi'\in{\rm supp}\left(\smash{\widehat{\psi}}\right)\subset V$
and hence $\xi'\in h^{-T}V\cap C\left(W,R\right)$, which means $h\in K_{o}\left(W,V,R\right)$.
Using the admissibility of $\psi$, we arrive at 
\[
\zeta\left(\xi'\right)=\int_{K_{o}\left(W,V,R\right)}\left|\widehat{\psi}\left(h^{T}\xi'\right)\right|^{2}\,{\rm d}h=\int_{H}\left|\widehat{\psi}\left(h^{T}\xi'\right)\right|^{2}\,{\rm d}h=1.
\]

Finally, for the proof of statement \eqref{enu:AlmostCharacterizationEtaIsSmooth},
we define the auxiliary function 
\[
\kappa:B_{\varepsilon_{1}/2}(x)\to\mathbb{C},\, z\mapsto\int_{\mathbb{R}^{d}}\int_{K_{o}(W,V,R)}W_{\psi}u(y,h)\cdot\left(\pi(y,h)\psi\right)(z)\,\frac{{\rm d}h}{\left|\det\left(h\right)\right|}\,{\rm d}y.
\]
In other words, $\kappa$ is obtained by pointwise evaluation of the
integral defining $\eta$ weakly in equation \eqref{eqn:def_eta}.
Let us first prove that $\kappa$ is well-defined and smooth on $B_{\varepsilon_{1}/2}(x)$.
To this end, we set 
\[
\kappa_{1}(z):=\int_{B_{\varepsilon_{1}}(x)}\int_{K_{o}(W,V,R)}W_{\psi}u(y,h)\cdot\left(\pi(y,h)\psi\right)(z)\,\frac{{\rm d}h}{\left|\det\left(h\right)\right|}\,{\rm d}y
\]
and $\kappa_{2}:=\kappa-\kappa_{1}$.

We want to show that integration and partial differentiation (with
respect to $z$) are interchangeable. For $\beta\in\mathbb{N}_{0}^{d}$,
we use the chain rule and Lemma \ref{lem:related_to_assumptions}\ref{enu:NormBoundedOnKoAndDeterminantEstimate},
as well as Hadamard's inequality $\left|\det\left(g\right)\right|\leq\left\Vert g\right\Vert ^{d}$
to estimate 
\begin{align*}
 & \phantom{=\,}\left|\partial^{\beta}|_{z'=z}\left(z'\mapsto\left|\det\left(h\right)\right|^{-1}\cdot\left[\pi\left(y,h\right)\psi\right]\left(z'\right)\right)\right|\\
 & =\left|\det\left(h\right)\right|^{-3/2}\cdot\left|\partial^{\beta}|_{z'=z}\left(z'\mapsto\psi\left(h^{-1}\left(z'-y\right)\right)\right)\right|\\
 & \leq C\left(\beta\right)\cdot\left\Vert h\right\Vert ^{-\frac{3}{2}\alpha_{1}d}\cdot\left(1+\left\Vert h^{-1}\right\Vert \right)^{\left|\beta\right|}\cdot\max_{\left|\sigma\right|\leq\left|\beta\right|}\left|\left(\smash{\partial^{\sigma}\psi}\right)\left(h^{-1}\left(z-y\right)\right)\right|\\
 & \leq C'\left(\beta\right)\cdot\left|\psi\right|_{\left|\beta\right|+N}\cdot\left\Vert h\right\Vert ^{-\alpha_{1}\left(\left|\beta\right|+\frac{3}{2}d\right)}\cdot\left(1+\left|h^{-1}\left(z-y\right)\right|\right)^{-N}.
\end{align*}
Thus, the partial derivatives of the integrand in the definition of
$\kappa$ can be bounded by 
\begin{align}
 & \hphantom{\leq\,}\left|\partial^{\beta}|_{z'=z}\left(z'\mapsto W_{\psi}u\left(y,h\right)\cdot\left|\det\left(h\right)\right|^{-1}\cdot\left[\pi\left(y,h\right)\psi\right]\left(z'\right)\right)\right|\nonumber \\
 & \leq\left|\psi\right|_{\left|\beta\right|+N}C'\left(\beta\right)\cdot\left|W_{\psi}u\left(y,h\right)\right|\cdot\left\Vert h\right\Vert ^{-\alpha_{1}\left(\left|\beta\right|+\frac{3}{2}d\right)}\cdot\left(1+\left|h^{-1}\left(z-y\right)\right|\right)^{-N}.\label{eqn:est_partial_z}
\end{align}
Furthermore, for the treatment of $\kappa_{1}$ we can additionally
employ the assumption concerning the decay of $W_{\psi}u$ (cf. equation
\eqref{eq:SufficientConditionRepeated}), namely 
\[
|W_{\psi}u(y,h)|\le C_{M}\|h\|^{M}
\]
for all $M\in\mathbb{N}$, and all $y\in B_{\varepsilon_{1}}(x)$
and $h\in K_{o}(W,V,R)$. Together with the trivial bound $\left(1+\left|h^{-1}\left(z-y\right)\right|\right)^{-N}\leq1$,
we arrive at 
\begin{align}
 & \phantom{\leq\,}\left|\partial^{\beta}|_{z'=z}\left(z'\mapsto W_{\psi}u\left(y,h\right)\cdot\left|\det\left(h\right)\right|^{-1}\cdot\left[\pi\left(y,h\right)\psi\right]\left(z'\right)\right)\right|\nonumber \\
 & \leq\left|\psi\right|_{\left|\beta\right|+N}C_{M}C'\left(\beta\right)\cdot\left\Vert h\right\Vert ^{M-\alpha_{1}\left(\left|\beta\right|+\frac{3}{2}d\right)}.\label{eq:EtaPartialDerivativeEstimateOnKo}
\end{align}
Recall from Definition \ref{defn:micro_regular}\ref{enu:NormIntegrability}
that 
\[
\int_{K_{o}\left(W,V,R\right)}\left\Vert h\right\Vert ^{\gamma}\,{\rm d}h\leq\int_{K_{o}\left(W_{0},V,R_{0}\right)}\left\Vert h\right\Vert ^{\gamma}\,{\rm d}h
\]
is finite for $\gamma=\alpha_{2}$. Using the boundedness of $\left\Vert h\right\Vert $
on $K_{o}\left(W_{0},V,R_{0}\right)$ (cf. Lemma \ref{lem:related_to_assumptions}\ref{enu:NormBoundedOnKoAndDeterminantEstimate}),
we see that $\left\Vert h\right\Vert ^{\gamma}$ can be estimated
by (a constant multiple of) $\left\Vert h\right\Vert ^{\delta}$ for
$\gamma\geq\delta$. Together, we see that the right-hand side of
equation \eqref{eq:EtaPartialDerivativeEstimateOnKo} is independent
of $z\in B_{\varepsilon_{1}/2}\left(x\right)$ and of $y\in B_{\varepsilon_{1}}\left(x\right)$
and integrable over $\left(y,h\right)\in B_{\varepsilon_{1}}\left(x\right)\times K_{o}\left(W,V,R\right)$,
as soon as 
\[
M-\alpha_{1}\left(\left|\beta\right|+\frac{3}{2}d\right)\geq\alpha_{2}.
\]
But $M\in\mathbb{N}$ can be chosen arbitrarily, so that $\kappa_{1}$
is well-defined and smooth with absolute convergence of the integral.

Finally, in order to prove smoothness of $\kappa_{2}$, we first employ
Lemma \ref{lem:wc_decay_general}\ref{enu:WaveletTrafoGeneralEstimateForDistributions}
together with Lemma \ref{lem:related_to_assumptions}\ref{enu:NormBoundedOnKoAndDeterminantEstimate}
to obtain 
\begin{align*}
|W_{\psi}u(y,h)| & \le C\cdot\left|\det(h)\right|^{-1/2}(1+\|h^{-1}\|)^{M}\cdot\max\left\{ 1,\|h\|^{M}\right\} \cdot(1+\left|y\right|)^{M}\\
 & \leq C'\left(u\right)\cdot\left\Vert h\right\Vert ^{-\alpha_{1}d/2}\cdot\left\Vert h\right\Vert ^{-M\alpha_{1}}\cdot\left(1+\left|y\right|\right)^{M}
\end{align*}
for suitable $M=M\left(u\right)\in\mathbb{N}$ and $C=C\left(u,\psi\right)>0$.

On the other hand, $\left|z-y\right|=\left|hh^{-1}\left(z-y\right)\right|\leq\left\Vert h\right\Vert \cdot\left(1+\left|h^{-1}\left(z-y\right)\right|\right)$,
which implies 
\[
\left(1+\left|h^{-1}\left(z-y\right)\right|\right)^{-N}\le\|h\|^{N}\cdot\left|z-y\right|^{-N}.
\]
Note that $z\in B_{\varepsilon_{1}/2}(x)$, whereas in the definition
of $\kappa_{2}$, we integrate over $y\in\mathbb{R}^{d}\setminus B_{\varepsilon_{1}}(x)$.
This entails for all relevant $y$ the estimate 
\[
\|h\|^{N}\cdot|z-y|^{-N}\le\|h\|^{N}\cdot\left(\left|y-x\right|-\varepsilon_{1}/2\right)^{-N}.
\]
We substitute these estimates into equation \eqref{eqn:est_partial_z}
and recall that the norm $\left\Vert h\right\Vert $ is bounded on
$K_{o}(W,V,R)$ (cf. Lemma \ref{lem:related_to_assumptions}\ref{enu:NormBoundedOnKoAndDeterminantEstimate}),
to obtain the following inequality, uniform with respect to $z\in B_{\varepsilon_{1}/2}(x)$:
\begin{align}
 & \hphantom{\leq\,}\left|\partial^{\beta}|_{z'=z}\left(z'\mapsto W_{\psi}u\left(y,h\right)\cdot\left|\det\left(h\right)\right|^{-1}\cdot\left[\pi\left(y,h\right)\psi\right]\left(z'\right)\right)\right|\nonumber \\
 & \leq C''\left(u,\beta\right)\cdot\left|\psi\right|_{\left|\beta\right|+N}\cdot\left\Vert h\right\Vert ^{N-\alpha_{1}\left(M+\left|\beta\right|+2d\right)}\cdot\left(1+\left|y\right|\right)^{M}\left(\left|y-x\right|-\frac{\varepsilon_{1}}{2}\right)^{-N}.\label{eq:EtaDerivativeBoundPart2}
\end{align}
Choosing $N\in\mathbb{N}$ large enough, we can ensure that both 
\[
\mathbb{R}^{d}\setminus B_{\varepsilon_{1}}(x)\to\mathbb{R},\, y\mapsto(1+|y|)^{M}\cdot\left(\left|y-x\right|-\frac{\varepsilon_{1}}{2}\right)^{-N}
\]
and 
\[
K_{o}(W,V,R)\to\mathbb{R},h\mapsto\left\Vert h\right\Vert ^{N-\alpha_{1}\left(M+\left|\beta\right|+2d\right)}
\]
become integrable; for the latter, we use the same reasoning as for
$\kappa_{1}$. But this establishes (well-definedness and) smoothness
of $\kappa_{2}$, and thus of $\kappa$.

Finally, whenever $\varphi\in\mathcal{S}(\mathbb{R}^{d})$ satisfies
${\rm supp}(\varphi)\subset B_{\varepsilon_{1}/2}(x)$, we have 
\begin{align*}
\int_{\mathbb{R}^{d}}\kappa(z)\overline{\varphi(z)}\,{\rm d}z & =\int_{\mathbb{R}^{d}}\overline{\varphi(z)}\int_{\mathbb{R}^{d}}\int_{K_{o}(W,V,R)}W_{\psi}u(y,h)\cdot\left[\pi\left(y,h\right)\psi\right]\left(z\right)\,\frac{{\rm d}h}{\left|\det\left(h\right)\right|}\,{\rm d}y\,{\rm d}z\\
 & =\int_{\mathbb{R}^{d}}\int_{K_{o}(W,V,R)}W_{\psi}u(y,h)\cdot\smash{\underbrace{\int_{\mathbb{R}^{d}}\overline{\varphi(z)}\cdot\left[\pi\left(y,h\right)\psi\right]\left(z\right)\,{\rm d}z}_{=\overline{W_{\psi}\varphi(y,h)}}}\,\frac{{\rm d}h}{\left|\det\left(h\right)\right|}\,{\rm d}y\\
 & =\langle\eta\mid\varphi\rangle,
\end{align*}
as a comparison with equation (\ref{eqn:weak_def_eta}) shows.

This computation hinges on the applicability of Fubini's theorem in
the second line, which is justified because $\varphi$ is bounded
and has support in $B_{\varepsilon_{1}/2}\left(x\right)$, so that
the integral over $z\in\mathbb{R}^{d}$ is actually an integral over
$z\in B_{\varepsilon_{1}/2}\left(x\right)$. But equations \eqref{eq:EtaPartialDerivativeEstimateOnKo}
and \eqref{eq:EtaDerivativeBoundPart2} (with $\beta=0$) show that
the integrand can be bounded by 
\begin{align*}
 & \hphantom{\leq\:\,}\left|\varphi\left(z\right)\cdot\left(W_{\psi}u\right)\left(y,h\right)\cdot\left[\pi\left(y,h\right)\psi\right]\left(z\right)\right|/\left|\det\left(h\right)\right|\\
 & \leq\begin{cases}
\left\Vert \varphi\right\Vert _{\infty}\left|\psi\right|_{N}C_{M}C'\cdot\left\Vert h\right\Vert ^{M-\frac{3}{2}\alpha_{1}d}, & y\in B_{\varepsilon_{1}}\left(x\right),\\
C''\left(u\right)\cdot\left\Vert \varphi\right\Vert _{\infty}\left|\psi\right|_{N}\cdot\left\Vert h\right\Vert ^{N-\alpha_{1}\left(M+2d\right)}\cdot\left(1+\left|y\right|\right)^{M}\left(\left|y-x\right|-\frac{\varepsilon_{1}}{2}\right)^{-N}, & y\in\mathbb{R}^{d}\setminus B_{\varepsilon_{1}}\left(x\right),
\end{cases}
\end{align*}
where the right-hand sides are integrable (for sufficiently large
$M,N\in\mathbb{N}$) over the ranges $\left(z,y,h\right)\in B_{\varepsilon_{1}/2}\left(x\right)\times B_{\varepsilon_{1}}\left(x\right)\times K_{o}\left(W,V,R\right)$
and $\left(z,y,h\right)\in B_{\varepsilon_{1}/2}\left(x\right)\times\left(\mathbb{R}^{d}\setminus B_{\varepsilon_{1}}\left(x\right)\right)\times K_{o}\left(W,V,R\right)$,
respectively.

Thus, observation \eqref{enu:AlmostCharacterizationEtaIsSmooth} is
established, and the proof of part \ref{enu:RegularDirectedPointSufficientCondition}
is complete.
\end{proof}

\section{Wavelet characterizations of regular directed points}

\label{sect:character}We already remarked that Theorem \ref{thm:almost_char}
is not a characterization in the strict sense, since the sufficient
and necessary conditions in terms of wavelet coefficient decay refer
to different sets. In this section, we consider different possible
ways of closing this gap. All of them hinge on inclusions of the type
\[
K_{o}(W',V',R')\subset K_{i}(W,V,R)
\]
which allow to transfer a decay condition of the kind provided by
the necessary criterion in Theorem \ref{thm:almost_char}\ref{enu:RegularDirectedPointNecessaryCondition}
to one of the sort required in Theorem \ref{thm:almost_char}\ref{enu:RegularDirectedPointSufficientCondition}.
This might necessitate changing the wavelet $\psi$.

The following definition contains several distinct conditions which
will be seen to allow wavelet characterizations. 
\begin{defn}
\label{def:ConeApproximationProperties}Let $\xi\in\mathcal{O}\cap S^{d-1}$. 
\begin{enumerate}[label=(\alph*)]
\item Let $\mathcal{V}=(V_{n})_{n\in\mathbb{N}}$ be a family of subsets
$\emptyset\not=V_{n}\Subset\mathcal{O}$ and with $V_{n+1}\subset V_{n}$
for all $n\in\mathbb{N}$.

The dual action has the \textbf{weak $\mathcal{V}$-cone approximation
property at $\xi$} if for all $\xi$-neighborhoods $W\subset S^{d-1}$
and all $R>0$, there exist $n\in\mathbb{N}$ as well as $R'>0$ and
a $\xi$-neighborhood $W'\subset S^{d-1}$ such that 
\begin{equation}
K_{o}(W',V_{n},R')\subset K_{i}(W,V_{n},R).\label{eq:WeakConeApproximationInclusion}
\end{equation}

The dual action has the \textbf{global weak $\mathcal{V}$-cone approximation
property}, if it has the weak-$\mathcal{V}$-cone approximation property
at all $\xi\in\mathcal{O}\cap S^{d-1}$.

\item Let $\emptyset\neq V_{0}\Subset\mathcal{O}$. The dual action has
the \textbf{$V_{0}$-cone approximation property at $\xi$} if for
all $\xi$-neighborhoods $W\subset S^{d-1}$ and all $R>0$ there
are $R'>0$ and a $\xi$-neighborhood $W'\subset S^{d-1}$ such that
\[
K_{o}(W',V_{0},R')\subset K_{i}(W,V_{0},R).
\]

The dual action has the \textbf{global $V_{0}$-cone approximation
property} if it has the $V_{0}$-cone approximation property at $\xi$
for all $\xi\in\mathcal{O}\cap S^{d-1}$.

\end{enumerate}
\end{defn}
\begin{rem*}
Note that the $V_{0}$-cone approximation property is the same as
the weak $\mathcal{V}$-cone approximation property, for the sequence
$\mathcal{V}=(V_{0})_{n\in\mathbb{N}}$. The purpose of the family
$\mathcal{V}$ is to allow the possibility that the family of sets
$\left(V_{n}\right)_{n\in\mathbb{N}}$ can get ``arbitrarily small''.

If we avoid mention of the set $V_{0}$, we also call the $V_{0}$-cone
approximation property the \textbf{strong cone approximation property}
to distinguish it from the weak ($\mathcal{V}$-)cone approximation
property.
\end{rem*}
The following lemma is a direct consequence of the inclusion properties
for $K_{i/o}$ observed in Remark \ref{rem:incl_prop}. 
\begin{lem}
\label{lem:ConeApproximationForSubsets}Assume that the dual action
has the $V_{0}$-cone approximation property at $\xi$. Then for all
open $\emptyset\neq V\subset V_{0}$, all $\xi$-neighborhoods $W\subset S^{d-1}$
and all $R>0$, there exist $R'>0$ and a $\xi$-neighborhood $W'\subset S^{d-1}$
such that 
\[
K_{o}(W',V,R')\subset K_{i}(W,V,R).
\]
In other words, the dual action has the $V$-cone approximation property
at $\xi$ for every open $\emptyset\neq V\subset V_{0}$.

Also, if the inclusion in equation \eqref{eq:WeakConeApproximationInclusion}
is valid for some $n=n_{0}\in\mathbb{N}$, it automatically holds
for all $n\geq n_{0}$.\end{lem}
\begin{proof}
Simply note $K_{o}\left(W',V,R'\right)\subset K_{o}\left(W',V_{0},R'\right)\subset K_{i}\left(W,V_{0},R\right)\subset K_{i}\left(W,V,R\right)$
for suitable $R',W'$. The proof regarding the inclusion \eqref{eq:WeakConeApproximationInclusion}
is completely analogous (because the sequence $\mathcal{V}=\left(V_{n}\right)_{n\in\mathbb{N}}$
is nonincreasing).\end{proof}
\begin{rem}
The different versions of the cone approximation property can be motivated
as follows: Our aim is to characterize regular directed points by
the decay of wavelet coefficients corresponding to ``scales'' contained
in the set $K_{o}(W,V,R)$. As pointed out above, the inclusion property
$K_{o}(W',V',R')\subset K_{i}(W,V,R)$ is a means of closing the gap
between necessary and sufficient conditions in Theorem \ref{thm:almost_char}.
The weak cone approximation property is tailored to guarantee multiple
wavelet characterizations, and it can be read as a mathematically
precise way of saying that the system using multiple wavelets sucessfully
adapts to varying cone apertures.

By contrast, the strong cone approximation property can be understood
informally as an \emph{increase of the angular resolution in the wavelet
system with decreasing scales} (i.e., increasing frequencies). This
phenomenon is intricately linked to anisotropy, and it has been commented
on in the context of shearlets, curvelets, etc. Our definition provides
a rigourous and workable description of this property.

The cone approximation property is, however (at least potentially),
\emph{direction dependent}; currently, it seems possible that the
set $V_{0}$ may be required to vary nontrivially across $\mathcal{O}\cap S^{d-1}$.
Thus, if one is interested in simultaneously characterizing all directed
points $(x,\xi)$, with arbitrary $\xi\in\mathcal{O}\cap S^{d-1}$
using a \emph{single} wavelet, the global strong cone approximation
property is the natural prerequisite. Note that similar considerations
apply to the other technical condition, microlocal admissibility,
which we also need to control globally.

We finally remark that Lemmas \ref{lem:SingleOrbitMicrolocalAdmissibility}
and \ref{lem:SingleOrbitConeApproximationProperty} (will) show that
the pathological behaviour that $V_{0}$ has to vary with $\xi\in\mathcal{O}\cap S^{d-1}$
can not occur in the single orbit case.
\end{rem}
The following lemma gives a rigorous formulation of the fact that
the $V_{0}$-cone approximation property can indeed only hold for
``anisotropic groups'', as the above remark already indicated. It
shows that in order to fulfil the strong cone approximation property,
the group is not allowed to contain nontrivial positive scalar dilations.
\begin{lem}
\label{lem:anisotropy_necessary}Assume that $\mathcal{O}=H^{T}\xi_{0}$
is a single open orbit and let $\emptyset\neq V_{0}\Subset\mathcal{O}$.
Assume that $H\cap\left(0,\infty\right)\cdot{\rm id}$ is nontrivial
and that $d\geq2$.

Then the dual action does \emph{not} have the (strong) $V_{0}$-cone
approximation property at $\xi$, for any $\xi\in\mathcal{O}\cap S^{d-1}$.\end{lem}
\begin{proof}
Assume towards a contradiction that the dual action has the $V_{0}$-cone
approximation property at $\xi$ for some $\xi\in\mathcal{O}\cap S^{d-1}$.

Let $\eta\in V_{0}\subset\mathcal{O}$ be arbitrary. Because $\mathcal{O}=H^{T}\xi_{0}$
is a single orbit, we have $\eta\in\mathcal{O}=H^{T}\xi$ and thus
$h_{\xi}^{T}\xi=\eta\in V_{0}$ for some $h_{\xi}\in H$.

Define $\Phi:\mathbb{R}^{d}\setminus\left\{ 0\right\} \to S^{d-1},x\mapsto\frac{x}{\left|x\right|}$
and note that $\Phi\left(rx\right)=\Phi\left(x\right)$ for all $r\in\mathbb{R}^{\ast}$.
If $\Phi\left(\left[h_{\xi}^{-T}V_{0}\right]\setminus\left\{ 0\right\} \right)$
had at most one element $x\in\mathbb{R}^{d}$, this would imply 
\[
h_{\xi}^{-T}V_{0}\subset\mathbb{R}x
\]
in contradiction to the fact that $h_{\xi}^{-T}V_{0}$ is a nonempty
open subset of $\mathbb{R}^{d}$ with $d\geq2$. Hence, there are
$x,y\in\Phi\left(\left[h_{\xi}^{-T}V_{0}\right]\setminus\left\{ 0\right\} \right)$
with $x\neq y$. Let $s:=\left|x-y\right|>0$.

Let $R:=1$ and $W:=B_{s/2}\left(\xi\right)\cap S^{d-1}$. By assumption,
the dual action has the $V_{0}$-cone approximation property at $\xi$,
which yields a $\xi$-neighborhood $W'\subset S^{d-1}$ and some $R'>0$
with 
\[
K_{o}\left(W',V_{0},R'\right)\subset K_{i}\left(W,V_{0},R\right).
\]

By assumption on $H$, the subgroup $H\cap\left(0,\infty\right)\cdot{\rm id}$
is nontrivial. This ensures existence of some $\alpha\in\left(0,\smash{\left(1+R'\right)^{-1}}\right)$
with $\alpha\cdot{\rm id}\in H$. But $\xi\in h_{\xi}^{-T}V_{0}$
by choice of $h_{\xi}$ and furthermore $\xi\in W'$, which implies
$\alpha^{-1}\xi\in C\left(W'\right)$. Finally, $\xi\in S^{d-1}$
implies $\left|\alpha^{-1}\xi\right|=\alpha^{-1}>1+R'>R'$. All in
all, we arrive at 
\[
\alpha^{-1}\xi\in\left(\alpha h_{\xi}\right)^{-T}V_{0}\cap C\left(W',R'\right)\neq\emptyset.
\]
Hence, 
\[
\alpha h_{\xi}\in K_{o}\left(W',V_{0},R'\right)\subset K_{i}\left(W,V_{0},R\right).
\]
By definition of $K_{i}\left(W,V_{0},R\right)$, this means 
\[
\alpha^{-1}\cdot h_{\xi}^{-T}V_{0}=\left(\alpha h_{\xi}\right)^{-T}V_{0}\subset C\left(W,R\right)\subset C\left(W\right).
\]
Using the definition of $C\left(W\right)$ and of $\Phi$, we see
\[
x,y\in\Phi\left(\left[h_{\xi}^{-T}V_{0}\right]\setminus\left\{ 0\right\} \right)=\Phi\left(\left[\alpha^{-1}\cdot h_{\xi}^{-T}V_{0}\right]\setminus\left\{ 0\right\} \right)\subset W=B_{s/2}\left(\xi\right)\cap S^{d-1}.
\]
Thus, $s=\left|x-y\right|\leq\left|x-\xi\right|+\left|\xi-y\right|<\frac{s}{2}+\frac{s}{2}=s$,
a contradiction.
\end{proof}
In the case that $\mathcal{O}=H^{T}\xi_{0}$ is a single open orbit,
we saw in Lemma \ref{lem:SingleOrbitMicrolocalAdmissibility} that
it suffices to check the $V$-microlocal admissibility at a single
$\xi_{1}\in\mathcal{O}$. The same is true for the cone approximation
properties, as the following lemma shows. 
\begin{lem}
\label{lem:SingleOrbitConeApproximationProperty}Assume that $\mathcal{O}=H^{T}\xi_{0}$
is a single open orbit and let $\emptyset\neq V_{0}\Subset\mathcal{O}$.

If the dual action has the $V_{0}$-cone approximation property at
$\xi_{1}$ for some $\xi_{1}\in\mathcal{O}\cap S^{d-1}$, then the
dual action has the global $V_{0}$-cone approximation property.

Similarly, if the dual action has the weak $\mathcal{V}$-cone approximation
property at $\xi_{1}$ for some $\xi_{1}\in\mathcal{O}\cap S^{d-1}$,
then the dual action has the global weak $\mathcal{V}$-cone approximation
property.\end{lem}
\begin{proof}
In view of the remark following Definition \ref{def:ConeApproximationProperties},
it suffices to consider the weak case. We start with the following
observation: For each $\xi\in S^{d-1}$ and each $\xi$-neighborhood
$W\subset S^{d-1}$, there is some $\delta_{\xi,W}>0$ such that $\frac{v}{\left|v\right|}$
is a (well-defined) element of $W$ for all $v\in B_{\delta_{\xi,W}}\left(\xi\right)$.
To see this, observe that $B_{1}\left(\xi\right)\subset\mathbb{R}^{d}\setminus\left\{ 0\right\} $
is open and that the map 
\[
\Phi:B_{1}\left(\xi\right)\to S^{d-1},v\mapsto\frac{v}{\left|v\right|}
\]
is continuous with $\Phi\left(\xi\right)=\xi\in W$. Hence, $\Phi^{-1}\left(W\right)\subset B_{1}\left(\xi\right)$
is open with $\xi\in\Phi^{-1}\left(W\right)$, which implies the existence
of $\delta_{\xi,W}$.

Now, assume that the dual action has the weak $\mathcal{V}$-cone
approximation property at $\xi_{1}$. Let $\xi\in\mathcal{O}\cap S^{d-1}$
be arbitrary. Choose an arbitrary $\xi$-neighborhood $W\subset S^{d-1}$
and some $R>0$. By assumption, $\mathcal{O}$ is a single orbit,
so that $\xi=h_{\xi}^{T}\xi_{1}$ holds for some $h_{\xi}\in H$.

Let $R_{1}:=R\cdot\left\Vert \vphantom{h_{\xi}}\smash{h_{\xi}^{-1}}\right\Vert $
and $W_{1}:=\left[h_{\xi}^{-T}\cdot B_{\delta_{\xi,W}}\left(\xi\right)\right]\cap S^{d-1}$.
Note that $W_{1}$ is indeed a neighborhood of $\xi_{1}$. The weak
cone-approximation property yields some $R_{1}'>0$, a $\xi_{1}$-neighborhood
$W_{1}'\subset S^{d-1}$ as well as $n\in\mathbb{N}$ with 
\[
K_{o}\left(W_{1}',V_{n},R_{1}'\right)\subset K_{i}\left(W_{1},V_{n},R_{1}\right).
\]

Finally, let $R':=\left\Vert h_{\xi}\right\Vert \cdot R_{1}'$ and
$W':=\left[h_{\xi}^{T}\cdot B_{\delta_{\xi_{1},W_{1}'}}\left(\xi_{1}\right)\right]\cap S^{d-1}$.
Observe that $W'$ is indeed a $\xi$-neighborhood. It remains to
prove the inclusion 
\[
K_{o}\left(W',V_{n},R'\right)\subset K_{i}\left(W,V_{n},R\right).
\]
To this end, let $h\in K_{o}\left(W',V_{n},R'\right)$ be arbitrary.
By definition, this yields some $v\in V_{n}$ with $h^{-T}v\in C\left(W',R'\right)$
and hence $\left|h^{-T}v\right|>R'$ as well as 
\[
h^{-T}v=r\cdot w'=r\cdot h_{\xi}^{T}w\qquad\text{ for some }w\in B_{\delta_{\xi_{1},W_{1}'}}\left(\xi_{1}\right),
\]
and thus $\left(h_{\xi}h\right)^{-T}v=r\cdot w$.

On the one hand, we can use submultiplicativity of the norm to derive
\[
\left\Vert h_{\xi}\right\Vert \cdot R_{1}'=R'<\left|h^{-T}v\right|=\left|h_{\xi}^{T}\cdot h_{\xi}^{-T}h^{-T}v\right|\leq\left\Vert h_{\xi}\right\Vert \cdot\left|\left(h_{\xi}h\right)^{-T}v\right|
\]
and on the other hand, the choice of $\delta_{\xi_{1},W_{1}'}$ implies
\[
\left(h_{\xi}h\right)^{-T}v=r\left|w\right|\cdot\frac{w}{\left|w\right|}\in\left(0,\infty\right)\cdot W_{1}'\subset C\left(W_{1}'\right).
\]
Together, these considerations show 
\[
\left(h_{\xi}h\right)^{-T}v\in\left[\left(h_{\xi}h\right)^{-T}\cdot V_{n}\right]\cap C\left(W_{1}',R_{1}'\right)\neq\emptyset
\]
and hence $h_{\xi}h\in K_{o}\left(W_{1}',V_{n},R_{1}'\right)\subset K_{i}\left(W_{1},V_{n},R_{1}\right)$,
which implies 
\[
h_{\xi}^{-T}h^{-T}V_{n}=\left(h_{\xi}h\right)^{-T}V_{n}\subset C\left(W_{1},R_{1}\right)\qquad\text{ and hence }\qquad h^{-T}V_{n}\subset h_{\xi}^{T}\cdot C\left(W_{1},R_{1}\right),
\]
so that it suffices to prove $h_{\xi}^{T}\cdot C\left(W_{1},R_{1}\right)\subset C\left(W,R\right)$.

To this end, let $x\in C\left(W_{1},R_{1}\right)$ be arbitrary. This
means $\left|x\right|>R_{1}$ and $x=r\cdot w$ for some $r>0$ and
$w\in W_{1}=\left[h_{\xi}^{-T}\cdot B_{\delta_{\xi,W}}\left(\xi\right)\right]\cap S^{d-1}$.
Hence, $v:=h_{\xi}^{T}w\in B_{\delta_{\xi,W}}\left(\xi\right)$, so
that the definition of $\delta_{\xi,W}$ yields $\frac{v}{\left|v\right|}\in W$
and thus 
\[
h_{\xi}^{T}x=r\cdot h_{\xi}^{T}w=r\left|v\right|\cdot\frac{v}{\left|v\right|}\in C\left(W\right).
\]
Finally, 
\[
R\cdot\left\Vert \vphantom{h_{\xi}}\smash{h_{\xi}^{-1}}\right\Vert =R_{1}<\left|x\right|=\left|\vphantom{h_{\xi}}\smash{h_{\xi}^{-T}\cdot h_{\xi}^{T}x}\right|\leq\left\Vert \vphantom{h_{\xi}}\smash{h_{\xi}^{-T}}\right\Vert \cdot\left|\vphantom{h_{\xi}}\smash{h_{\xi}^{T}}x\right|=\left\Vert \vphantom{h_{\xi}}\smash{h_{\xi}^{-1}}\right\Vert \cdot\left|\vphantom{h_{\xi}}\smash{h_{\xi}^{T}}x\right|,
\]
which implies $\left|\vphantom{h_{\xi}}\smash{h_{\xi}^{T}}x\right|>R$
and thus $h_{\xi}^{T}x\in C\left(W,R\right)$. As $x\in C\left(W_{1},R_{1}\right)$
was arbitrary, the proof is complete.
\end{proof}
We can now formulate our main result, a wavelet characterization of
the wavefront set.
\begin{thm}
\label{thm:char_wfset}Let $H\leq{\rm GL}\left(\mathbb{R}^{d}\right)$
be a matrix group fulfilling the assumptions \ref{assume:proper_dual}.
\begin{enumerate}[label=(\alph*)]
\item \label{enu:WaveFrontSetCharacterizationWeakConeProperty} Assume
that the dual action of $H$ has the weak $\mathcal{V}$-cone approximation
at $\xi\in\mathcal{O}\cap S^{d-1}$, for some nonincreasing family
$\mathcal{V}=(V_{n})_{n\in\mathbb{N}}$ with $\emptyset\not=V_{n}\Subset\mathcal{O}$,
and in addition that the dual action is $V_{n_{0}}$-microlocally
admissible in direction $\xi$ for some $n_{0}\in\mathbb{N}$.

Furthermore, assume the existence of a family $(\psi_{n})_{n\in\mathbb{N}}$
of admissible Schwartz functions satisfying ${\rm supp}\left(\smash{\widehat{\psi}_{n}}\right)\subset V_{n}$.\\
Then, for each $x\in\mathbb{R}^{d}$, the following are equivalent:
\begin{enumerate}[label=(\arabic*)]
\item \label{enu:FundamentalEquivalenceRegularDirectedPoint}$\left(x,\xi\right)$
is a regular directed point of $u$,
\item \label{enu:FundamentalEquivalenceForAll}there is some $\xi$-neighborhood
$W\subset S^{d-1}$, some $R>0$ and some neighborhood $U\subset\mathbb{R}^{d}$
of $x$, as well as some $n_{1}\in\mathbb{N}$ such that for \emph{all}
$n\geq n_{1}$, the following holds:
\[
\qquad\qquad\forall N\in\mathbb{N}\,\exists C_{N}>0\,\forall y\in U\,\forall h\in K_{o}(W,V_{n},R)~:~|W_{\psi_{n}}u(y,h)|\le C_{N}\cdot\left\Vert h\right\Vert ^{N},
\]

\item \label{enu:FundamentalEquivalenceExistence}there is \emph{some} $n\geq n_{0}$,
a neighborhood $U$ of $x$, some $R>0$ and a $\xi$-neighborhood
$W\subset S^{d-1}$ such that 
\[
\qquad\qquad\forall N\in\mathbb{N}\,\exists C_{N}>0\,\forall y\in U\,\forall h\in K_{o}(W,V_{n},R)~:~|W_{\psi_{n}}u(y,h)|\le C_{N}\cdot\left\Vert h\right\Vert ^{N}.
\]

\end{enumerate}

If the dual action is globally $V_{n_{0}}$-microlocally admissible
for some $n_{0}\in\mathbb{N}$ and has the global weak $\mathcal{V}$-cone
approximation property for some nonincreasing family $\mathcal{V}=(V_{n})_{n\in\mathbb{N}}$,
then the same sequence of wavelets $(\psi_{n})_{n\in\mathbb{N}}$
can be used to simultaneously characterize (in the sense described
above) all regular directed points $(x,\xi)$ with $\xi\in\mathcal{O}\cap S^{d-1}$
arbitrary.

\item \label{enu:WaveFrontSetCharacterizationConeProperty}Assume that the
dual action is $V$-microlocally admissible in direction $\xi\in\mathcal{O}\cap S^{d-1}$
and has the $V$-cone approximation property at $\xi$ for some $\emptyset\neq V\Subset\mathcal{O}$.
Then for all admissible $\psi\in\mathcal{S}(\mathbb{R}^{d})$ with
${\rm supp}\left(\smash{\widehat{\psi}}\right)\subset V$ and all
$x\in\mathbb{R}^{d}$, the following equivalence holds, for all $u\in\mathcal{S}'(\mathbb{R}^{d})$:\\
 $(x,\xi)$ is a regular directed point of $u$ iff there exists a
neighborhood $U$ of $x$, some $R>0$ and a $\xi$-neighborhood $W\subset S^{d-1}$
such that 
\[
\forall N\in\mathbb{N}\,\exists C_{N}>0\,\forall y\in U\,\forall h\in K_{o}(W,V,R)~:~|W_{\psi}u(y,h)|\le C_{N}\cdot\left\Vert h\right\Vert ^{N}.
\]

If the dual action is globally $V$-microlocally admissible and has
the global $V$-cone approximation property, then the same wavelet
$\psi$ can be used to simultaneously characterize all regular directed
points $(x,\xi)$, with $\xi\in\mathcal{O}\cap S^{d-1}$ arbitrary.

\end{enumerate}
\end{thm}
\begin{proof}
We first prove part \ref{enu:WaveFrontSetCharacterizationWeakConeProperty}.
For ``\ref{enu:FundamentalEquivalenceRegularDirectedPoint}$\Rightarrow$\ref{enu:FundamentalEquivalenceForAll}'',
assume that $(x,\xi)$ is a regular directed point of $u$. By Theorem
\ref{thm:almost_char}\ref{enu:RegularDirectedPointNecessaryCondition},
there is some neighborhood $U\subset\mathbb{R}^{d}$ of $x$, some
$R>0$ and some $\xi$-neighborhood $W\subset S^{d-1}$, such that
for each admissible $\psi\in\mathcal{S}\left(\mathbb{R}^{d}\right)$
with ${\rm supp}(\widehat{\psi})\subset V_{n_{0}}$, the following
is true:
\begin{equation}
\forall N\in\mathbb{N}\,\exists C_{N}>0\,\forall y\in U\,\forall h\in K_{i}(W,\widehat{\psi}^{-1}\left(\mathbb{C}\setminus\left\{ 0\right\} \right),R)~:~|W_{\psi}u(y,h)|\le C_{N}\cdot\left\Vert h\right\Vert ^{N}.\label{eq:WaveFrontCharacterizationAlmostApplication}
\end{equation}
The weak $\mathcal{V}$-cone approximation property yields some $n_{1}\in\mathbb{N}$,
as well as $R'>0$ and some $\xi$-neighborhood $W'\subset S^{d-1}\cap\mathcal{O}$
such that the inclusion 
\[
K_{o}(W',V_{n_{1}},R')\subset K_{i}(W,V_{n_{1}},R)
\]
is valid. By Lemma \ref{lem:ConeApproximationForSubsets}, this implies
\[
K_{o}\left(W',V_{n},R'\right)\subset K_{i}\left(W,V_{n},R\right)
\]
for all $n\geq n_{1}$. Now let $n\geq\max\left\{ n_{0},n_{1}\right\} $
be arbitrary. This implies ${\rm supp}(\widehat{\psi_{n}})\subset V_{n}\subset V_{n_{0}}$,
so that equation \eqref{eq:WaveFrontCharacterizationAlmostApplication}
is valid for $\psi_{n}$ instead of $\psi$. But because of the inclusion
properties for $K_{i}$ (cf. Remark \ref{rem:incl_prop}), and because
of $\widehat{\psi_{n}}^{-1}\left(\mathbb{C}\setminus\left\{ 0\right\} \right)\subset{\rm supp}(\widehat{\psi_{n}})\subset V_{n}$,
we have
\[
K_{o}\left(W',V_{n},R'\right)\subset K_{i}\left(W,V_{n},R\right)\subset K_{i}\left(W,\widehat{\psi_{n}}^{-1}\left(\mathbb{C}\setminus\left\{ 0\right\} \right),R\right),
\]
which yields the desired decay estimate (with $W'$ for $W$ and $R'$
for $R$). It should be observed that $U,W'$ and $R'$ indeed do
\emph{not} depend on the particular $n\in\mathbb{N}$, as long as
$n\geq\max\left\{ n_{0},n_{1}\right\} $.

The implication ``\ref{enu:FundamentalEquivalenceForAll}$\Rightarrow$\ref{enu:FundamentalEquivalenceExistence}''
is immediate.

Finally, ``\ref{enu:FundamentalEquivalenceExistence}$\Rightarrow$\ref{enu:FundamentalEquivalenceRegularDirectedPoint}''
is a consequence of Theorem \ref{thm:almost_char}\ref{enu:RegularDirectedPointSufficientCondition},
because $V_{n_{0}}$-microlocal admissibility (in direction $\xi$)
implies $V_{n}$-microlocal admissibility (in direction $\xi$) for
$n\geq n_{0}$ (cf. Remark \ref{rem:MicrolocalAdmissiblityVCanBeShrunk}).

Part \ref{enu:WaveFrontSetCharacterizationConeProperty} is a special
case of part \ref{enu:WaveFrontSetCharacterizationWeakConeProperty},
since the $V$-cone approximation property is a special case of the
weak $\mathcal{V}$-cone approximation property, with $\mathcal{V}=\left(V\right)_{n\in\mathbb{N}}$,
which enables us to use $\psi_{n}=\psi$ for all $n\in\mathbb{N}$.

The statements about simultaneous characterizations are clear.
\end{proof}
As a further interesting benefit of the weak cone approximation property,
we note that it also simplifies the verification of microlocal admissibility,
at least in the single orbit case.
\begin{lem}
\label{lem:verify_adm_from_cone}Let $\mathcal{O}=H^{T}\xi_{0}$ be
an open orbit of $H$ with associated compact stabilizers and assume
that the dual action of $H$ has the weak $\mathcal{V}$-cone approximation
property at $\xi\in S^{d-1}\cap\mathcal{O}$ for some (nonincreasing)
family $\mathcal{V}=\left(V_{n}\right)_{n\in\mathbb{N}}$.

Furthermore, assume that there is some $n\in\mathbb{N}$ such that
$V_{n}$ fulfils condition \ref{enu:NormOfInverseEstimateOnKo} of
Definition \ref{defn:micro_regular} at $\xi$, i.e. there exists
a $\xi$-neighborhood $W_{0}\subset S^{d-1}\cap\mathcal{O}$, some
$R_{0}>0$, $\alpha_{1}>0$ and $C>0$ such that 
\[
\left\Vert h^{-1}\right\Vert \le C\cdot\left\Vert h\right\Vert ^{-\alpha_{1}}
\]
holds for all $h\in K_{o}(W_{0},V_{n},R_{0})$.

Then there is some $n_{0}\in\mathbb{N}$ such that the dual action
of $H$ is globally $V_{n}$-microlocally admissible for all $n\geq n_{0}$.\end{lem}
\begin{proof}
By Lemma \ref{lem:SingleOrbitMicrolocalAdmissibility}, it suffices
to show that the there is some $n_{0}\in\mathbb{N}$ such that the
dual action of $H$ is $V_{n}$-microlocally admissible \emph{in direction
$\xi$} for all $n\geq n_{0}$.

Let $n_{1}\in\mathbb{N}$ such that $V_{n_{1}}$ fulfils condition
\ref{enu:NormOfInverseEstimateOnKo} of Definition \ref{defn:micro_regular}
(as in the statement of the lemma). Since the weak $\mathcal{V}$-cone
approximation property at $\xi$ holds, there is some $n_{2}\in\mathbb{N}$,
some $R'>0$ as well as some $\xi$-neighborhood $W'\subset S^{d-1}$
such that
\[
K_{o}\left(W',V_{n_{2}},R'\right)\subset K_{i}\left(W_{0},V_{n_{2}},R_{0}\right).
\]
For $n\geq n_{0}:=\max\left\{ n_{1},n_{2}\right\} $, this yields
\begin{align}
K_{o}\left(W',V_{n},R'\right) & \subset K_{o}\left(W',V_{n_{2}},R'\right)\subset K_{i}\left(W_{0},V_{n_{2}},R_{0}\right)\nonumber \\
 & \subset K_{i}\left(W_{0},V_{n},R_{0}\right)\subset K_{o}\left(W_{0},V_{n_{1}},R_{0}\right).\label{eq:AdmissibilityByConeApproximationFundamentalInclusion}
\end{align}
Lemma \ref{lem:related_to_assumptions}\ref{enu:HEstimatedByXi}
implies that there are constants $\alpha,C>0$ such that $\left\Vert h\right\Vert \leq C\cdot\left|h^{-T}\eta_{0}\right|^{-\alpha}$
holds for all $\eta_{0}\in V_{n}\subset V_{n_{1}}$. Thus,
\begin{align*}
\int_{K_{o}(W',V_{n},R')}\left\Vert h\right\Vert ^{\alpha_{2}}\,{\rm d}h & \leq C\cdot\int_{K_{o}(W',V_{n},R')}\left|h^{-T}\eta_{0}\right|^{-\alpha\alpha_{2}}\,{\rm d}h\\
 & \le C\cdot\int_{K_{i}(W_{0},V_{n},R_{0})}\left|h^{-T}\eta_{0}\right|^{-\alpha\alpha_{2}}\,{\rm d}h\\
 & =C\cdot\int_{K_{i}(W_{0},V_{n},R_{0})^{-1}}\left|h^{T}\eta_{0}\right|^{-\alpha\alpha_{2}}\cdot\Delta_{H}(h)^{-1}\,{\rm d}h\\
 & =C\cdot\int_{K_{i}(W_{0},V_{n},R_{0})^{-1}}\left|h^{T}\eta_{0}\right|^{-\alpha\alpha_{2}}\cdot\left|\det\left(h\right)\right|^{-1}\cdot\Delta_{G}(h)^{-1}\,{\rm d}h.
\end{align*}
Here, $\Delta_{G}$ and $\Delta_{H}$ denote the modular functions
of $H$ and $G$, which are related by the fact that $\Delta_{G}$
is constant on the cosets of the translation subgroup (and thus can
be considered as a function on $H$), and by the relation $\Delta_{G}(h)=\Delta_{H}(h)\cdot\left|\det\left(h\right)\right|^{-1}$.

Since $K_{i}(W_{0},V_{n},R_{0})^{-1}\subset K_{o}(W_{0},V_{n_{1}},R_{0})^{-1}$,
Lemma \ref{lem:related_to_assumptions}\ref{enu:NormBoundedOnKoAndDeterminantEstimate}
-- together with Hadamard's inequality -- implies that $\left|\det\left(h\right)\right|^{-1}=\left|\det\left(h^{-1}\right)\right|\leq\left\Vert h^{-1}\right\Vert ^{d}$
is bounded on $K_{i}(W_{0},V_{n},R_{0})^{-1}$. Hence we may continue
the estimates via
\begin{align*}
\dots & \le C'\cdot\int_{K_{i}(W_{0},V_{n},R_{0})^{-1}}\left|h^{T}\eta_{0}\right|^{-\alpha\alpha_{2}}\cdot\Delta_{G}(h)^{-1}\,{\rm d}h\\
 & \le C'\cdot\int_{H}\chi_{C(W_{0},R_{0})}\left(h^{T}\eta_{0}\right)\cdot\left|h^{T}\eta_{0}\right|^{-\alpha\alpha_{2}}\cdot\Delta_{G}(h)^{-1}\,{\rm d}h~,
\end{align*}
since $h\in K_{i}(W_{0},V_{n},R_{0})^{-1}$ entails $h^{T}V_{n}\subset C(W_{0},R_{0})$,
and thus $\chi_{C(W_{0},R_{0})}(h^{T}\eta_{0})=1$.

Furthermore, we recall from \cite{Fu96} that we may write, for any
Borel-measurable $F:\mathcal{O}\to\mathbb{R}^{+}$, 
\[
\int_{\mathcal{O}}F(\xi)\,{\rm d}\xi=\frac{1}{c_{0}}\cdot\int_{H}F(h^{T}\eta_{0})\cdot\Delta_{G}(h)^{-1}\,{\rm d}h
\]
for some fixed $c_{0}>0$. In the present setting, we get
\begin{align*}
\int_{H}\chi_{C(W_{0},R_{0})}\left(h^{T}\eta_{0}\right)\cdot\left|h^{T}\eta_{0}\right|^{-\alpha\alpha_{2}}\cdot\Delta_{G}(h)^{-1}\,{\rm d}h & =c_{0}\cdot\int_{C(W_{0},R_{0})}\left|\xi\right|^{-\alpha\alpha_{2}}\,{\rm d}\xi\\
 & \leq c_{0}\cdot\int_{\mathbb{R}^{d}\setminus\overline{B_{R_{0}}}\left(0\right)}\left|\xi\right|^{-\alpha\alpha_{2}}\,{\rm d}\xi<\infty,
\end{align*}
as soon as $-\alpha\alpha_{2}<-d$, i.e. $\alpha_{2}>\frac{d}{\alpha}$.
This shows that part \ref{enu:NormIntegrability} of Definition \ref{defn:micro_regular}
is satisfied for $V_{n}$ as soon as $n\geq n_{0}$. Part \ref{enu:NormOfInverseEstimateOnKo}
trivially follows from the assumptions together with equation \eqref{eq:AdmissibilityByConeApproximationFundamentalInclusion}.

This establishes $V_{n}$-microlocal admissibility in direction $\xi$
for all $n\geq n_{0}$. Now transitivity of the action of $H$ yields
global microlocal admissibility via Lemma \ref{lem:SingleOrbitMicrolocalAdmissibility}.
\end{proof}
The following corollaries summarize our results for the important
case where $\mathcal{O}$ is a single orbit:
\begin{cor}
\label{cor:char_wfset_orbit_weak_cone_approx}Assume that $\mathcal{O}=H^{T}\xi_{0}$
is an open $H$-orbit with associated compact stabilizers, and let
$\mathcal{V}=(V_{n})_{n\in\mathbb{N}}$ denote a nonincreasing family
of open, relatively compact subsets of $\mathcal{O}$. \\
 Assume that the dual action of $H$ fulfils, for some $\xi_{1}\in\mathcal{O}\cap S^{d-1}$
the following properties:
\begin{itemize}
\item the weak $\mathcal{V}$-cone approximation property at $\xi_{1}$,
\item condition \ref{defn:micro_regular}\ref{enu:NormOfInverseEstimateOnKo},
with $V=V_{n}$ for some $n\in\mathbb{N}$, a suitable $R_{0}>0$
and a $\xi_{1}$-neighborhood $W_{0}\subset S^{d-1}$.
\end{itemize}
Pick any sequence $(\psi_{n})_{n\in\mathbb{N}}\subset\mathcal{S}(\mathbb{R}^{d})$
of admissible wavelets with ${\rm supp}(\widehat{\psi}_{n})\subset V_{n}$.
Then, there is some $n_{0}\in\mathbb{N}$, such that the following
statements are equivalent, for all $(x,\xi)\in\mathbb{R}^{d}\times(\mathcal{O}\cap S^{d-1})$
and $u\in\mathcal{S}'\left(\mathbb{R}^{d}\right)$
\begin{enumerate}
\item $\left(x,\xi\right)$ is a regular directed point of $u$,
\item \label{enu:FundamentalEquivalenceForAll-1}there is some $\xi$-neighborhood
$W\subset S^{d-1}$, some $R>0$ and some neighborhood $U\subset\mathbb{R}^{d}$
of $x$, as well as some $n_{1}\in\mathbb{N}$ such that for \emph{all}
$n\geq n_{1}$, the following holds:
\[
\qquad\qquad\forall N\in\mathbb{N}\,\exists C_{N}>0\,\forall y\in U\,\forall h\in K_{o}(W,V_{n},R)~:~|W_{\psi_{n}}u(y,h)|\le C_{N}\cdot\left\Vert h\right\Vert ^{N},
\]

\item \label{enu:FundamentalEquivalenceExistence-1}there is \emph{some}
$n\geq n_{0}$, a neighborhood $U\subset\mathbb{R}^{d}$ of $x$,
some $R>0$ and a $\xi$-neighborhood $W\subset S^{d-1}$ such that
\[
\qquad\qquad\forall N\in\mathbb{N}\,\exists C_{N}>0\,\forall y\in U\,\forall h\in K_{o}(W,V_{n},R)~:~|W_{\psi_{n}}u(y,h)|\le C_{N}\cdot\left\Vert h\right\Vert ^{N}.
\]

\end{enumerate}
\end{cor}
\begin{proof}
This is an immediate consequence of Theorem \ref{thm:char_wfset}\ref{enu:WaveFrontSetCharacterizationWeakConeProperty},
if we observe that Lemma \ref{lem:SingleOrbitConeApproximationProperty}
implies that $H$ has the global weak $\mathcal{V}$-cone approximation
property and that Lemma \ref{lem:verify_adm_from_cone} yields some
$n_{0}\in\mathbb{N}$ such that the dual action of $H$ is globally
$V_{n}$-microlocally admissible for all $n\geq n_{0}$ (and hence
in particular for $n=n_{0}$).
\end{proof}
If we apply this corollary for the case $\mathcal{V}=\left(V\right)_{n\in\mathbb{N}}$
and $\left(\psi_{n}\right)_{n\in\mathbb{N}}=\left(\psi\right)_{n\in\mathbb{N}}$,
we get the following more convenient version (if the dual action fulfils
the strong cone approximation property).
\begin{cor}
\label{cor:char_wfset_orbit_strong_cone_approx}Assume that $\mathcal{O}=H^{T}\xi_{0}$
is an open $H$-orbit with associated compact stabilizers and let
$\emptyset\neq V\Subset\mathcal{O}$.\\
Assume that the dual action of $H$ fulfils, for some $\xi_{1}\in\mathcal{O}\cap S^{d-1}$,
the following properties:
\begin{itemize}
\item the (strong) $V$-cone approximation property at $\xi_{1}$,
\item part \ref{enu:NormOfInverseEstimateOnKo} of Definition \ref{defn:micro_regular}.
\end{itemize}
Then for each admissible $\psi\in\mathcal{S}\left(\mathbb{R}^{d}\right)$
with ${\rm supp}(\widehat{\psi})\subset V$, each tempered distribution
$u\in\mathcal{S}'\left(\mathbb{R}^{d}\right)$ and each $\left(x,\xi\right)\in\mathbb{R}^{d}\times\left(\mathcal{O}\cap S^{d-1}\right)$,
the following are equivalent:
\begin{enumerate}
\item $\left(x,\xi\right)$ is a regular directed point of $u$,
\item there exists a neighborhood $U\subset\mathbb{R}^{d}$ of $x$, some
$R>0$ and a $\xi$-neighborhood $W\subset S^{d-1}$ such that
\[
\forall N\in\mathbb{N}\,\exists C_{N}>0\,\forall h\in K_{o}\left(W,V,R\right):\quad\left|W_{\psi}u\left(y,h\right)\right|\leq C_{N}\cdot\left\Vert h\right\Vert ^{N}.
\]

\end{enumerate}
\end{cor}
In Lemma \ref{lem:anisotropy_necessary}, we showed that a dilation
group $H$ can never fulfil the (strong) $V_{0}$-cone approximation
property if it is isotropic in the sense that the group $H\cap\left[\left(0,\infty\right)\cdot{\rm id}\right]$
is nontrivial. In this case, our methods (in particular Corollary
\ref{cor:char_wfset_orbit_weak_cone_approx}) only yield a characterization
of the wavefront set using \emph{multiple} wavelets.

We will now show that this is not a defect of our method of proof;
indeed, isotropic groups (in the sense described above) can never
yield a characterization of the wavefront set (in the sense of Corollary
\ref{cor:char_wfset_orbit_strong_cone_approx}) using only a \emph{single}
wavelet, at least as long as one is allowed to choose the wavelet
freely, only subject to a condition on the Fourier support.
\begin{lem}
\label{lem:IsotropicGroupsDoNotAllowSingleWaveletCharacterization}Assume
that $d\geq2$ and that $\mathcal{O}=H^{T}\xi_{0}\subset\mathbb{R}^{d}$
is an open $H$-orbit with associated compact stabilizers. Furthermore,
assume that there is some $\xi\in\mathcal{O}\cap S^{d-1}$ and some
$x\in\mathbb{R}^{d}$ as well as some $\varnothing\neq V\Subset\mathcal{O}$
such that the following holds:

For every distribution $u\in\mathcal{S}'\left(\mathbb{R}^{d}\right)$
with regular directed point $\left(x,\xi\right)$ and \emph{every}
admissible $\psi\in\mathcal{S}\left(\mathbb{R}^{d}\right)$ with ${\rm supp}\left(\smash{\widehat{\psi}}\right)\subset V$
there exists an open $\xi$-neighborhood $W\subset S^{d-1}$ and some
$R>0$ such that
\begin{equation}
\forall N\in\mathbb{N}\,\exists C_{N}>0\,\forall h\in K_{o}\left(W,V,R\right):\qquad\left|\left(W_{\psi}u\right)\left(x,h\right)\right|\leq C_{N}\cdot\left\Vert h\right\Vert ^{N}.\label{eq:WaveletDecayAssumption}
\end{equation}
Then $H$ is ``anisotropic'' in the sense that $H\cap\left[\left(0,\infty\right)\cdot{\rm id}\right]=\left\{ {\rm id}\right\} $.\end{lem}
\begin{rem*}
The proof will show that it actually suffices to assume that there
is some $N\in\mathbb{N}$ with $N>\frac{d}{2}-1$ such that
\[
\left|\left(W_{\psi}u\right)\left(x,h\right)\right|\leq C_{N,\psi,u}\cdot\left\Vert h\right\Vert ^{N}
\]
holds for all $h\in K_{o}\left(\left\{ \xi\right\} ,V,R\right)$ and
all admissible $\psi\in\mathcal{S}\left(\mathbb{R}^{d}\right)$ with
${\rm supp}\left(\smash{\widehat{\psi}}\right)\subset V$.\end{rem*}
\begin{proof}
Observe that $\pi\left(x,h\right)\psi=L_{x}D_{h}\psi$ with $\left(L_{x}f\right)\left(y\right)=f\left(y-x\right)$
and
\[
\left(D_{h}f\right)\left(y\right)=\left|\det\left(h\right)\right|^{-1/2}\cdot f\left(h^{-1}y\right).
\]
Hence, the assumption is also satisfied with $0$ instead of $x$,
because if the distribution $v\in\mathcal{S}'\left(\mathbb{R}^{d}\right)$
has the regular directed point $\left(0,\xi\right)$, then there is
some open $\xi$-neighborhood $W'\subset S^{d-1}$ as well as $\varphi\in C_{c}^{\infty}\left(\mathbb{R}^{d}\right)$
with $\varphi\equiv1$ on some neighborhood $U\subset\mathbb{R}^{d}$
of the origin such that
\[
\forall N\in\mathbb{N}\,\exists C_{N}>0\,\forall\eta\in C\left(W'\right):\qquad\left|\widehat{\varphi v}\left(\eta\right)\right|\leq C_{N}\cdot\left(1+\left|\eta\right|\right)^{-N}.
\]
But then $\gamma:=L_{x}\varphi\equiv1$ on the neighborhood $x+U\subset\mathbb{R}^{d}$
of $x$ and $u:=L_{x}v$ satisfies $\gamma u=L_{x}\left(\varphi v\right)$
as well as
\[
\left|\widehat{\gamma u}\left(\eta\right)\right|=\left|\widehat{L_{x}\left(\varphi v\right)}\left(\eta\right)\right|=\left|e^{-2\pi i\left\langle \eta,x\right\rangle }\cdot\widehat{\varphi v}\left(\eta\right)\right|\leq C_{N}\cdot\left(1+\left|\eta\right|\right)^{-N},
\]
for all $\eta\in C\left(W'\right)$, so that $u$ has the regular
directed point $\left(x,\xi\right)$. By assumption, this yields some
$\xi$-neighborhood $W\subset S^{d-1}$ and some $R>0$ such that
for arbitrary admissible $\psi$ with ${\rm supp}\left(\smash{\widehat{\psi}}\right)\subset V$,
the estimate
\begin{align*}
\left|\left(W_{\psi}v\right)\left(0,h\right)\right| & =\left|\left(W_{\psi}\left(L_{-x}u\right)\right)\left(0,h\right)\right|\\
 & =\left|\left\langle L_{-x}u\mid D_{h}\psi\right\rangle \right|\\
 & =\left|\left\langle u\mid L_{x}D_{h}\psi\right\rangle \right|=\left|\left\langle u\mid\pi\left(x,h\right)\psi\right\rangle \right|\\
 & =\left|\left(W_{\psi}u\right)\left(x,h\right)\right|\leq C_{N}\cdot\left\Vert h\right\Vert ^{N}
\end{align*}
holds for all $h\in K_{o}\left(W,V,R\right)$ and $N\in\mathbb{N}$,
so that the assumption is indeed also satisfied for $x=0$.

Now, let $\xi_{1}\in V\neq\varnothing$ be arbitrary. Using $\xi_{1}\in V\subset\mathcal{O}$
as well as $\xi\in\mathcal{O}$ and the fact that $\mathcal{O}=H^{T}\xi_{0}$
is a single orbit, we see that there is some $h\in H$ fulfilling
$h^{T}\xi_{1}=\xi$. This implies that $h^{T}V$ is open with $\xi\in h^{T}V\subset\mathcal{O}\subset\mathbb{R}^{d}\setminus\left\{ 0\right\} $
(note that $0\notin\mathcal{O}$ by Lemma \ref{lem:related_to_assumptions}).

Now there is some $\gamma\in h^{T}V$ with $\xi\notin{\rm span}\left(\left\{ \gamma\right\} \right)$,
because otherwise we would have $0\neq\xi=\alpha\cdot\gamma$ for
some $\alpha\in\mathbb{R}$, which implies $\alpha\neq0$ and hence
$\gamma=\xi/\alpha\in{\rm span}\left(\left\{ \xi\right\} \right)$
for arbitrary $\gamma\in h^{T}V$. But this would imply that $h^{T}V$
is contained in the one-dimensional space ${\rm span}\left(\left\{ \xi\right\} \right)$,
in contradiction to the fact that $h^{T}V\subset\mathbb{R}^{d}$ is
open with $d\geq2$.

Hence, we can choose some $\gamma\in h^{T}V$ with $\xi\notin{\rm span}\left(\left\{ \gamma\right\} \right)$
and set
\[
\mathcal{V}:=\left[{\rm span}\left(\left\{ \gamma\right\} \right)\right]^{\bot}.
\]
This implies $\xi\notin{\rm span}\left(\left\{ \gamma\right\} \right)=\mathcal{V}^{\bot}$.
Define now $u:=\delta_{\mathcal{V}}\in\mathcal{S}'\left(\mathbb{R}^{d}\right)$
by
\[
\delta_{\mathcal{V}}\left(f\right):=\int_{\mathcal{V}}f\left(x\right)\,{\rm d}S\left(x\right)\qquad\text{ for }f\in\mathcal{S}\left(\smash{\mathbb{R}^{d}}\right),
\]
where ${\rm d}S$ is the ($d-1$ dimensional) euclidean surface measure.
An application of \cite[Theorem 8.15]{HoermanderPartialDifferentialOp1}
shows that $\left(0,\xi\right)\in\mathbb{R}^{d}\times\left(S^{d-1}\cap\mathcal{O}\right)$
is a regular directed point of $u=\delta_{\mathcal{V}}$.

Choose some $\psi\in\mathcal{S}\left(\mathbb{R}^{d}\right)$ with
$\widehat{\psi}\in C_{c}^{\infty}\left(V\right)$, $\widehat{\psi}\geq0$
as well as $\widehat{\psi}\left(h^{-T}\gamma\right)>0$ and
\begin{equation}
\int_{H}\left|\widehat{\psi}\left(g^{T}\gamma\right)\right|^{2}\,{\rm d}g=1.\label{eq:WaveletAdmissibility}
\end{equation}
Note that such a choice is possible because of $h^{-T}\gamma\in h^{-T}h^{T}V\subset V$,
where $V\subset\mathcal{O}$ is open. This implies that all conditions
except for equation \eqref{eq:WaveletAdmissibility} can be fulfilled.
But these conditions already imply $\int_{H}\left|\smash{\widehat{\psi}\left(g^{T}\gamma\right)}\right|^{2}\,{\rm d}g>0$
because the integrand is continuous and nonnegative with $\left|\widehat{\psi}\left(\left(h^{-1}\right)^{T}\gamma\right)\right|^{2}>0$.
Hence, equation \eqref{eq:WaveletAdmissibility} can be achieved by
rescaling. The discussion after Assumption \ref{assume:proper_dual}
shows that $\psi$ is indeed an admissible wavelet.

By assumption, there exists a $\xi$-neighborhood $W\subset S^{d-1}\cap\mathcal{O}$
and some $R>0$ such that equation \eqref{eq:WaveletDecayAssumption}
is fulfilled. Assume towards a contradiction that $H_{1}:=H\cap\left[\left(0,\infty\right)\cdot{\rm id}\right]\neq\left\{ {\rm id}\right\} $.
As $H_{1}\leq H$ is a subgroup, this implies that there is a sequence
$\left(\alpha_{n}\right)_{n\in\mathbb{N}}$ in $\left(0,1\right)$
with $\alpha_{n}\to0$ and $\alpha_{n}\cdot{\rm id}\in H$ for all
$n\in\mathbb{N}$.

Let $g:=h^{-1}\in H$ and $g_{n}:=\alpha_{n}g\in H$. Then $g_{n}^{-T}\xi_{1}=\alpha_{n}^{-1}\cdot h^{T}\xi_{1}=\alpha_{n}^{-1}\cdot\xi\in C\left(\left\{ \xi\right\} \right)\subset C\left(W\right)$
(because of $\xi\in W$) with
\[
\left|g_{n}^{-T}\xi_{1}\right|=\left|\alpha_{n}^{-1}\cdot\xi\right|=\alpha_{n}^{-1}\xrightarrow[n\to\infty]{}\infty.
\]
Hence, $g_{n}^{-T}\xi_{1}\in C\left(\left\{ \xi\right\} ,R\right)\subset C\left(W,R\right)$
for $n\geq n\left(R\right)$ large enough. Because of $\xi_{1}\in V$,
this implies $g_{n}\in K_{o}\left(\left\{ \xi\right\} ,V,R\right)\subset K_{o}\left(W,V,R\right)$
for $n\geq n\left(R\right)$ large enough. By equation \eqref{eq:WaveletDecayAssumption},
this implies that for each $N\in\mathbb{N}$, there is some constant
$C_{N}>0$ such that
\[
\left|W_{\psi}u\left(0,g_{n}\right)\right|\leq C_{N}\cdot\left\Vert g_{n}\right\Vert ^{N}=C_{N}\left\Vert g\right\Vert ^{N}\cdot\alpha_{n}^{N}
\]
holds for all $n\in\mathbb{N}$.

But
\begin{eqnarray*}
\left(W_{\psi}u\right)\left(0,g_{n}\right) & = & \left\langle u\,\mid\,\pi\left(0,g_{n}\right)\psi\right\rangle \\
 & = & \left\langle \widehat{u}\,\mid\,\mathcal{F}\left(\pi\left(0,g_{n}\right)\psi\right)\right\rangle \\
 & \overset{\text{Gl. }\eqref{eq:QuasiRegularOnFourierSide}}{=} & \left|\det\left(g_{n}\right)\right|^{1/2}\cdot\left\langle \widehat{\delta_{\mathcal{V}}}\,\mid\, e^{-2\pi i\left\langle 0,\cdot\right\rangle }\cdot\widehat{\psi}\left(g_{n}^{T}\cdot\right)\right\rangle \\
 & \overset{\left(\ast\right)}{=} & c\cdot\alpha_{n}^{d/2}\cdot\left|\det\left(g\right)\right|^{1/2}\cdot\left\langle \delta_{\mathcal{V}^{\bot}}\,,\,\overline{\widehat{\psi}\left(g_{n}^{T}\cdot\right)}\right\rangle \\
 & \overset{\widehat{\psi}\text{ real}}{=} & c\cdot\alpha_{n}^{d/2}\cdot\left|\det\left(g\right)\right|^{1/2}\cdot\int_{\mathcal{V}^{\bot}}\widehat{\psi}\left(g_{n}^{T}\theta\right)\,{\rm d}\theta\\
 & = & c\cdot\alpha_{n}^{d/2}\cdot\left|\det\left(g\right)\right|^{1/2}\cdot\int_{\mathcal{V}^{\bot}}\widehat{\psi}\left(g^{T}\cdot\alpha_{n}\theta\right)\,{\rm d}\theta\\
 & \overset{\varrho=\alpha_{n}\theta,\quad\dim\left(\mathcal{V}^{\bot}\right)=1}{=} & c\cdot\alpha_{n}^{d/2}\cdot\left|\det\left(g\right)\right|^{1/2}\cdot\alpha_{n}^{-1}\int_{\mathcal{V}^{\bot}}\widehat{\psi}\left(g^{T}\varrho\right)\,{\rm d}\varrho\\
 & = & C_{d,g,\psi,\mathcal{V}}\cdot\alpha_{n}^{\frac{d}{2}-1}.
\end{eqnarray*}
In the step marked with $\left(\ast\right)$, we used \cite[Theorem 7.1.25]{HoermanderPartialDifferentialOp1}.
The constant $c>0$ depends only on $d\in\mathbb{N}$ and $d-1=\dim\left(\mathcal{V}\right)$
and comes from the fact that Hörmander uses a slightly different normalization
of the Fourier transform than in this paper.

Observe that the integrand of $\int_{\mathcal{V}^{\bot}}\widehat{\psi}\left(g^{T}\varrho\right)\,{\rm d}\varrho$
is a non-negative, continuous function which satisfies $\widehat{\psi}\left(g^{T}\gamma\right)=\widehat{\psi}\left(h^{-T}\gamma\right)>0$
and $\gamma\in\mathcal{V}^{\bot}$. Thus, $C_{d,g,\psi,\mathcal{V}}>0$. 

Putting everything together, we arrive at
\[
0<C_{d,g,\psi,\mathcal{V}}=\alpha_{n}^{1-\frac{d}{2}}\cdot\left|W_{\psi}u\left(0,g_{n}\right)\right|\leq C_{N}\left\Vert g\right\Vert ^{N}\cdot\alpha_{n}^{N+1-\frac{d}{2}}\xrightarrow[n\to\infty]{}0
\]
as soon as $N+1-\frac{d}{2}>0$. This is the desired contradiction.
\end{proof}

\section{Geometric reformulation of the cone approximation property}

\label{sect:geometric}The various cone approximation properties (cf.
Definition \ref{def:ConeApproximationProperties}) are defined in
terms of inclusions between sets of the form 
\begin{align*}
K_{o}\left(W,V,R\right) & =\left\{ h\in H\with h^{-T}V\cap C\left(W,R\right)\neq\emptyset\right\} ,\\
K_{i}\left(W,V,R\right) & =\left\{ h\in H\with h^{-T}V\subset C\left(W,R\right)\right\} 
\end{align*}
which are subsets of the dilation group $H$.

In this section, we will formulate so-called \textbf{geometric cone
approximation properties} which replace inclusions of the form 
\[
K_{o}\left(W',V,R'\right)\subset K_{i}\left(W,V,R\right)
\]
by inclusions of the form 
\[
C_{o}\left(W',V,R';H\right)\subset C_{i}\left(W,V,R;H\right),
\]
where the sets $C_{i/o}$ are subsets of $\mathcal{O}\subset\mathbb{R}^{d}$,
which are therefore more accessible for geometric arguments and geometric
intuition (at least for $d\leq3$). We will then show that these geometric
conditions are equivalent to the conditions in Definition \ref{def:ConeApproximationProperties}. 
\begin{defn}
\label{def:GeometricConeApproximationProperty}Let $H$ be a dilation
group satisfying Assumption \ref{assume:proper_dual}. Let $W\subset S^{d-1}$
be open and let $\emptyset\neq V\Subset\mathcal{O}$ as well as $R>0$.
Define 
\begin{align*}
C_{o}\left(W,V,R;H\right) & :=\bigcup\left\{ h^{T}V\with h\in H\text{ with }h^{T}V\cap C\left(W,R\right)\neq\emptyset\right\} =\bigcup\left\{ h^{-T}V\with h\in K_{o}\left(W,V,R\right)\right\} ,\\
C_{i}\left(W,V,R;H\right) & :=\bigcup\left\{ h^{T}V\with h\in H\text{ with }h^{T}V\subset C\left(W,R\right)\right\} =\bigcup\left\{ h^{-T}V\with h\in K_{i}\left(W,V,R\right)\right\} .
\end{align*}
Let $\xi\in\mathcal{O}\cap S^{d-1}$. 
\begin{enumerate}[label=(\alph*)]
\item Let $\mathcal{V}=(V_{n})_{n\in\mathbb{N}}$ be a nonincreasing family
of subsets $\emptyset\not=V_{n}\Subset\mathcal{O}$.

The dual action has the \textbf{weak geometric $\mathcal{V}$-cone
approximation property at $\xi$} if for all $\xi$-neighborhoods
$W\subset S^{d-1}$ and all $R>0$, there exist $n\in\mathbb{N}$
as well as $R'>0$ and a $\xi$-neighborhood $W'\subset S^{d-1}$
such that 
\[
C_{o}\left(W',V_{n},R';H\right)\subset C_{i}\left(W,V_{n},R;H\right).
\]

\item Let $\emptyset\neq V_{0}\Subset\mathcal{O}$. The dual action has
the \textbf{geometric $V_{0}$-cone approximation property at $\xi$}
if for all $\xi$-neighborhoods $W\subset S^{d-1}$ and all $R>0$
there are $R'>0$ and a $\xi$-neighborhood $W'\subset S^{d-1}$ such
that 
\[
C_{o}\left(W',V_{0},R';H\right)\subset C_{i}\left(W,V_{0},R;H\right).
\]

\end{enumerate}
\end{defn}
Again, note that the geometric cone approximation property is a special
case of its weak sibling. We now observe some simple implications
between inclusions of the sets $C_{i/o}$ and inclusions of the sets
$K_{i/o}$, which will then show that the geometric cone approximation
properties are indeed equivalent to the earlier defined versions. 
\begin{lem}
\label{lem:RelationsBetweenKioAndCio}Let $\emptyset\neq V,V'\Subset\mathcal{O}$
and $W,W'\subset S^{d-1}$ as well as $R,R'>0$. 
\begin{enumerate}[label=(\alph*)]
\item \label{enu:CInclusionYieldsKInclusion}The inclusion $C_{o}\left(W',V',R';H\right)\subset C_{i}\left(W,V,R;H\right)$
implies 
\[
K_{o}\left(W',V',R'\right)\subset K_{i}\left(W,V',R\right).
\]

\item \label{enu:KInclusionYieldsCInclusion}The inclusion $K_{o}\left(W',V',R'\right)\subset K_{i}\left(W,V,R\right)$
implies 
\[
C_{o}\left(W',V',R';H\right)\subset C_{i}\left(W,V,R;H\right)
\]
as long as $V'\subset V$ holds. 
\item In particular, for each $\emptyset\neq V_{0}\Subset\mathcal{O}$ and
all $\xi\in\mathcal{O}\cap S^{d-1}$,

\begin{enumerate}[label=(\roman*)]
\item the dual action has the weak $\mathcal{V}$-cone approximation property
at $\xi$ if and only if it has the \emph{geometric} weak $\mathcal{V}$-cone
approximation property at $\xi$. 
\item the dual action has the $V_{0}$-cone approximation property at $\xi$
if and only if it has the \emph{geometric} $V_{0}$-cone approximation
property at $\xi$.
\end{enumerate}
\end{enumerate}
\end{lem}
\begin{proof}
For the first statement, let $h\in K_{o}\left(W',V',R'\right)$ be
arbitrary. Note that the inclusion $C_{i}\left(W,V,R;H\right)\subset C\left(W,R\right)$
is an easy consequence of the definitions. Hence, 
\[
h^{-T}V'\subset C_{o}\left(W',V',R';H\right)\subset C_{i}\left(W,V,R;H\right)\subset C\left(W,R\right),
\]
which implies $h\in K_{i}\left(W,V',R\right)$.

For the second statement, observe 
\begin{align*}
C_{o}\left(W',V',R';H\right) & =\bigcup\left\{ h^{-T}V'\with h\in K_{o}\left(W',V',R'\right)\right\} \\
 & \subset\bigcup\left\{ h^{-T}V'\with h\in K_{i}\left(W,V,R\right)\right\} \\
 & \subset\bigcup\left\{ h^{-T}V\with h\in K_{i}\left(W,V,R\right)\right\} =C_{i}\left(W,V,R;H\right).
\end{align*}
Here, $V'\subset V$ was used in the last line.

For the remaining statements, it suffices to consider the weak case.
Assume that the dual action has the weak $\mathcal{V}$-cone approximation
property at $\xi$. For any $\xi$-neighborhood $W\subset S^{d-1}$
and $R>0$, this yields $n\in\mathbb{N}$, a $\xi$-neighborhood $W'\subset S^{d-1}$
and some $R'>0$ such that the inclusion $K_{o}\left(W',V_{n},R'\right)\subset K_{i}\left(W,V_{n},R\right)$
is satisfied. By part \ref{enu:KInclusionYieldsCInclusion} (with
$V':=V_{n}\subset V_{n}=:V$), this yields 
\[
C_{o}\left(W',V_{n},R';H\right)\subset C_{i}\left(W,V_{n},R;H\right),
\]
so that the dual action also has the weak geometric $\mathcal{V}$-cone
approximation property.

The converse is proved analogously via \ref{enu:CInclusionYieldsKInclusion}
(with $V'=V=V_{n}$).
\end{proof}
We remark that it is also possible to introduce global versions of
the geometric cone approximation properties, and to extend the equivalence
to the global versions.

\section{Examples}

\label{sect:examples}In this section, we discuss various examples
of dilation groups, and verify the technical conditions introduced
in the paper (wherever this is possible). All the examples considered
below belong to the irreducible setting, i.e., the dilation group
acts with a single open orbit and compact stabilizers. Thus Corollary
\ref{cor:char_wfset_orbit_weak_cone_approx} applies and yields wavelet
characterizations of regular directed points, as soon as the dual
action fulfils a certain list of conditions. In the single orbit case,
the task of verifying these conditions simplifies considerably: 
\begin{itemize}
\item The conditions listed in \ref{assume:proper_dual} are automatically
fulfilled. 
\item In view of Lemmas \ref{lem:SingleOrbitMicrolocalAdmissibility} and
\ref{lem:SingleOrbitConeApproximationProperty}, it is sufficient
to check microlocal admissibility and the (weak) cone approximation
property at a single, conveniently chosen point $\xi_{1}\in\mathcal{O}\cap S^{d-1}$.
\item With the weak cone approximation property already established, only
the first condition of microlocal admissibility needs to be checked,
by Lemma \ref{lem:verify_adm_from_cone}, at least if one is only
interested in $V_{n}$-microlocal admissibility for $n$ sufficiently
large (or if the dual action satisfies the strong cone approximation
property).

This task is further simplified by the fact that we may replace (in
a suitable sense) the set $K_{o}$ by the smaller set $K_{i}$ in
this condition, as shown in Lemma \ref{lem:MicrolocalAdmissibilityKiStattKo}
below.

\item Also, there is an easily checked necessary condition for validity
of the strong cone approximation property, provided by Lemma \ref{lem:anisotropy_necessary}.
\end{itemize}
The following examples will show that the remaining steps can indeed
be carried out in a variety of settings. But first, let us prove the
result mentioned above that the set $K_{o}$ may be replaced by the
set $K_{i}$ in the verification of microlocal admissibility:
\begin{lem}
\label{lem:MicrolocalAdmissibilityKiStattKo}Let $\mathcal{O}=H^{T}\xi_{0}$
be an open orbit of $H$ with associated compact stabilizers, let
$\xi\in S^{d-1}\cap\mathcal{O}$ and assume that the dual action of
$H$ has the weak $\mathcal{V}$-cone approximation property at $\xi$
for some nonincreasing family $\mathcal{V}=\left(V_{n}\right)_{n\in\mathbb{N}}$
with $\emptyset\neq V_{n}\Subset\mathcal{O}$.

Furthermore, assume that there exists a $\xi$-neighborhood $W_{0}\subset S^{d-1}\cap\mathcal{O}$
and some $R_{0}>0$ such that for all sufficiently large $n\in\mathbb{N}$,
there are constants $\alpha_{n}>0$ and $C_{n}>0$ such that 
\[
\left\Vert h^{-1}\right\Vert \le C_{n}\cdot\left\Vert h\right\Vert ^{-\alpha_{n}}
\]
holds for all $h\in K_{i}(W_{0},V_{n},R_{0})$.

Then there is some $n_{0}\in\mathbb{N}$ such that the dual action
of $H$ is globally $V_{n}$-microlocally admissible for all $n\geq n_{0}$.\end{lem}
\begin{rem*}
In the case of the $V_{0}$-cone approximation property, i.e. $\mathcal{V}=\left(V_{0}\right)_{n\in\mathbb{N}}$,
this lemma implies that we can indeed replace $K_{o}$ by $K_{i}$
for the verification of (the first condition of) microlocal admissibility.\end{rem*}
\begin{proof}
By the weak cone approximation property (cf. Definition \ref{def:ConeApproximationProperties}),
there is some $n_{1}\in\mathbb{N}$, some $\xi$-neighborhood $W'\subset S^{d-1}\cap\mathcal{O}$
and some $R'>0$ such that 
\[
K_{o}\left(W',V_{n_{1}},R'\right)\subset K_{i}\left(W_{0},V_{n_{1}},R_{0}\right).
\]
By Lemma \ref{lem:ConeApproximationForSubsets}, this yields $K_{o}\left(W',V_{n},R'\right)\subset K_{i}\left(W_{0},V_{n},R_{0}\right)$
for all $n\geq n_{1}$.

Making use of the assumptions, we derive $\left\Vert h^{-1}\right\Vert \leq C_{n}\cdot\left\Vert h\right\Vert ^{-\alpha_{n}}$
for all $h\in K_{o}\left(W',V_{n},R'\right)$ and all sufficiently
large $n\in\mathbb{N}$. It remains to invoke Lemma \ref{lem:verify_adm_from_cone}
to conclude the proof.
\end{proof}

\subsection{The similitude group}

The similitude group was the first dilation group in higher dimensions
for which continuous wavelet transforms were studied \cite{Mu}; for
the study of wavefront sets, it was employed for example in \cite{Moritoh_1995,Pilipovic_et_al_2006}.
The group is given by $H=\mathbb{R}^{+}\cdot SO\left(d\right)$. We
only consider the case $d\geq2$. In this case, $H$ has the unique
open dual orbit $\mathbb{R}^{d}\setminus\{0\}$, on which it acts
with compact stabilizers. As it contains all scalar dilations, we
know by Lemma \ref{lem:anisotropy_necessary} that the best we can
expect of $H$ is the weak cone approximation property. We will now
verify this, together with microlocal admissibility. Hence, by Corollary
\ref{cor:char_wfset_orbit_weak_cone_approx}, $H$ allows a multiple
wavelet characterization of regular directed points. This result partly
generalizes \cite{Pilipovic_et_al_2006}.

\subsubsection{$H$ fulfils the weak cone approximation property}

We write arbitrary elements $h\in H$ as $h=a\vartheta$, with $a>0$
and $\vartheta\in SO\left(d\right)$. We pick $\xi_{1}=(1,0,\ldots,0)^{T}$,
and define the sequence $\mathcal{V}=(V_{n})_{n\in\mathbb{N}}$ by
$V_{n}=B_{1/n}(\xi_{1})$.

Now let $W\subset S^{d-1}$ be an open neighborhood of $\xi_{1}$
and let $R>0$. Choose $\epsilon>0$ with $B_{\epsilon}(\xi_{1})\cap S^{d-1}\subset W$.
Let $W'=B_{\epsilon'}(\xi_{1})\cap S^{d-1}$, $R'>0$ and $n\ge2$;
our aim is to specify $\epsilon',R',n$ (only depending on $\varepsilon,R$)
in a way that ensures $K_{o}(W',V_{n},R')\subset K_{i}(W,V_{n},R)$.

Hence, assume that $h=a\vartheta\in K_{o}(W',V_{n},R')$. This ensures
existence of some $\xi\in V_{n}$ with $a^{-1}\vartheta\xi\in C(W',R')$
because of $\vartheta^{-T}=\vartheta$.

In particular, $R'<|a^{-1}\vartheta\xi|$, which entails via $\xi\in V_{n}$
that 
\[
a^{-1}>\frac{R'}{1+1/n}~.
\]
For $\xi'\in V_{n}$ arbitrary, we then obtain 
\begin{equation}
|a^{-1}\vartheta\xi'|\ge a^{-1}(1-1/n)>\frac{1-1/n}{1+1/n}R'=\frac{n-1}{n+1}R'~.\label{eqn:est_sim_a}
\end{equation}

Furthermore, the fact that $a^{-1}\vartheta\xi\in C(W',R')$ implies
\[
\frac{\vartheta\xi}{|\xi|}=\frac{a^{-1}\vartheta\xi}{|a^{-1}\vartheta\xi|}\in W'=B_{\epsilon'}(\xi_{1})\cap S^{d-1}~.
\]

In addition, we have
\begin{align*}
\left|\frac{\vartheta\xi}{\left|\xi\right|}-\frac{\vartheta\xi'}{\left|\xi'\right|}\right| & \le\frac{\left|\vartheta(\xi-\xi')\right|}{\left|\xi\right|}+\left|\vartheta\xi'\right|\left|\frac{1}{\left|\xi'\right|}-\frac{1}{\left|\xi\right|}\right|\\
 & \le\frac{2\left|\xi-\xi'\right|}{\left|\xi\right|}\le4\frac{1/n}{1-1/n}=\frac{4}{n-1},
\end{align*}
leading to 
\begin{equation}
\left|\xi_{1}-\frac{a^{-1}\vartheta\xi'}{\left|a^{-1}\vartheta\xi'\right|}\right|\le\left|\xi_{1}-\frac{\vartheta\xi}{\left|\xi\right|}\right|+\frac{4}{n-1}<\varepsilon'+\frac{4}{n-1}.\label{eqn:est_sim_b}
\end{equation}
Now (\ref{eqn:est_sim_a}) and (\ref{eqn:est_sim_b}) combined yield
that whenever 
\[
\epsilon'<\frac{\epsilon}{2},\quad\frac{4}{n-1}<\frac{\epsilon}{2}\quad\mbox{ and }\quad\frac{n-1}{n+1}R'>R,
\]
it follows that $a^{-1}\vartheta\xi'\in C(W,R)$, for all $\xi'\in V_{n}$.
This means that $h\in C_{i}(W,V_{n},R)$, and the weak $\mathcal{V}$-cone
approximation property is shown.

\subsubsection{The dual action of $H$ is microlocally admissible}

In view of the previous subsection, it remains to verify condition
\ref{defn:micro_regular}\ref{enu:NormOfInverseEstimateOnKo}. But
this condition is implied by the identitiy 
\[
\left\Vert h^{-1}\right\Vert =\|(a\vartheta)^{-1}\|=\|a^{-1}\vartheta^{-1}\|=a^{-1}=\|a\vartheta\|^{-1}=\left\Vert h\right\Vert ^{-1}~,
\]
valid for \emph{all} $h=a\vartheta\in H$.

By Lemma \ref{lem:verify_adm_from_cone}, this implies that there
is some $n_{0}\in\mathbb{N}$ such that the dual action of $H$ is
$V_{n}$-microlocally admissible for all $n\geq n_{0}$.

\subsection{The diagonal group}

The diagonal group is given by 
\[
H=\left\{ \left(\begin{array}{cccc}
a_{1}\\
 & a_{2}\\
 &  & \ddots\\
 &  &  & a_{d}
\end{array}\right)\in\mathbb{R}^{d\times d}\,:\,\prod_{i=1}^{d}a_{i}\not=0\right\} ~.
\]
The dual action of $H$ has the single open orbit $\left(\mathbb{R}^{\ast}\right)^{d}=\left(\mathbb{R}\setminus\left\{ 0\right\} \right)^{d}=H^{T}\left(1,\dots,1\right)^{T}$
and acts on this orbit with trivial (hence compact) stabilizers.

To our knowledge, this group has not yet been investigated for its
properties to characterize wavefront sets. We will show that it allows
a multiple wavelet characterization, noting that again by Lemma \ref{lem:anisotropy_necessary},
$H$ does not fulfil the cone approximation property, and hence one
cannot hope for single wavelet characterizations (at least not in
the sense of Corollary \ref{cor:char_wfset_orbit_strong_cone_approx},
cf. Lemma \ref{lem:IsotropicGroupsDoNotAllowSingleWaveletCharacterization}).

\subsubsection{$H$ fulfils the weak cone approximation property}

We verify the weak cone approximation property at $\xi_{0}=\frac{1}{\sqrt{d}}(1,1,\ldots,1)^{T}$.
Given $\epsilon>0$, we let 
\[
U_{\varepsilon}=\left\{ \xi'\in S^{d-1}\with\forall i=1,\ldots,d:\frac{1}{1+\varepsilon}<\sqrt{d}\xi_{i}'<\left(1+\varepsilon\right)\right\} \subset\left(0,\infty\right)^{d}.
\]
Let $W\subset S^{d-1}$ denote a neighborhood of $\xi_{0}$, and let
$R>0$. Then there exists $\varepsilon>0$ with $U_{\varepsilon}\subset W$.

We are going to establish the weak cone approximation property with
respect to the sequence $\mathcal{V}=(V_{n})_{n\in\mathbb{N}}$ defined
by 
\[
V_{n}=\left\{ \xi\in\mathbb{R}^{d}\with\frac{n}{n+1}<|\xi|<\frac{n+1}{n}\text{ and }\frac{\xi}{|\xi|}\in U_{1/n}\right\} \subset\left(0,\infty\right)^{d},
\]
for $n\in\mathbb{N}$.

For this purpose, let $W'=U_{\varepsilon'}$, and fix $R'>0$ and
$n\in\mathbb{N}$. We will show that these three parameters can be
chosen in such a way that $K_{o}\left(W',V_{n},R'\right)\subset K_{i}\left(W,V,R\right)$
holds. To see this, let $h={\rm diag}(a_{1},\ldots,a_{d})\in K_{o}(W',V_{n},R')$.
This simply means that there exists some $\xi\in V_{n}$ with $h^{-T}\xi\in C(W',R')$.
Let $\xi'\in V_{n}$ be arbitrary. We have to show that this implies
$h^{-T}\xi'\in C\left(W,R\right)$ (for suitable values of $W',R',n$
depending only on $\varepsilon,R$).

First note that $\xi\in V_{n}$ entails, via $\frac{n}{n+1}<|\xi|<\frac{n+1}{n}$
and $\frac{\xi}{|\xi|}\in U_{1/n}$, that 
\begin{equation}
\forall i=1,\ldots,d:\qquad\frac{1}{\sqrt{d}}\left(\frac{n}{n+1}\right)^{2}<\xi_{i}<\frac{1}{\sqrt{d}}\left(\frac{n+1}{n}\right)^{2}~.\label{eqn:diag_est_xi}
\end{equation}
The same estimate also holds for $\xi'$ instead of $\xi$.

Together with $h^{-T}\xi\in C\left(W',R'\right)$, equation \eqref{eqn:diag_est_xi}
implies
\[
R'<\left|\left(a_{1}^{-1}\xi_{1},\dots,a_{d}^{-1}\xi_{d}\right)^{T}\right|<\frac{1}{\sqrt{d}}\cdot\left(\frac{n+1}{n}\right)^{2}\cdot\left|\left(a_{1}^{-1},\dots,a_{d}^{-1}\right)^{T}\right|.
\]
In combination with \eqref{eqn:diag_est_xi} (for $\xi'$ instead
of $\xi$), this immediately yields 
\begin{align}
\gamma' & :=\left|h^{-T}\xi'\right|=\left|\left(a_{1}^{-1}\xi_{1}',\ldots,a_{d}^{-1}\xi_{d}'\right)^{T}\right|\nonumber \\
 & >\frac{1}{\sqrt{d}}\cdot\left(\frac{n}{n+1}\right)^{2}\cdot\left|\left(a_{1}^{-1},\dots,a_{d}^{-1}\right)^{T}\right|>\left(\frac{n}{n+1}\right)^{4}\cdot R'~.\label{eqn:est_diag_ca1}
\end{align}

Now let $\gamma:=\left|h^{-T}\xi\right|=\left|\left(a_{1}^{-1}\xi_{1},\dots,a_{d}^{-1}\xi_{d}\right)\right|>0$.
Using $h^{-T}\xi\in C\left(W',R'\right)$, we arrive at
\[
\frac{1}{\gamma}\cdot\left(a_{1}^{-1}\xi_{1},\dots,a_{d}^{-1}\xi_{d}\right)=\frac{h^{-T}\xi}{\left|h^{-T}\xi\right|}\in W'=U_{\varepsilon'},
\]
which implies
\begin{equation}
\frac{1}{1+\varepsilon'}<\frac{\sqrt{d}a_{i}^{-1}\xi_{i}}{\gamma}<1+\varepsilon'\qquad\forall i\in\left\{ 1,\dots,d\right\} .\label{eq:DiagonalGroupXiNormalizedEstimate}
\end{equation}
Using this together with equation \eqref{eqn:diag_est_xi} (for $\xi$
as well as $\xi'$), we derive
\begin{align*}
a_{i}^{-1}\xi_{i}' & =\frac{\sqrt{d}a_{i}^{-1}\xi_{i}}{\gamma}\cdot\frac{\xi_{i}'}{\xi_{i}}\cdot\frac{\gamma}{\sqrt{d}}\\
 & <\left(1+\varepsilon'\right)\cdot\frac{\frac{1}{\sqrt{d}}\cdot\left(\frac{n+1}{n}\right)^{2}}{\frac{1}{\sqrt{d}}\cdot\left(\frac{n}{n+1}\right)^{2}}\cdot\frac{\gamma}{\sqrt{d}}\\
 & =\underbrace{\left(1+\varepsilon'\right)\cdot\left(\frac{n+1}{n}\right)^{4}}_{=:C_{\varepsilon',n}}\cdot\frac{\gamma}{\sqrt{d}},
\end{align*}
as well as
\[
a_{i}^{-1}\xi_{i}'=\frac{\sqrt{d}a_{i}^{-1}\xi_{i}}{\gamma}\cdot\frac{\xi_{i}'}{\xi_{i}}\cdot\frac{\gamma}{\sqrt{d}}>\frac{1}{1+\varepsilon'}\cdot\frac{\frac{1}{\sqrt{d}}\cdot\left(\frac{n}{n+1}\right)^{2}}{\frac{1}{\sqrt{d}}\cdot\left(\frac{n+1}{n}\right)^{2}}\cdot\frac{\gamma}{\sqrt{d}}=C_{\varepsilon',n}^{-1}\cdot\frac{\gamma}{\sqrt{d}}.
\]

This implies
\begin{align*}
C_{\varepsilon',n}^{-1}\cdot\gamma & =C_{\varepsilon',n}^{-1}\cdot\sqrt{\sum_{i=1}^{d}\left(\frac{\gamma}{\sqrt{d}}\right)^{2}}<\sqrt{\sum_{i=1}^{d}\left(a_{i}^{-1}\xi_{i}'\right)^{2}}=\gamma'\\
 & =\sqrt{\sum_{i=1}^{d}\left(a_{i}^{-1}\xi_{i}'\right)^{2}}<C_{\varepsilon',n}\cdot\sqrt{\sum_{i=1}^{d}\left(\frac{\gamma}{\sqrt{d}}\right)^{2}}=C_{\varepsilon',n}\cdot\gamma,
\end{align*}
which -- together with equations \eqref{eq:DiagonalGroupXiNormalizedEstimate}
and \eqref{eqn:diag_est_xi} (for $\xi$ as well as $\xi'$) -- finally
yields
\begin{align*}
\sqrt{d}\left(\frac{h^{-T}\xi'}{\left|h^{-T}\xi'\right|}\right)_{i} & =\frac{\sqrt{d}a_{i}^{-1}\xi_{i}'}{\gamma'}=\frac{\sqrt{d}a_{i}^{-1}\xi_{i}}{\gamma}\cdot\frac{\gamma}{\gamma'}\cdot\frac{\xi_{i}'}{\xi_{i}}\\
 & <\left(1+\varepsilon'\right)\cdot C_{\varepsilon',n}\cdot\frac{\frac{1}{\sqrt{d}}\cdot\left(\frac{n+1}{n}\right)^{2}}{\frac{1}{\sqrt{d}}\cdot\left(\frac{n}{n+1}\right)^{2}}\\
 & =\left(1+\varepsilon'\right)^{2}\cdot\left(\frac{n+1}{n}\right)^{8}.
\end{align*}
A completely analogous computation also shows 
\[
\sqrt{d}\left(\frac{h^{-T}\xi'}{\left|h^{-T}\xi'\right|}\right)_{i}>\frac{1}{\left(1+\varepsilon'\right)^{2}\cdot\left(\frac{n+1}{n}\right)^{8}}.
\]
Hence (cf. also equation \eqref{eqn:est_diag_ca1}), we have $h^{-T}\xi'\in C\left(W,R\right)$
as soon as $\left(1+\varepsilon'\right)^{2}\cdot\left(\frac{n+1}{n}\right)^{8}<1+\varepsilon$
and $\left(\frac{n}{n+1}\right)^{4}\cdot R'>R$ hold. But it is easy
to see that this is true for suitable $n=n\left(\varepsilon\right)$,
$\varepsilon'=\varepsilon'\left(\varepsilon\right)$ and $R'=R'\left(R\right)$.

As $h\in K_{o}\left(W',V_{n},R'\right)$ and $\xi'\in V_{n}$ were
arbitrary, this shows that each $h\in K_{o}\left(W',V_{n},R'\right)$
maps $V_{n}$ into $C\left(W,R\right)$. Hence the weak $\mathcal{V}$-cone
approximation property holds.

\subsubsection{The dual action of $H$ is microlocally admissible}

Since we already established the weak cone approximation property,
Lemma \ref{lem:MicrolocalAdmissibilityKiStattKo} shows that it is
sufficient to verify the estimate required in Definition \ref{defn:micro_regular}\ref{enu:NormOfInverseEstimateOnKo}
on the smaller set $K_{i}(W,V_{n},R)$,  for the fixed $\xi_{0}$-neighborhood
$W:=U_{1}$ and $R=1$, but for \emph{all} sufficiently large $n\in\mathbb{N}$.
This will then yield $V_{n}$-microlocal admissibility for all sufficiently
large $n\in\mathbb{N}$.

So let $h={\rm diag}(a_{1},\ldots,a_{d})\in K_{i}(W,V_{n},R)$. Then
the entries of $h$ are necessarily positive, because of 
\[
\frac{1}{\sqrt{d}}\left(a_{1}^{-1},\dots,a_{d}^{-1}\right)^{T}=h^{-T}\xi_{0}\in h^{-T}V_{n}\subset C\left(W,R\right)\subset\left(0,\infty\right)^{d}.
\]
This also implies
\[
\frac{\left(a_{1}^{-1},\dots,a_{d}^{-1}\right)^{T}}{\left|\left(a_{1}^{-1},\dots,a_{d}^{-1}\right)^{T}\right|}=\frac{\frac{1}{\sqrt{d}}\left(a_{1}^{-1},\dots,a_{d}^{-1}\right)^{T}}{\left|\frac{1}{\sqrt{d}}\left(a_{1}^{-1},\dots,a_{d}^{-1}\right)^{T}\right|}\in W=U_{1},
\]
which yields the estimates 
\begin{align*}
 & \frac{1}{2}<\frac{\sqrt{d}\cdot a_{i}^{-1}}{\left|\left(a_{1}^{-1},\dots,a_{d}^{-1}\right)\right|}<2\\
\Leftrightarrow & \frac{\sqrt{d}}{2\cdot\left|\left(a_{1}^{-1},\dots,a_{d}^{-1}\right)\right|}<a_{i}<\frac{2\sqrt{d}}{\left|\left(a_{1}^{-1},\dots,a_{d}^{-1}\right)\right|}.
\end{align*}
In particular, we get 
\[
\frac{\max_{i}\, a_{i}}{\min_{i}\, a_{i}}\le4~,
\]
and thus 
\[
\left\Vert h^{-1}\right\Vert =\frac{1}{\min_{i}a_{i}}\le\frac{4}{\max_{i}a_{i}}=4\cdot\left\Vert h\right\Vert ^{-1}~.
\]
By Lemma \ref{lem:MicrolocalAdmissibilityKiStattKo}, we conclude
that the dual action of $H$ is globally mocrolocally admissible for
all sufficiently large $n\in\mathbb{N}$.

\subsection{The shearlet groups}

The shearlet transform in two dimensions was introduced in \cite{KuLa},
specifically for the purpose of characterizing wavefront sets. The
group-theoretical background of this transform was realized later
in \cite{DaKuStTe}.

The higher-dimensional generalizations were introduced in \cite{DaStTe10},
and further investigated (e.g.) in \cite{DaStTe12,CzaKi12}. In fact,
there is a whole family of shearlet groups in dimension $d\ge2$,
parameterized by a vector of real exponents $(c_{2},\ldots,c_{d})$,
usually taken between $0$ and $1$, see the more detailed description
below. It is the purpose of this subsection to establish both the
microlocal admissibility and the (strong) cone approximation property
for this class of groups, and thus to establish their suitability
for the characterization of regular directed points using a \emph{single}
wavelet. For $d=2$, it was shown in \cite{KuLa} (for the special
case $c_{2}=1/2$) that the full wavefront set of a tempered distribution
$u$ could be characterized using the wavelet transform with respect
to the dilation group $H$, together with the wavelet transform associated
to a second dilation group, namely 
\[
H'=\left(\begin{array}{cc}
0 & 1\\
1 & 0
\end{array}\right)H\left(\begin{array}{cc}
0 & 1\\
1 & 0
\end{array}\right)~.
\]
An alternative treatment and proof of this result can be found in
\cite{Grohs_2011}. The reasoning employed in both sources was to
decompose the frequency space into a horizontal and a vertical cone,
and to use $H$ for the characterization of horizontal directions,
and $H'$ for the vertical directions. Our results considerably extend
these findings: They entail that in the case $d=2$, the group $H$
can be used to characterize all directions except $\pm(0,1)^{T}$,
and for a whole range of parameters $c_{2}$. These observations could
possibly be deduced by adapting the arguments in \cite{KuLa,Grohs_2011},
but were not noted there. For $d=3$ , the results obtained by Guo
and Labate in \cite{GuLa_12} can be partly understood as wavefront
set characterizations for certain classes of tempered distributions
(with exponents $c_{2}=c_{3}=1/2$). The characterization for general
distributions seems to be missing so far, and for $d>3$ we are not
aware of a previous source containing even partial results.

The shearlet group $H$ in dimension $d\ge2$ is given by\cite{DaStTe10}:
\[
H=\left\{ \pm\left(\begin{array}{cccc}
a & b_{2} & \ldots & b_{d}\\
 & a^{c_{2}}\\
 &  & \ddots\\
 &  &  & a^{c_{d}}
\end{array}\right):a>0\;,b_{2},\ldots,b_{d}\in\mathbb{R}\right\} ~~.
\]
Here $(c_{2},\ldots,c_{d})\in\mathbb{R}^{d-1}$ is a vector of exponents
$c_{i}\in(0,1)$, which may be interpreted as anisotropy parameters.
In view of Lemma \ref{lem:anisotropy_necessary}, the exponents should
not all be identically one. The group has a single open orbit 
\[
\mathcal{O}=H^{T}\cdot\left(1,0,\dots,0\right)^{T}=\mathbb{R}^{\ast}\times\mathbb{R}^{d-1}
\]
and acts with trivial (hence compact) stabilizers.

Throughout this subsection, we will assume that the exponents $c_{2},\ldots,c_{d}$
lie strictly between $0$ and $1$; this condition will allow to establish
both the cone approximation property and microlocal admissibility.
In fact, the first-mentioned condition requires $c_{i}<1$, whereas
the second one (in addition) also needs $c_{i}>0$. Note that the
resulting characterization of the wavefront set is only valid for
directions $\xi\in S^{d-1}$ which lie in $\mathcal{O}=\mathbb{R}^{\ast}\times\mathbb{R}^{d-1}$,
i.e. satisfying $\xi_{1}\not=0$. In order to capture the remaining
directions, one may resort to the trick used in \cite{KuLa,GuLa_12}
for dimensions two and three, and employ in addition modified shearlet
transforms that are obtained by cyclically permuting the coordinates.
Thus $d$ shearlet transforms (associated to different coordinate
shifts) suffice for a full characterization.

\subsubsection{$H$ fulfils the (strong) cone approximation property}

We will establish the $V_{0}$-cone approximation property at $\xi_{0}=(1,0,\ldots,0)^{T}$
for $V_{0}:=\left(1,2\right)\times B_{1}^{\mathbb{R}^{d-1}}\left(0\right)$
under the assumption $c:=\max\left\{ c_{2},\dots,c_{d}\right\} <1$.
For this purpose, we introduce, for $R>0$ and $0<\epsilon<1$ the
set 
\[
W_{\epsilon}=\left\{ \xi=(\xi_{1},\xi_{2},\ldots,\xi_{d})^{T}\in S^{d-1}\with|\xi_{1}-1|<\epsilon\right\} \subset S^{d-1}\cap\left(\left(0,\infty\right)\times\mathbb{R}^{d-1}\right)~,
\]
with associated cone 
\[
C(W_{\epsilon},R)=\left\{ \xi=(\xi_{1},\xi_{2},\ldots,\xi_{d})^{T}\in S^{d-1}\with\left|\frac{\xi_{1}}{\left|\xi\right|}-1\right|<\epsilon~,\left|\xi\right|>R\right\} \subset\left(0,\infty\right)\times\mathbb{R}^{d-1}~.
\]
Observe that $\xi\in W_{\epsilon}$ for $\epsilon\in\left(0,1\right)$
implies $\xi_{1}>1-\epsilon>0$ and hence 
\[
1=\left|\xi\right|^{2}\geq\left|\xi_{1}\right|^{2}+\left|\xi_{i}\right|^{2}>\left(1-\epsilon\right)^{2}+\left|\xi_{i}\right|^{2},
\]
which yields $\left|\xi_{i}\right|<\sqrt{1-\left(1-\epsilon\right)^{2}}\xrightarrow[\epsilon\downarrow0]{}0$,
so that the family $\left(W_{\epsilon}\right)_{0<\epsilon<1}$ is
a neighborhood base of $\xi_{0}=\left(1,0,\dots,0\right)^{T}$.

We are going to employ the following criterion for containment in
$C(W_{\epsilon},R)$, which holds for $0<\epsilon<1$: Given $v=(v_{1},v_{2},\ldots,v_{d})^{T}\in\mathbb{R}^{d}$,
with $v_{1}>0$, then 
\begin{equation}
v\in C(W_{\epsilon},R)\:\Longleftrightarrow\:|v|>R\text{ and }\frac{|(v_{2},\ldots,v_{d})^{T}|}{v_{1}}<\frac{\sqrt{2\epsilon-\epsilon^{2}}}{1-\epsilon}~.\label{eqn:cont_We}
\end{equation}
To see this, we clearly only need to prove equivalence of $v\in C\left(W_{\varepsilon}\right)$
to the last condition given in equation \eqref{eqn:cont_We}. As both
of these conditions are invariant under rescaling with positive scalars,
we may assume $\left|v\right|=1$. This implies $\left|\left(v_{2},\dots,v_{d}\right)^{T}\right|=\sqrt{1-v_{1}^{2}}$,
so that the last condition in equation \eqref{eqn:cont_We} is equivalent
to
\begin{eqnarray*}
\frac{\sqrt{1-v_{1}^{2}}}{v_{1}}<\frac{\sqrt{2\epsilon-\epsilon^{2}}}{1-\epsilon} & \overset{\epsilon<1,\, v_{1}>0}{\Longleftrightarrow} & \left(1-\epsilon\right)\sqrt{1-v_{1}^{2}}<v_{1}\cdot\sqrt{2\epsilon-\epsilon^{2}}\\
 & \overset{\epsilon<1,\, v_{1}>0}{\Longleftrightarrow} & \left(1-\epsilon\right)^{2}\cdot\left(1-v_{1}^{2}\right)<v_{1}^{2}\cdot\left(2\epsilon-\epsilon^{2}\right)\\
 & \Longleftrightarrow & \left(1-\epsilon\right)^{2}<v_{1}^{2}\cdot\left(2\epsilon-\epsilon^{2}+\left(1-\epsilon\right)^{2}\right)\\
 & \Longleftrightarrow & \left(1-\epsilon\right)^{2}<v_{1}^{2}\\
 & \overset{\epsilon<1,\, v_{1}>0}{\Longleftrightarrow} & 1-\epsilon<v_{1},
\end{eqnarray*}
which is equivalent to $1-v_{1}<\epsilon$. But because of $v_{1}\leq\left|v\right|=1$,
this is equivalent to $\left|1-v_{1}\right|<\epsilon$, i.e. to $v\in W_{\epsilon}$.
Because of $\left|v\right|=1$, this is equivalent to $v\in C\left(W_{\epsilon}\right)$.

To establish the $V_{0}$-cone approximation property, it suffices
(thanks to Remark \ref{rem:incl_prop} and the fact that $\left(W_{\epsilon}\right)_{0<\epsilon<1}$
is a neighborhood base of $\xi_{0}$) to prove, for any given $\epsilon,R>0$,
the existence of $\epsilon',R'>0$ with $K_{o}(W_{\epsilon'},V_{0},R')\subset K_{i}(W_{\epsilon},V_{0},R)$.
To this end, we consider arbitrary $\epsilon'<1/2$ and $R'>4$, and
derive estimates for the entries of $h\in K_{o}(W_{\epsilon'},V_{0},R')$,
which will allow us to prove the desired inclusion under suitable
conditions on $\epsilon',R'$ (depending only on $\epsilon,R$).

So assume that $h\in K_{o}(W_{\epsilon'},V_{0},R')$ is given, where
\begin{equation}
h=\pm h(a,b)=\pm\left(\begin{array}{cccc}
a & b_{2} & \ldots & b_{d}\\
 & a^{c_{2}}\\
 &  & \ddots\\
 &  &  & a^{c_{d}}
\end{array}\right)\label{eqn:h(a,b)}
\end{equation}
with $a>0$ and $b=(b_{2},\ldots,b_{d})\in\mathbb{R}^{d-1}$. The
inverse transpose oh $h$ is computed as 
\begin{equation}
h^{-T}=\pm\left(\begin{array}{cccc}
a^{-1}\\
-a^{-1-c_{2}}b_{2} & a^{-c_{2}}\\
 &  & \ddots\\
-a^{-1-c_{d}}b_{d} &  &  & a^{-c_{d}}
\end{array}\right)~.\label{eqn:h_minus_transp}
\end{equation}
The assumption $h\in K_{o}(W_{\epsilon'},V_{0},R')$ yields the existence
of $\xi=(\xi_{1},\ldots,\xi_{d})\in V_{0}\subset\left(0,\infty\right)\times\mathbb{R}^{d-1}$
such that 
\[
h^{-T}\xi=\pm\left(\begin{array}{c}
a^{-1}\xi_{1}\\
\xi'
\end{array}\right)\in C\left(W_{\varepsilon'},R'\right)\subset\left(0,\infty\right)\times\mathbb{R}^{d-1},
\]
with $\xi'\in\mathbb{R}^{d-1}$ given by 
\[
\xi'=\left(\begin{array}{c}
-b_{2}a^{-1-c_{2}}\xi_{1}+a^{-c_{2}}\xi_{2}\\
\vdots\\
-b_{d}a^{-1-c_{d}}\xi_{1}+a^{-c_{d}}\xi_{d}
\end{array}\right)~.
\]
Because of $\pm a^{-1}\xi_{1}>0$ and $\xi_{1}>0$, as well as $a>0$,
we can rule out the negative sign for the ``$\pm$''. Furthermore,
the assumption $h^{-T}\xi\in C(W_{\epsilon'},R')$ entails 
\begin{equation}
\frac{a^{-1}\xi_{1}}{|h^{-T}\xi|}>1-\epsilon'~.\label{eqn:est_a}
\end{equation}
This, combined with $|h^{-T}\xi|>R'$, $0<\epsilon'<\frac{1}{2}$
and $\xi_{1}\in(1,2)$ as well as $R'>4$, yields 
\begin{equation}
a^{-1}>\frac{1-\epsilon'}{2}R'>\frac{R'}{4}>1~,\label{eqn:est_a_2}
\end{equation}
and hence $0<a<1$.

On the other hand, an application of equation (\ref{eqn:cont_We})
(noting that $\left(h^{-T}\xi\right)_{1}=a^{-1}\xi_{1}>0$ and $h^{-T}\xi\in C\left(W_{\epsilon'},R'\right)$),
together with $\xi_{1}\in\left(1,2\right)$ implies 
\begin{equation}
|\xi'|<a^{-1}\xi_{1}\frac{\sqrt{2\epsilon'-\left(\epsilon'\right)^{2}}}{1-\epsilon'}<a^{-1}\frac{2\sqrt{2\epsilon'-\left(\epsilon'\right)^{2}}}{1-\epsilon'}~.\label{eqn:est_xip}
\end{equation}

These inequalities can be combined to obtain a useful estimate involving
the shearing vector $(b_{2},\ldots,b_{d-1})^{T}$: We recall $c=\max\left\{ c_{2},\ldots,c_{d}\right\} <1$,
and get 
\begin{align}
\left|a\left(\begin{array}{c}
b_{2}a^{-1-c_{2}}\\
\vdots\\
b_{d}a^{-1-c_{d}}
\end{array}\right)\right| & =\left|\frac{a}{\xi_{1}}\left(\begin{array}{c}
b_{2}a^{-1-c_{2}}\xi_{1}\\
\vdots\\
b_{d}a^{-1-c_{d}}\xi_{1}
\end{array}\right)\right|\nonumber \\
 & \le\left|\frac{a}{\xi_{1}}\xi'\right|+\left|\frac{a}{\xi_{1}}\left(\begin{array}{c}
a^{-c_{2}}\xi_{2}\\
\vdots\\
a^{-c_{d}}\xi_{d}
\end{array}\right)\right|\nonumber \\
 & \le a|\xi'|+a^{1-c}\left|\left(\begin{array}{c}
\xi_{2}\\
\vdots\\
\xi_{d}
\end{array}\right)\right|\nonumber \\
 & \le\frac{2\sqrt{2\epsilon'-\left(\epsilon'\right)^{2}}}{1-\epsilon'}+a^{1-c}~.\label{eqn:est_shearvec}
\end{align}
Here the penultimate inequality used $\xi_{1}>1$ and $a\le1$ as
well as $c<1$, and the last inequality used $|(\xi_{2},\ldots,\xi_{d})^{T}|<1$
by definition of $V_{0}$, as well as equation (\ref{eqn:est_xip}).

Now let $\widetilde{\xi}=(\widetilde{\xi_{1}},\widetilde{\xi}_{2},\ldots,\widetilde{\xi}_{d})^{T}\in V_{0}=\left(1,2\right)\times B_{1}^{\mathbb{R}^{d-1}}\left(0\right)$
be arbitrary. Then 
\[
h^{-T}\widetilde{\xi}=\left(\begin{array}{c}
a^{-1}\widetilde{\xi_{1}}\\
\xi''
\end{array}\right)\in\left(0,\infty\right)\times\mathbb{R}^{d-1}~,
\]
with 
\[
\xi''=\left(\begin{array}{c}
-b_{2}a^{-1-c_{2}}\widetilde{\xi}_{1}+a^{-c_{2}}\widetilde{\xi}_{2}\\
\vdots\\
-b_{d}a^{-1-c_{d}}\widetilde{\xi}_{1}+a^{-c_{d}}\widetilde{\xi}_{d}
\end{array}\right)\in\mathbb{R}^{d-1}~.
\]

We first observe that choosing $R'>\max\left\{ 4,4R\right\} $ ensures,
via equation (\ref{eqn:est_a_2}), that 
\[
R<\frac{R'}{4}<a^{-1}<a^{-1}\widetilde{\xi}_{1}<\left|h^{-T}\widetilde{\xi}\right|.
\]
In order to apply criterion \eqref{eqn:cont_We}, we estimate
\begin{eqnarray*}
\frac{\left|\left(\left(h^{-T}\widetilde{\xi}\right)_{2},\dots,\left(h^{-T}\widetilde{\xi}\right)_{d}\right)\right|}{\left(h^{-T}\widetilde{\xi}\right)_{1}}=\frac{|\xi''|}{a^{-1}\widetilde{\xi}_{1}} & \le & \frac{a}{\widetilde{\xi}_{1}}\left|\left(\begin{array}{c}
-b_{2}a^{-1-c_{2}}\widetilde{\xi}_{1}\\
\vdots\\
-b_{d}a^{-1-c_{d}}\widetilde{\xi}_{1}
\end{array}\right)\right|+\frac{a}{\widetilde{\xi}_{1}}\left|\left(\begin{array}{c}
a^{-c_{2}}\widetilde{\xi}_{2}\\
\vdots\\
a^{-c_{d}}\widetilde{\xi}_{d}
\end{array}\right)\right|\\
 & \le & \frac{2\sqrt{2\epsilon'-\left(\epsilon'\right)^{2}}}{1-\epsilon'}+2\cdot a^{1-c}~,
\end{eqnarray*}
where the last inequality employed (\ref{eqn:est_shearvec}), together
with $1<\widetilde{\xi}_{1}<2$ and $|(\widetilde{\xi}_{2},\ldots,\widetilde{\xi_{d}})^{T}|<1$
by definition of $V_{0}$. Observe that equation \eqref{eqn:est_a_2},
together with $c-1<0$ implies 
\[
a^{1-c}=\left(a^{-1}\right)^{c-1}<\left(\frac{R'}{4}\right)^{c-1}=4^{1-c}\cdot\left(R'\right)^{c-1}.
\]
Thus, by invoking criterion (\ref{eqn:cont_We}), we see that choosing
$R'>\max\left\{ 4,4R\right\} $ sufficiently large and $\epsilon'>0$
sufficiently small to fulfil 
\[
\frac{2\sqrt{2\epsilon'-\left(\epsilon'\right)^{2}}}{1-\epsilon'}+2\cdot4^{1-c}\cdot\left(R'\right)^{c-1}<\frac{\sqrt{2\epsilon-\epsilon^{2}}}{1-\epsilon}~
\]
ensures $h^{-T}\widetilde{\xi}\in C(W_{\epsilon},R)$ for all $\widetilde{\xi}\in V_{0}$
and all $h\in K_{o}(W_{\epsilon'},V_{0},R')$. (Observe that $c<1$
allows such a choice of $R'$.) Thus, the global $V_{0}$-cone approximation
property is established (cf. Lemma \ref{lem:SingleOrbitConeApproximationProperty}).

\subsubsection{The dual action of $H$ is $V_{0}$-microlocally admissible}

Here, we impose the additional assumption $c':=\min\left\{ c_{2},\dots,c_{d}\right\} >0$.
By Lemma \ref{lem:MicrolocalAdmissibilityKiStattKo}, the fact that
we already established the $V_{0}$-cone approximation property allows
us to verify condition \ref{defn:micro_regular} \ref{enu:NormOfInverseEstimateOnKo}
on the smaller set $K_{i}$ instead of $K_{o}$. Thus, it suffices
to show that there exist $\alpha_{1}>0$ and $C>0$ such that 
\[
\left\Vert h^{-1}\right\Vert \le C\cdot\left\Vert h\right\Vert ^{-\alpha_{1}}
\]
holds for all $h=\pm h(a,b)\in K_{i}(W_{\epsilon},V_{0},R)$, for
some $\epsilon<1/2$ and $R>4$. Here we use the notation of the previous
subsection. Our choice of $\epsilon$ and $R$ then entails $a<1$
(cf. equation \eqref{eqn:est_a_2}) and rules out the negative sign
for the ``$\pm$''.

Let $v_{0}:=(3/2,0,\ldots,0)\in V_{0}$. Then $h^{-T}v_{0}\in C(W_{\epsilon},R)$
because of $h\in K_{i}(W_{\epsilon},V_{0},R)$. We recall 
\[
\left(h(a,b)\right)^{-1}=h\left(a^{-1},(-a^{-1-c_{2}}b_{2},\ldots,-a^{-1-c_{d}}b_{d})\right)
\]
and thus 
\[
h^{-T}v_{0}=\left(\begin{array}{c}
\frac{3}{2}a^{-1}\\
-\frac{3}{2}b_{2}a^{-1-c_{2}}\\
\vdots\\
-\frac{3}{2}b_{d}a^{-1-c_{d}}
\end{array}\right)~.
\]
Now equation (\ref{eqn:cont_We}) (together with $\left(h^{-T}v_{0}\right)_{1}=\frac{3}{2}a^{-1}>0$
and $h^{-T}v_{0}\in C\left(W_{\epsilon},R\right)$) yields the estimate
\begin{equation}
C_{1}:=\frac{\sqrt{2\epsilon-\epsilon^{2}}}{1-\epsilon}>\frac{\left|(b_{2}a^{-1-c_{2}},\ldots,b_{d}a^{-1-c_{d}})^{T}\right|}{a^{-1}}\ge a^{-c'}\left|(b_{2},\ldots,b_{d})^{T}\right|\label{eqn:est_shearvec2}
\end{equation}
where we used $c'=\min\left\{ c_{2},\ldots,c_{d}\right\} >0$ and
$a<1$. Now the fact that all matrix norms are equivalent allows us
to conclude
\begin{align*}
\left\Vert \left(h(a,b)\right)^{-1}\right\Vert  & \le C\cdot\max\left\{ a^{-1},\left|\left(b_{2}a^{-1-c_{2}},\ldots,b_{d}a^{-1-c_{d}}\right)^{T}\right|,\left|\left(a^{-c_{2}},\ldots,a^{-c_{d}}\right)^{T}\right|\right\} \\
 & \le C'\cdot a^{-1}~.
\end{align*}
On the other hand, using $0<c'\leq c_{i}\leq c<1$ and $0<a<1$, we
have by (\ref{eqn:est_shearvec2}) that
\begin{align*}
\left\Vert h(a,b)\right\Vert  & \le C'\cdot\max\left\{ a,\left|\left(b_{2},\ldots,b_{d}\right)^{T}\right|,\left|\left(a^{c_{2}},\ldots,a^{c_{d}}\right)^{T}\right|\right\} \\
 & \le C'\cdot\max\left\{ a,a^{c'}\right\} =C'\cdot a^{c'}~.
\end{align*}
This finally leads to 
\[
\left\Vert h^{-1}\right\Vert \le C''\cdot\left\Vert h\right\Vert ^{-1/c'},
\]
which is the desired inequality.

\section*{Concluding remarks}

Microlocal admissibility is a nontrivial condition on the dilation
group, as can be seen by the example provided by the shearlet dilation
group for $c'=\min\left\{ c_{2},\dots,c_{d}\right\} <0$.

It is easy to see that the various conditions studied in this paper
are preserved under conjugation: Whenever the pair $H,\mathcal{O}$
fulfils our technical assumptions \ref{assume:proper_dual}, and $H'=gHg^{-1}$
is given, then $H',g^{-T}\mathcal{O}$ also fulfil the assumptions.
Moreover, one easily verifies that microlocal admissibility and the
(weak) cone approximation property are also preserved under conjugation.
Together with a classification result from \cite{FuCuba}, this shows
that the list of examples studied in Section \ref{sect:examples}
exhausts all possible dilation groups with open orbits and compact
stabilizers that can arise in dimension two.

\section*{Acknowledgements}

This research was funded partly by the Excellence Initiative of the
German federal and state governments, and by DFG, under the contract
FU 402/5-1.

\bibliographystyle{abbrv}
\bibliography{wfset}

\end{document}